\documentclass[10pt,amstex,amssymb]{amsart}
\usepackage{amsmath,amsthm,amsfonts,amssymb,amscd}
\usepackage[latin1]{inputenc}
\usepackage[all]{xy}
\usepackage[dvips]{graphicx}
\usepackage{xcolor}
\usepackage{comment}
\def\vr{{\varphi}}
\newcommand\dint{\displaystyle\int}
\def\fa{{\mathcal{F}}}

\def\Diff{\mbox{Diff}}

\newtheorem{theorem}{Theorem}

\newtheorem{Theorem}{Theorem}[section]
\newtheorem{Corollary}{Corollary}[section]
\newtheorem{Lemma}{Lemma}[section]
\newtheorem{Question}{Question}[section]
\newtheorem{Claim}{Claim}[section]

\newtheorem{Definition}{Definition}[section]
\newtheorem{Remark}{Remark}[section]
\newtheorem{Example}{Example}[section]
\newtheorem{Conjecture}{Conjecture}[section]

\title[On singular Frobenius]{On singular Frobenius  for linear  differential equations of second and third order, part 1: ordinary differential equations}
\author{V. Le\'on}
\address{V. Le\'on. ILACVN - CICN, Universidade Federal of the Integração Latino-Americano, Parque tecnológico de Itaip\'u, Foz do  Iguaçu-PR, 85867-970 - Brazil}
\email{victor.leon@unila.edu.br}

\author{B. Sc\'ardua}
\address{B. Sc\'ardua. Instituto de Matem\'atica - Universidade Federal do the Rio de Janeiro,
CP. 68530-Rio de Janeiro-RJ, 21945-970 - Brazil}
\email{bruno.scardua@gmail.com}


\begin{document}

\begin{abstract}
We study  second order and third order linear differential equations with analytic coefficients
under the viewpoint of finding formal solutions and studying their convergence. We address some untouched aspects of Frobenius methods for second order as the convergence of formal solutions
and the existence of Liouvillian solutions. A characterization of
regular singularities is given in terms of the space of solutions. An  analytic classification of such linear homogeneous ODEs is obtained. This is done by associating to such an ODE  a Riccati differential equation and therefore a {\it global holonomy group}. This group is a computable group of Moebius maps.
These techniques apply  to classical equations as Bessel and Legendre equations.
In the second part of this work we study third order  equations. We prove a theorem similar to classical Frobenius  theorem, which describes all the possible cases and solutions to this type of ODE. Once armed with
this we pass to investigate the existence of solutions in the non-homogeneous case and
also the existence of a convergence theorem in the same line as done for second order above.
Our results are concrete and (computationally) constructive and are aimed to shed a new light in this important,
useful and attractive field of science.
\end{abstract}

\maketitle

\tableofcontents

\section{Introduction}

\par Differential equations are among the most powerful tools in mathematics and physics
(\cite{A,F,L,Fowles2,Dushman}). Roughly speaking, these are equations involving one or more functions and their derivatives.
Their study has many aspects, from quantitative theory, ie. the search of solutions, to qualitative theory.
There are two main groups of differential equations: {\it ordinary} differential equations (ODEs for short) and {\it partial} differential equations (PDEs for short). The first consists of  equations depending on a single variable (time for instance). By its turn PDEs involves partial derivatives, depending on several variables.

\par Since the first appearance of Newton's laws of motion (\cite{newton}), the study of ordinary differential equations has been associate
with fundamental problems in physics and science in general. This has been reinforced by the work of many
scientists (mathematicians, physicists, meteorologists, etc) through their contributions in problems as:
universal gravitation and planetary dynamics, dynamics of particles under the action of a force field
as the electromagnetic field, thermodynamics, meteorology and  weather forecast, study of climate
phenomena as typhoons and hurricanes, aerodynamics and hydrodynamics, atomic models, etc.
The list is as long and the possibilities of human scientific development.

\par Thanks to the nature of Newton's laws and other laws as Maxwell's equations or Faraday's and Kepler's laws (\cite{Cordani,Kepler}), most of the pioneering  work is in  ODEs. Furthermore, these classical equations are   of first or second order (the higher order of the derivatives is not beyond two).
Of special interest are the laws of the oscillatory movement (pendulum equation and Hill lunar movement equation\cite{Hill}) and Hooke's law (spring extension or compression).
Let us not forget that classical fundamental solutions to problems as heat conduction (heat equation),
vibration (wave equation) and others (Laplace equation). Though these problems  are modeled  by partial differential equations, they may be solved with the aid of ordinary differential equations. This is for instance the idea of the method of separation of variables and eventual use of Fourier series.

\par All these classical equations above are, or have nice approximations by, linear equations.
Among the linear equations the homogeneous case is a first step and quite meaningful.
Thus, to be able to solve classical ordinary linear homogeneous differential equations is an important subject of active study in mathematics. The arrival of features like scientific computing gives new
breath to the problem of solution of a given ODE by looking for solutions via power series.
This of course in the real analytic framework, which is quite common in the Nature (\cite{Fowles}).
In this sake, a classical and powerful method is due to Frobenius.
The {\it method of Frobenius} can be summarized as follows:
Given a linear homogeneous second order differential equation
$a(x) y^{\prime \prime} + b(x)y^\prime  + c(x)y=0$ for some real analytic functions $a(x), b(x), c(x)$ at some point $x_0\in \mathbb R$, we look for solutions which are of the form
$y(x)= \sum\limits_{n=0}^\infty d_n (x-x_0)^{n+r}$ where $d_n$ and $r$ are constants.
We shall not detail this method now, but we must say that this is based on {\it Euler's equation}
 $a x^2 y^{\prime \prime} + b x y ^\prime + cy=0$ and the idea of looking for  solutions of the form $y=x^r$.
 Then $r$ must be a  root of the so called {\it indicial equation}. The main point is that Frobenius method works pretty well
in a suitable class of second order ODEs, so called {\it regular singular} ODEs around $x_0$.

 \vglue.2in

Third-order differential
  equations models are used in modeling an important number of high energy physical problems.
  One of the  first that comes to mind is
the problem of   oscillations in a  nuclear reactor (\cite{Troy}).
The  deflection of a curved beam
  having a constant or a varying cross-section is another example. Other examples are  three layer beams, electromagnetic waves
  or gravity-driven flows. We also have
  Barenblatt's equation for diffusion in a porous fissured medium.
In neurobiology, for example modeling current flow in neurons with microstructure,
$V_t = V_{xx} +gV_{txx} -V$ where $g$ is a constant (see \cite{Omori}). More generally, third order ODEs appear in astrodynamics:
the Clohessy-Wiltsjire equations (relative motion about a circular orbit),
the Tschauner-Hempel equations (relative motion about an ellipse),
the two-body equation after the substitution y=1/r. There are other situations. For instance, third order differential equation is the one
for the temperature appearing in the heat transport theory of materials
contradicting the ``fading memory paradigm".
Finally, more Physical examples are: the
Abraham-Lorentz force (electron self-force) (\cite{Dushman,FR}) and the Jerk for parabolic curves in the roads.

\par In this paper we study both second order and third order differential equations with analytic coefficients
under the viewpoint of finding solutions and studying their convergence. In very few words,
we study forgotten aspects of Frobenius methods for second order as convergence of formal solutions
and the existence of Liouvillian solutions. We also discuss the characterization of the so called {\it
regular singularities} in terms of the space of solutions. An  analytic classification is obtained via associating to such an ODE  a Riccati differential equation and therefore a {\it global holonomy group}. This group is a computable group of Moebius maps.
Next we  apply these techniques and results to classical equations as Bessel and Legendre equations.

\par In the second part of this work we study third order linear differential equations. After presenting a model for the Euler equation and its corresponding indicial equation in this case, we  introduce the notion
of regular singular point for this class of equations. Then we prove a theorem similar to classical Frobenius
theorem, which describes all the possible cases and solutions to this type of ODE. Once armed with
this we pass to investigate the existence of solutions in the non-homogeneous case and
also the existence of a convergence theorem in the same line as done for second order above.

Our results are concrete and (computationally) constructive and are aimed to shed a new light in this important,
useful and attractive field of science.
\vglue.1in
Next we give a more detailed description of our results.

\subsection{The classical method of Frobenius  for second order}

The classical method of Frobenius is a very useful tool in finding solutions of
a homogeneous second order linear ordinary differential equations with analytic coefficients.
These are equations that write in the form
$a(x) y^{\prime \prime} + b(x)y^\prime  + c(x)y=0$ for some real analytic functions $a(x), b(x), c(x)$ at some point
$x_0\in \mathbb R$. It well known that if $x_0$ is an ordinary point, i.e., $a(x_0)\ne 0$ then there
are two linearly independent
solutions $y_1(x), y_2(x)$ of the ODE, admitting  power series expansions converging in some common neighborhood of
$x_0$. This is a consequence of the classical theory of ODE and also shows that the solution space of this ODE
has dimension two, i.e., any solution is of the form $c_1 y_1(x) + c_2 y_2(x)$ for some constants $c_1, c_2 \in \mathbb R$.
Second order linear homogeneous differential equations appear in many concrete
problems in natural sciences, as physics,
chemistry, meteorology and even biology. Thus solving such equations is an important task.
The existence of solutions for the case of an ordinary point is not enough for most of the applications.
Indeed, most of the relevant equations are connected to the singular (non-ordinary) case. We can mention
Bessel equation $x^2 y^{\prime \prime}  + x y ^\prime + (x^2 - \nu^2) y=0$, whose range of applications goes from
heat conduction, to the model of the hydrogen atom (\cite{BS,Gr}). This equation has  the origin $x=0$ as a singular point.
Another remarkable equation is the {\it Laguerre equation} $x y ^{\prime \prime} + (\nu +1 -x) y^\prime + \lambda y=0$ where
$\lambda, \nu \in \mathbb R$ are parameters. This equation is quite relevant  in quantum mechanics, since
it appears in the modern quantum mechanical description of the hydrogen atom.
All these are examples of equations with a {\it regular singular point}. According to Frobenius
a singular point $x=x_0$ of the ODE $a(x) y^{\prime \prime} + b(x)y^\prime  + c(x)y=0$
is {\it regular} if $\displaystyle\lim_{x \to x_0} (x-x_0)\frac{b(x)}{a(x)}$ and $\displaystyle\lim_{x\to x_0} (x-x_0)^2 \frac{c(x)}{a(x)}$ are finite.
We shall refer to this as follows: the ODE  $y^{\prime \prime} + \alpha(x) y ^\prime + \beta(x)=0$
has a regular singular point
at $x=x_0$ if $\displaystyle\lim_{x\to x_0} (x-x_0) \alpha(x)$ and $\displaystyle\lim_{x \to x_0}(x-x_0)^2\beta(x)$
admit  extensions which are analytic at  $x=x_0$.  In this case we have the following classical theorem of Frobenius:

\begin{Theorem}[Frobenius theorem, \cite{Boyce}, \cite{Frobenius}, \cite{C}]
Assume that the ODE $(x-x_0)^2y^{\prime \prime}+(x-x_0)b(x)y^\prime+c(x)y=0$
has a regular  singularity at $x=x_0$, where the functions $b(x), c(x)$ are
analytic with convergent power series in $|x-x_0|<R$. Then there is at least one solution
 of the form
$y(x)=|x-x_0|^r \sum\limits_{n=0}^\infty d_n (x-x_0)^n$ where  $r$ is a root of the indicial equation, $d_0=1$
where the series converges for $|x-x_0|<R$.
\end{Theorem}

The method of Frobenius  (for this case of second order ODE) consists in associating to the original
ODE an {\it Euler equation}, i.e., an equation of the form $ A(x-x_0)^2y^{\prime \prime} + B (x-x_0)y ^\prime +
Cy=0$ and looking for solutions (to this equation) of the form $y_0(x)=(x-x_0)^r$. This gives an algebraic equation
of degree two $A r (r-1) + B r + C=0$, so called {\it indicial equation}, whose zeroes $r$ give solutions $y_0(x)=(x-x_0)^r$ of the Euler equation.
The Euler equation associate to the original ODE with a regular singular point at $x=x_0$ is given by
$(x-x_0)^2y^{\prime \prime} + p_0 (x-x_0)y ^\prime +
q_0y=0$ where $p_0=\displaystyle\lim_{x \to x_0} (x-x_0)\frac{b(x)}{a(x)}$ and $q_0=\displaystyle\lim_{x\to x_0} (x-x_0)^2 \frac{c(x)}{a(x)}$.
Then Frobenius method consists in looking for solutions of the original ODE as of the form $y_1(x)=|x-x_0|^r \sum\limits_{n=0}^\infty d_n (x-x_0)^n$ where $r$ is the greater real part zero of the indicial equation given by the Euler equation as above. The equation whether there is a second
linearly independent solution is related to the roots of the indicial equation. Indeed, there is some zoology
and in general the second solution is of the form
 $y_2(x)=|x-x_0|^{\tilde r} \sum\limits_{n=0}^\infty\tilde{d}_n(x-x_0)^n$ in case there is a second root $\tilde r$ of the
 indicial equation and this root is such that $r-\tilde r\not\in \mathbb Z$. If $\tilde r =r$ then there is a solution of the form
 $y_2(x)=y_1(x)\log |x-x_0| + |x-x_0|^{r+1}\sum\limits_{n=0}^\infty \hat{d}_n (x-x_0)^n$. Finally,
  if $0 \ne   r - \tilde r\in \mathbb N$ then we have a second solution of the form
  $y_2(x)= ky_1(x)\log |x-x_0| + |x-x_0|^{\tilde r}\sum\limits_{n=0}^\infty \check{d}_n (x-x_0)^n$.
  This brief description of the method of Frobenius already suggests that there may exist a higher order
  version of this result. For instance for third order linear homogeneous ODEs.

 \subsection{Second order  equations}
Section~2 is dedicated to the study of second order linear differential equations.
We start with the following.
\subsubsection{Convergence of formal solutions for second order linear homogenous ODEs}
We shall now discuss the problem of convergence of formal solutions for linear homogeneous
ODEs of order two. We recall that there are examples of ODEs admitting a formal solution that is nowhere convergent
( cf. Example~\ref{2exeminf}).

Our next result may be seen as a version of a theorem due to Malgrange and also to
Mattei-Moussu for holomorphic integrable systems of differential forms. By a {\it formal solution centered at $x_0\in \mathbb R$}
 of an  ODE we shall mean a formal power series $\hat y(x)=\sum\limits_{n=0}^\infty a_n (x-x_0)^n$
 with complex coefficients $a_n \in \mathbb C$.
We prove:
\begin{theorem}[formal solutions order two]
\label{Theorem:A}
Consider a second order ordinary differential equation given by
\begin{equation}\label{Edo1}
a(x)y^{\prime\prime}+b(x)y^\prime+c(x)y=0
\end{equation}
where $a,b,c$ are analytic functions at $x_0 \in \mathbb R$.
Suppose also that there exist two linearly independent formal
solutions  $\hat y_1(x)$ and $\hat y_2(x)$ centered at $x_0$    of  equation (\ref{Edo1}).
Then $x_0$ is an ordinary point or a regular singular point of
 (\ref{Edo1}).
Moreover,
$\hat y_1(x)$ and $\hat y_2(x)$ are convergent.
\end{theorem}

\begin{theorem}[1 formal solution second order regular singularity]
\label{Theorem:B}
Consider a second order ordinary differential equation given by
\begin{equation}\label{Edo2}
a(x)y^{\prime\prime}+b(x)y^\prime+c(x)y=0
\end{equation}
where $a,b,c$ are analytic functions at $x_0 \in \mathbb R$.
Suppose (\ref{Edo2}) has at $x_0$  an ordinary point or a regular singular point.
Given a formal solution $\hat y(x)$ of  (\ref{Edo2}) then
this solution is convergent.
\end{theorem}

\subsubsection{Characterization of regular singular points in order two}

We shall say that a function
$u(x)$ for $x$ in a disc $|x|<R$ centered at the origin $0\in \mathbb R, \mathbb C$
is  {\it an analytic combination of log and power} ({\it anclop} for short) if it can be written as $u(x)=
\alpha(x) + \log(x) \beta(x) + \gamma(x) x^r$ for some analytic functions $\alpha(x), \beta(x), \gamma(x)$ defined in the disc $|x|<R$ and $r \in \mathbb R$ or $r \in \mathbb C$. In the real case we assume that
$x>0$ in case we have $\beta \not\equiv 0$ or $\gamma\not\equiv 0$ and a power $x^r$ with $r\in \mathbb R \setminus \mathbb Q$.

\begin{Definition}
{\rm A one-variable complex function $u(z)$ considered in a domain $U\subset \mathbb C$ will be called {\it
analytic up to log type singularities } ({\it autlos} for short) if:
 \begin{enumerate}
 \item $u(z)$ is holomorphic in $U\setminus \sigma$ where $\sigma\subset U$ is a discrete set of points, called {\it singularities}.
 \item Given a singularity $p \in \sigma$ either $p$ is a removable singularity of $u(z)$ or
 there is a germ of real analytic curve $\gamma\colon[0,\epsilon) \to U$ such that $\gamma(0)=p$ and
 $u(z)$ is holomorphic in $D\setminus \gamma[0,\epsilon)$ for some disc $D\subset U$ centered at $p$.
 \end{enumerate}

 A one variable real function $u(x)$ defined in an interval $J \subset \mathbb R$ will be called
{\it analytic up to log type singularities } ({\it autlos} for short) if, after complexification,
the corresponding function $u_\mathbb C(z)$, which is defined in some neighborhood $J \times \{0\}\subset U\subset \mathbb C$ is analytic up to log type singularities, as defined above.}
 \end{Definition}

\begin{theorem}[characterization of regular points]
\label{Theorem:characterizationregularpointordertwo}
Consider a linear homogeneous ordinary differential equation of second order given by
$$a(x)y^{\prime\prime}+b(x)y^\prime+c(x)y=0$$
where $a,b,c$ are analytic functions at $x_0 \in \mathbb R$.
Then the following conditions are equivalent:
\begin{enumerate}
\item The equation admits two
linearly independent  solutions $y_1(x), y_2(x)$ which are anclop (analytic combinations of log and power). Then $x_0$ is an ordinary point or a regular singular point
for the ODE.
\item The equation admits two solutions $y_1(x), y_2(x)$ which are autlos (analytic up to logarithmic singularities).
\item The equation has an ordinary point or a regular singular point at $x_0$.
\end{enumerate}
\end{theorem}

\subsubsection{Riccati model and holonomy of a second order equation}

We start with a polynomial second order linear equation of the form
$a(z)u^{\prime \prime}+b(z)u^\prime+c(z)u=0$ in the complex plane.
By introducing the change of coordinates $t = u ^\prime / u$ we obtain a
 first order  Riccati equation which writes as
$$\frac{dt}{dz}=-\frac{a(z)t^2+b(z)t+c(z)}{a(z)}.$$

\begin{Definition}
{\rm The Riccati differential equation above is called {\it Riccati model} of the ODE
\, \, $a(z)u^{\prime\prime}+b(z)u^\prime+c(z)u=0$. }
 \end{Definition}
By its turn, since the work of Paul Painlev\'e (\cite{Pain}), a polynomial Riccati equation is studied from the point of view of its transversality with respect to the vertical fibers $z=constant$, even at the points at the infinity. With the advent of the theory of foliations, due to Ehresmann, the notion of holonomy was introduced as well as the notion of global holonomy of a foliation transverse to the fibers of a fibration. This is the case of a polynomial Riccati foliation once placed in the ruled surface $\mathbb P^1(\mathbb C) \times \mathbb P^1(\mathbb C)$, where $\mathbb P^1(\mathbb C)=\mathbb C \cup \{\infty\}$ is the Riemann sphere.
This allows us to introduce the notion of {\it global holonomy} of a second order linear  equation as above.
This permits the study of the equation from this group theoretical point of view, since the global holonomy will be a group of Moebius maps of the form $t \mapsto \frac{ \alpha t + \beta}{\gamma _t + \delta}$.
We do calculate this group in some special cases and reach some interesting consequences for the
original ODE.

\begin{theorem}
\label{Theorem:Riccatitrivialholonomy}
Consider a second order polynomial ODE given by
$$u^{\prime\prime}+b(z)u^\prime+c(z)u=0$$
where  $b,c$ are  complex polynomials of a variable $z$. Then the equation above admits
a general solution of the form
 $$
u_{\ell,k}(z)=k \exp\big(\int_0^z \frac{\ell D(\xi) - B(\xi)}{A(\xi) - \ell C(\xi)}d\xi\big), \, k , \ell \in \mathbb C.
$$
\end{theorem}

\subsubsection{Liouvillian solutions}

One important class of solutions for ODEs is the class of {\it Liouvillian solutions}, a notion
introduced by Liouville, developed by Rosentich and Ross among other authors. The question
of whether a polynomial first order ODE admits a Liouvillian solution or first integral
has been addressed by M. Singer in \cite{Singer} and  others.

Recall the notion of Liouvillian function in $n$ complex variables as introduced in \cite{Singer}.
Such a  function  has (holomorphic) analytic branches  in some Zariski dense open subset of $\mathbb C^n$.
In particular we can ask whether an ODE admits such a solution.
\begin{Question}
What are the polynomial  ODEs of the form $a(z) u^{\prime \prime} + b(z) u ^\prime + c(z)u=0$ admitting
a Liouvillian solution $u(z)$?
\end{Question}

\begin{theorem}[characterization liouville]
\label{Theorem:characterizationliouville}
Consider a polynomial complex linear homogeneous ordinary differential equation of second order given by
\begin{equation}
L[u](z)=a(z)u^{\prime\prime}+b(z)u^\prime+c(z)u=0
\end{equation}
where $a,b,c$ complex polynomials.
Then we have the  following:

\begin{enumerate}
\item If $L[u]=0$  admits a  solution satisfying a Liouvillian relation
then it has a Liouvillian first integral (cf. Corollary pages 674,675 \cite{Singer}).
\item If $L[u]=0$  admits a Liouvillian solution then it has a Liouvillian first integral (cf. Corollary page 674,675 \cite{Singer}).

\item If $L[y]=0$ admits a Liouvillian first integral then its solutions are Liouvillian
and given by one of the forms below:

\begin{itemize}
\item[{\rm (a)}] $
 u(z)=\exp\big(-\int^z\gamma(\eta)d\eta\big)\bigg[ k \int^z \exp\big(\int^{\eta} \frac{2\gamma(\xi)-b(\xi)}{a(\xi)}d\xi\big)d\eta+\ell\bigg]$
for constants $k,\;\ell \in \mathbb C$ and $\gamma(z)$ is rational solution for the Riccati equation.

\item[{\rm(b)}] $u(z) = k_1 + k_2 \int^z \exp\big(-\int^\eta \frac{b(\xi)}{a(\xi)}d\xi\big)d\eta$, for constants $k_1,k_2 \in \mathbb C$.

\end{itemize}
\end{enumerate}

\end{theorem}

\subsection{Third order  equations}

In \S~3 we study third order linear ordinary differential equations  in the
homogeneous and non-homogeneous cases.
For this sake we first extend the notion of regular singular point above to ODEs of order $n \geq 2$.

 \subsubsection{Frobenius method for third order ODEs}

Consider a linear ordinary differential equation with variable coefficients of the form
\begin{equation}\label{eqs1}
a_0(x)y^{(n)}+a_1(x)y^{(n-1)}+\ldots+a_{n-1}(x)y^\prime+a_n(x)y=0.
\end{equation}
We shall assume that the coefficients $a_0,a_1,\ldots,a_{n-1},a_n$ are analytic at
some point $x_0$, and we shall study the case where $a_0(x_0)=0$. A point $x_0$ such
that $a_0(x_0)=0$, is called \textit{singular point} of  equation (\ref{eqs1}), otherwise it is an {\it ordinary point}.

\begin{Definition}{\rm We shall say that $x_0$ is a \textit{regular singular point} of (\ref{eqs1}), if the equation can be written  as follows
\begin{equation}\label{eqs2}
(x-x_0)^ny^{(n)}+b_1(x)(x-x_0)^{n-1}y^{(n-1)}+\ldots+b_{n-1}(x)(x-x_0)y^\prime+b_n(x)y=0
\end{equation}
for $x$ close enough to $x_0$, where the functions $b_1,\ldots,b_{n-1},b_n$ are analytic at $x_0$.}
\end{Definition}

\begin{Remark}$\;${\rm
\begin{enumerate}
\item  If the functions $b_1,\ldots,b_{n-1},b_n$ can be written  as
$$b_k(x)=(x-x_0)^k\beta_k(x)\;\;\mbox{ for every }\;k=1,\ldots,n,$$ where $\beta_1,\ldots,\beta_{n-1},\beta_n$ are analytic functions in $x_0$, we see that (\ref{eqs2}) is transformed into equation
\begin{equation}\label{eqs3}
y^{(n)}+\beta_1(x)y^{(n-1)}+\ldots+\beta_{n-1}(x)y^{\prime}+\beta_n(x)y=0.
\end{equation}
In this case, by classical theorems of ODEs, there are three linearly independent analytic solutions
$y_j(x)=\sum\limits_{n=0}^\infty a_n ^j x^n$, converging for $|x|<R$. Moreover, any solution is a linear combination of these solutions.

\item An  equation of the form

$$c_0(x)(x-x_0)^ny^{(n)}+c_1(x)(x-x_0)^{n-1}y^{(n-1)}+\ldots+c_{n-1}(x)y^\prime+c_n(x)y=0$$
has a regular singular point at $x_0$ if $c_0,c_1,\ldots,c_{n-1},c_n$ are analytic at $x_0$, and $c_0(x_0)\neq0$.
\end{enumerate}}
\end{Remark}

We first obtain an existence theorem like Frobenius theorem above. This reads as:

\begin{theorem}[Existence of a first solution]\label{thmfrobb3I}{\rm Consider the equation
\begin{equation}\label{eq1}
L(y):=x^3y^{\prime\prime\prime}+x^2a(x)y^{\prime\prime}+xb(x)y^\prime+c(x)y=0,
\end{equation}
 with $a(x),b(x)$ and $c(x)$ real analytic functions defined for $|x|<R$, $R>0$. Let $r_1,\;r_2$ and $r_3$  be the roots of the indicial polynomial $$q(r)=r(r-1)(r-2)+r(r-1)a(0)+rb(0)+c(0),$$ ordered in such a way that
  $\mbox{Re}(r_1)\geq \mbox{Re}(r_2)\geq\mbox{Re}(r_3)$. Then for $0<|x|<R$ there is a solution $\varphi_1$ of  equation (\ref{eq1}) given by
$$\varphi_1(x)=|x|^{r_1}\sum^{\infty}_{n=0}d_nx^n,\;\;\;d_0=1,$$where the series converges for $|x|<R$.
Moreover, if $r_1-r_2\notin\mathbb{Z}^+_0$, $r_1-r_3\notin\mathbb{Z}^+_0$ and $r_2-r_3\notin\mathbb{Z}^+_0$ then there exist other two linearly independent solutions $\varphi_2$ and $\varphi_3$ defined in $0<|x|<R$, given by:
$$\varphi_2(x)=|x|^{r_2}\sum^{\infty}_{n=0}\tilde{d_n}x^n\;\;\;\;\tilde{d_0}=1$$
and
$$\varphi_3(x)=|x|^{r_3}\sum^{\infty}_{n=0}\hat{d_n}x^n\;\;\;\;\hat{d_0}=1$$where the series converge for $|x|<R$. The coefficients $d_n$, $\tilde{d_n}$, $\hat{d_n}$ can be obtained replacing the solutions in equation (\ref{eq1}).
}

\end{theorem}
We shall consider a first simplest case of an equation, not of type (\ref{eqs3}),
that has a regular singular point. This equation is an Euler equation,
that is of the form (\ref{eqs2}) with $b_1,\ldots,b_n$ constants.

In the case of  a second order differential equation with a  regular singular point
\begin{equation}\label{eqsl4}
x^2y^{\prime\prime}+xb(x)y^\prime+c(x)y=0,\;x>0,
\end{equation}
where $b,c$ are analytic at $0$, the classical method of Frobenius  determines the form of the solutions of (\ref{eqsl4}) which are given by
$$\psi(x)=x^{r}\rho(x)+x^{s}\eta(x)\log x,$$where $r,s$ are constants and $\rho,\eta$ are analytic at $0$.

In this work we shall study the general third order linear equation with a  regular singular point, and describe the method of obtaining  solutions in a neighborhood of the singular point.
Motivated by the classical method of Frobenius and using some
techniques introduced in (\cite{C}, Chapter IV Section 4.5) we prove the following complete result:

Theorem~\ref{thmfrobb3I} is valid in the case of complex ordinary differential equations.

\begin{theorem}[non exceptional case]\label{thmfrobb3complexI}{\rm Consider the equation
\begin{equation}\label{eqc1}
L(u):=z^3u^{\prime\prime\prime}+z^2a(z)u^{\prime\prime}+zb(z)u^\prime+c(z)u=0,
\end{equation}
 with $a(z),b(z)$ and $c(z)$ analytic for $|z|<R$, $R>0$. Let $r_1,\;r_2$ and $r_3$
 be the roots of the indicial
 polynomial $$q(r)=r(r-1)(r-2)+r(r-1)a(0)+rb(0)+c(0).$$ Let us assume that $\mbox{Re}(r_1)\geq \mbox{Re}(r_2)\geq\mbox{Re}(r_3)$. Then for $0<|z|<R$ there
 is a solution $\varphi_1$ of  equation (\ref{eqc1}) given by
$$\varphi_1(z)=z^{r_1}\sum^{\infty}_{n=0}d_nz^n,\;\;\;d_0=1,$$where the series converges for $|z|<R$. Moreover, if $r_1-r_2\notin\mathbb{Z}^+_0$, $r_1-r_3\notin\mathbb{Z}^+_0$ and $r_2-r_3\notin\mathbb{Z}^+_0$ then there exist other two linearly independent solutions $\varphi_2$ and $\varphi_3$ defined in $0<|z|<R$, given by:
$$\varphi_2(z)=z^{r_2}\sum^{\infty}_{n=0}\tilde{d_n}z^n\;\;\;\;\tilde{d_0}=1$$
and
$$\varphi_3(z)=z^{r_3}\sum^{\infty}_{n=0}\hat{d_n}z^n\;\;\;\;\hat{d_0}=1$$where the series converge for $|z|<R.$ The coefficients $d_n$, $\tilde{d_n}$, $\hat{d_n}$ can be obtained replacing the solutions in equation (\ref{eqc1}).
}

\end{theorem}

\subsubsection{Special cases:}

We divide the special cases into four groups, according to the
roots $r_1,r_2,r_3$ (always ordered such that $Re(r_1)\geq Re(r_2)\geq Re(r_3)$) of the indicial polynomial, satisfy:
\begin{itemize}
\item[(i)] $r_1=r_2=r_3$
\item[(ii)] $r_1=r_2$ and $r_1-r_3\in\mathbb{Z}^+$
\item[(iii)] $r_2=r_3$ and $r_1-r_2\in\mathbb{Z}^+$
\item[(iv)] $r_1-r_2\in\mathbb{Z}^+$ and $r_2-r_3\in\mathbb{Z}^+$
\end{itemize}
We want to find solutions defined for $|x|<R$.

\begin{theorem}[exceptional cases]\label{thmfrobb3II}{\rm Consider the equation

$$L(y):=x^3y^{\prime\prime\prime}+x^2a(x)y^{\prime\prime}+xb(x)y^\prime+c(x)y=0,$$

 with $a(x),b(x)$ and $c(x)$ analytic for $|x|<R$, $R>0$. Let $r_1,\;r_2$ and $r_3$ ($\mbox{Re}(r_1)\geq \mbox{Re}(r_2)\geq\mbox{Re}(r_3)$) be roots of the indicial polynomial $$q(r)=r(r-1)(r-2)+r(r-1)a(0)+rb(0)+c(0).$$ \begin{itemize}
\item[(i)] If $r_1=r_2=r_3$ there exist three linearly independent solutions $\varphi_1,\varphi_2,\varphi_3$ defined in $0<|x|<R$, which has the following form:
$$\varphi_1(x)=|x|^{r_1}\sigma_1(x),\;\;\varphi_2(x)=|x|^{r_1+1}\sigma_2(x)+(\log |x|)\varphi_1(x)$$
and
$$\varphi_3(x)=|x|^{r_1+1}\sigma_3(x)+2(\log |x|)\varphi_2(x)-(\log|x|)^2\varphi_1(x), $$
where $\sigma_1,\sigma_2,\sigma_3$ are analytic in $|x|<R$ and $\sigma_1(0)\neq0$.

\item[(ii)] If $r_1=r_2$ and $r_1-r_3\in\mathbb{Z}^+$ there exist three linearly independent solutions   $\varphi_1,\varphi_2,\varphi_3$ defined in $0<|x|<R$, which has the form:
$$\varphi_1(x)=|x|^{r_1}\sigma_1(x),\;\;\varphi_2(x)=|x|^{r_1+1}\sigma_2(x)+(\log |x|)\varphi_1(x)$$
and
$$\varphi_3(x)=|x|^{r_3+2}\sigma_3(x)+c\;(\log |x|)\varphi_1(x), $$
where $c$ constant, $\sigma_1,\sigma_2,\sigma_3$ are analytic in $|x|<R$, $\sigma_1(0)\neq0$ and $\sigma_3(0)\neq0$ .

\item[(iii)] If $r_2=r_3$ and $r_1-r_2\in\mathbb{Z}^+$ there exist three linearly independent solutions $\varphi_1,\varphi_2,\varphi_3$ defined in $0<|x|<R$, which has the form:
$$\varphi_2(x)=|x|^{r_2}\sigma_2(x),\;\;\varphi_3(x)=|x|^{r_2+1}\sigma_3(x)+(\log |x|)\varphi_2(x)$$
and
$$\varphi_1(x)=|x|^{r_2+2}\sigma_1(x)+c\;(\log |x|)\varphi_2(x), $$
where $c$ constant, $\sigma_1,\sigma_2,\sigma_3$ are analytic in $|x|<R$, $\sigma_1(0)\neq0$ and $\sigma_2(0)\neq0$.

\item[(iv)] If $r_1-r_2\in\mathbb{Z}^+$ and $r_2-r_3\in\mathbb{Z}^+$ there exist three linearly independent solutions $\varphi_1,\varphi_2,\varphi_3$ defined in $0<|x|<R$, which has the form:
$$\varphi_1(x)=|x|^{r_1}\sigma_1(x),\;\;\varphi_2(x)=|x|^{r_2}\sigma_2(x)+c\;(\log |x|)\varphi_1(x)$$
and
$$\varphi_3(x)=|x|^{r_3}\sigma_3(x)+\tilde{c}\;(\log |x|)\varphi_2(x), $$
where $c$ and $\tilde{c}$ are constants, $\sigma_1,\sigma_2,\sigma_3$ are analytic in $|x|<R$, $\sigma_1(0)\neq0$, $\sigma_2(0)\neq0$ and $\sigma_3(0)\neq0$.
\end{itemize}}
\end{theorem}
Theorem~\ref{thmfrobb3II} is valid in the case of complex ordinary differential equations.

\begin{theorem}[complex analytic ODEs]\label{thmfrobb3complexII}
 Consider the equation

$$L(u):=z^3u^{\prime\prime\prime}+z^2a(z)u^{\prime\prime}+zb(z)u^\prime+c(z)u=0,$$

 with $a(z),b(z)$ and $c(z)$ complex analytic for $|z|<R$, $R>0$. Let $r_1,\;r_2$ and $r_3$ be the roots of the indicial polynomial $$q(r)=r(r-1)(r-2)+r(r-1)a(0)+rb(0)+c(0),$$ ordered such that
  $\mbox{Re}(r_1)\geq \mbox{Re}(r_2)\geq\mbox{Re}(r_3)$. \begin{itemize}
\item[(i)] If $r_1=r_2=r_3$ then there exist three linearly independent solutions $\varphi_1,\varphi_2,\varphi_3$ of the form:
$$\varphi_1(z)=z^{r_1}\sigma_1(z),\;\;\varphi_2(z)=z^{r_1+1}\sigma_2(z)+(\log z)\varphi_1(z)$$
and
$$\varphi_3(z)=z^{r_1+1}\sigma_3(z)+2(\log z)\varphi_2(z)-(\log z)^2\varphi_1(z), $$
where $\sigma_1,\sigma_2,\sigma_3$ are analytic in $|z|<R$ and $\sigma_1(0)\neq0$.

\item[(ii)] If $r_1=r_2$ and $r_1-r_3\in\mathbb{Z}^+$ there exist three linearly independent solutions   $\varphi_1,\varphi_2,\varphi_3$ of the form:
$$\varphi_1(z)=z^{r_1}\sigma_1(z),\;\;\varphi_2(z)=z^{r_1+1}\sigma_2(z)+(\log z)\varphi_1(z)$$
and
$$\varphi_3(x)=z^{r_3+2}\sigma_3(z)+c\;(\log z)\varphi_1(z), $$
where $c$ constant, $\sigma_1,\sigma_2,\sigma_3$ are analytic in $|z|<R$, $\sigma_1(0)\neq0$ and $\sigma_3(0)\neq0$ .

\item[(iii)] If $r_2=r_3$ and $r_1-r_2\in\mathbb{Z}^+$ there exist three linearly independent solutions $\varphi_1,\varphi_2,\varphi_3$ of the form:
$$\varphi_2(z)=z^{r_2}\sigma_2(z),\;\;\varphi_3(z)=z^{r_2+1}\sigma_3(z)+(\log z)\varphi_2(z)$$
and
$$\varphi_1(z)=z^{r_2+2}\sigma_1(z)+c\;(\log z)\varphi_2(z), $$
where $c$ constant, $\sigma_1,\sigma_2,\sigma_3$ are analytic in $|z|<R$, $\sigma_1(0)\neq0$ and $\sigma_2(0)\neq0$.

\item[(iv)] If $r_1-r_2\in\mathbb{Z}^+$ and $r_2-r_3\in\mathbb{Z}^+$ there exist three linearly independent solutions $\varphi_1,\varphi_2,\varphi_3$ of  the form:
$$\varphi_1(z)=z^{r_1}\sigma_1(z),\;\;\varphi_2(z)=z^{r_2}\sigma_2(z)+c\;(\log z)\varphi_1(z)$$
and
$$\varphi_3(z)=z^{r_3}\sigma_3(z)+\tilde{c}\;(\log z)\varphi_2(z), $$
where $c$ and $\tilde{c}$ are constants, $\sigma_1,\sigma_2,\sigma_3$ are analytic in $|z|<R$, $\sigma_1(0)\neq0$, $\sigma_2(0)\neq0$ and $\sigma_3(0)\neq0$.
\end{itemize}
\end{theorem}

The relevance  of this situation is explained by  Example~\ref{exeminf}.

\noindent{\bf Acknowledgement}: Theorem~\ref{thmfrobb3complexII} or, more generally, Frobenius method for order $n$ linear homogeneous
complex analytic equations can be found in \cite{Ince} throughout Chapter XVI (\S 16.1 page 396 and on).
Indeed, in \S 16.1 (page 396) the author proceeds in the classical way to prove the existence of a formal solution of the form $z^r \sum\limits_{n=0}^\infty a_n z^n$ where $r$ is a (maybe complex) root of the indicial equation. Next, in \S 16.2 page 398 the same author makes use of Cauchy Integral Formula to prove, always in the case of a complex regular singular point, the existence of  a first "convergent" solution of the form $z^r \sum\limits_{n=0}^\infty a_n z^n$ where the  power series $\sum\limits_{n=0}^\infty a_n z^n$ converges in some neighborhood of the singular point. Next, in \S 16.3 page 400 the  author hints the form of the other possible solutions according to the disposition of the roots of the indicial equation. Our Theorems~\ref{thmfrobb3complexI} and \ref{thmfrobb3complexII} above give a solid confirmation of this statement. Indeed, we discuss more accurately the connection between the disposition of the roots of the indicial equation and the types of the solutions. Moreover, our proof of the convergence is more elementary, without the need of Cauchy Integral Formula, and our estimates in the coefficients of the power series
are more clear and may be used in a computing process in order to control the speed of the convergence, something which is fundamental in applications to engineering. Finally, it is not clear from the argumentation in \cite{Ince} that the real case, ie., the case of real ODEs may be treated in the same way. Indeed, the roots of the indicial equation may be complex non-real and therefore the coefficients in the power series in the formal solution may be complex non-real, which would not be useful in the search of real solutions.
A stronger evidence of this fact is given in Example~\ref{Example:realcomplex} where a real ODE of third order gives rise to one real root and two complex conjugate roots for the indicial equation. This complexity spreads throughout the recurrence and gives complex coefficients for the power series of the solutions. This shows that the
natural idea of starting from a real analytic ODE, considering its complexification, applying the Frobenius
methods in \cite{Ince} and then considering a sort of "decomplexification" of the solution, may not be
a reasonable way of finding the (real) solutions of the original (real) ODE.
To overcome this difficult is one of the gains of our \S 3.

\subsubsection{Convergence of formal solutions in order three}

Similarly to Theorem~\ref{Theorem:B} we may prove, for third order ODEs the following
convergence theorem:

\begin{theorem}[1 formal solution third order regular singularity]
\label{Theorem:1formalregularorderthree}
Consider the third order ordinary differential equation given by
\begin{equation}\label{edo3}
a(x)y^{\prime\prime\prime}+b(x)y^{\prime\prime}+c(x)y^\prime+d(x)y=0,
\end{equation}
where $a,b,c$ are analytic functions at the origin $0 \in \mathbb R$.
Assume that \eqref{edo3}
has at $x=0$  an ordinary point or a regular singular point.
Then a formal solution $\hat y(x) = \sum\limits_{n=0}^\infty a_n x^n$ is always convergent in some neighborhood
$|x|<R$ of the origin. Indeed, this solution  converges in the same maximal interval $]-R,R[\subset \mathbb R$ where the coefficients $a(x), b(x), c(x)$ are analytic.\end{theorem}

\begin{Conjecture}Consider a linear homogeneous ordinary differential equation of third order given by
$$a(x)y^{\prime\prime\prime }+b(x)y^{\prime \prime} +c(x)y^\prime + d(x) y=0$$
where $a,b,c, d$ are analytic functions at $x_0 \in \mathbb R$. Suppose that this equation admits three
linearly independent formal  solutions $y_1(x), y_2(x), y_3(x)$. Then $x_0$ is an ordinary point or a regular singular point for the ODE  and therefore the solutions are analytic convergent.
\end{Conjecture}

\begin{Conjecture}
Consider a linear homogeneous ordinary differential equation of third order given by
$$a(x)y^{\prime\prime\prime}+b(x)y^{\prime\prime}+c(x)y^\prime+d(x)y=0$$
where $a,b,c, d$ are analytic functions at $x_0 \in \mathbb R$. Suppose that this equation admits three
linearly independent  solutions $y_1(x), y_2(x), y_3(x)$ which are linear combinations of log type and
real or complex power type with  analytic functions as coefficients. Then $x_0$ is an ordinary point or a regular singular point for the ODE.
\end{Conjecture}

\section{Second order equations}

The classical second order {\it Euler equation}  is $ax^2 y^{\prime \prime} + b x y ^\prime + c y=0$ where
$a, b, c \in \mathbb R, \mathbb C$ are constants and $ a \ne 0$. The classical method of solution is to look for solutions of the form $y=x^r$ and to obtain a second degree equation on $r$ of the form $ar(r-1) +
br + c=0$, called associate {\it indicial equation} (see \cite{Boyce} Chapter~5). This is the basis of the Frobenius methods for solving non-constant coefficients suitable second order ODEs as it is well-known (\cite{Boyce}, \cite{C}). In this part of our work we push further these techniques by looking at the
convergence of formal solutions, characterization of the regular singularity case and complete description of the cases where there are solutions of Liouvillian type. We also associate to  a polynomial second order linear homogeneous complex ODE, a Riccati first order ODE and consequently a group of Moebius maps, called the
{\it global holonomy} of the ODE. The case where the group is trivial corresponds to an important class of ODEs which can be explicitly integrated.

\subsection{Second order Euler equations}

Let us first give a characterization of Euler equations in the complex plane, in terms of
regularity of its singular points.

\begin{Theorem}\label{chareulerequation}{\rm Consider a second order differential equation
\begin{equation}\label{chareuler1}
z^2u^{\prime\prime}+zb(z)u^\prime+c(z)u=0
\end{equation}
where $b,c$ are entire functions in the complex plane. Then (\ref{chareuler1}) is an  Euler equation if, and only if,  the origin and the infinity are regular singular points of (\ref{chareuler1}).}
\end{Theorem}
\begin{proof}{\rm It is a straightforward computation to check that an Euler equation has the origin and the infinity as singular points. Let us see the converse. Putting  $z=1/t$. Let $w(t)=u\big(\frac{1}{t}\big)$
$$u^\prime\big(\frac{1}{t}\big)=-t^2w^\prime(t), \, \,
u^{\prime\prime}\big(\frac{1}{t}\big)=t^4w^{\prime\prime}(t)+2t^3w^\prime(t),$$
hence
$$\frac{1}{t^2}u^{\prime\prime}\big(\frac{1}{t}\big)+\frac{1}{t}b\big(\frac{1}{t}\big)u^\prime\big(\frac{1}{t}\big)+c\big(\frac{1}{t}\big)u\big(\frac{1}{t}\big)=0$$
$$\frac{1}{t^2}[t^4w^{\prime\prime}(t)+2t^3w^\prime(t)]+\frac{1}{t}b\big(\frac{1}{t}\big)[-t^2w^\prime(t)]+c\big(\frac{1}{t}\big)w(t)=0$$
\begin{equation}\label{chareuler2}
t^2w^{\prime\prime}(t)+t\big[2-b\big(\frac{1}{t}\big)\big]w^\prime(t)+c\big(\frac{1}{t}\big)w(t)=0.
\end{equation}
Given that the infinity is a regular singular point of (\ref{chareuler1}) we have that $t=0$ is a regular singular point of (\ref{chareuler2}) consequently there exist the limits

$$\displaystyle\lim_{t\to 0}\frac{t^2\big[2-b\big(\frac{1}{t}\big)\big]}{t^2}=\displaystyle\lim_{t\to 0}\big[2-b\big(\frac{1}{t}\big)\big]=\displaystyle\lim_{t\to 0}\big[(2-b_0)-\frac{b_1}{t}+\frac{b_2}{t^2}+\ldots\big]$$
 and
$$\displaystyle\lim_{t\to 0}\frac{t^2c\big(\frac{1}{t}\big)}{t^2}=\displaystyle\lim_{t\to 0}c\big(\frac{1}{t}\big)=\displaystyle\lim_{t\to 0}\big[c_0+\frac{c_1}{t}+\frac{c_2}{t^2}+\ldots\big].$$
Given that these limits exist we have that $0=b_1=b_2=\ldots$ and $0=c_1=c_2=\ldots$. Hence equation (\ref{chareuler1}) is of the form
$$z^2u^{\prime\prime}+zb_0u^\prime+c_0u=0$$
  which is an  Euler equation.}
\end{proof}

A far more general statement is found below:

\begin{Theorem}\label{charsing}{\rm Consider a second order differential equation
\begin{equation}\label{charsing1}
a(z)u^{\prime\prime}+b(z)u^\prime+c(z)u=0
\end{equation}
where $a,b,c$ are entire functions in the complex plane with $a^\prime(0)\neq0$. The origin and the infinity are regular singular points of (\ref{charsing1}) if and only if there exist  $k=0,1,2,\ldots$ in such a way that equation (\ref{charsing1}) is of the form
\begin{equation}\label{charsing2}
(A_1z+\ldots+A_{k+2}z^{k+2})u^{\prime\prime}+(B_0+B_1z+\ldots+B_{k+1}z^{k+1})u^\prime+(C_0+C_1z+\ldots+C_{k}z^{k})u=0
\end{equation}
where $A_1,\ldots,A_{k+2},B_0,\ldots,B_{k+1},C_0,\ldots,C_{k}$ are constants such that $A_1,A_{k+2}\neq0$.
.}
\end{Theorem}
\begin{proof}{\rm Let us first see that (\ref{charsing2}) has the origin and the infinity as regular singular point. Clearly the origin é singular point of (\ref{charsing2}) and since there exist the limits

$$\displaystyle\lim_{z\to 0}\frac{z(B_0+B_1z+\ldots+B_{k+1}z^{k+1})}{(A_1z+\ldots+A_{k+2}z^{k+2})}=\displaystyle\lim_{z\to 0}\frac{B_0+B_1z+\ldots+B_{k+1}z^{k+1}}{(A_1+\ldots+A_{k+2}z^{k+1})}=\frac{B_0}{A_1}$$
 and
$$\displaystyle\lim_{z\to 0}\frac{z^2(C_0+C_1z+\ldots+C_{k}z^{k})}{A_1z+\ldots+A_{k+2}z^{k+2}}=\displaystyle\lim_{z\to 0}\frac{z(C_0+C_1z+\ldots+C_{k}z^{k})}{A_1+\ldots+A_{k+2}z^{k+1}}=0$$we have that the origin is a regular singular point. Putting  $z=1/t$. Let $w(t)=u\big(\frac{1}{t}\big)$
$$u^\prime\big(\frac{1}{t}\big)=-t^2w^\prime(t), \, \, u^{\prime\prime}\big(\frac{1}{t}\big)=t^4w^{\prime\prime}(t)+2t^3w^\prime(t),$$
hence
$$\big(\frac{A_1}{t}+\ldots+\frac{A_{k+2}}{t^{k+2}}\big)u^{\prime\prime}\big(\frac{1}{t}\big)+\big(B_0+\frac{B_1}{t}+\ldots+\frac{B_{k+1}}{t^{k+1}}\big)u^\prime\big(\frac{1}{t}\big)+\big(C_0+\frac{C_1}{t}+\ldots+\frac{C_{k}}{t^{k}}\big)u\big(\frac{1}{t}\big)=0$$

$$\begin{array}{c} \big(\frac{A_1}{t}+\ldots+\frac{A_{k+2}}{t^{k+2}}\big)[t^4w^{\prime\prime}(t)+2t^3w^\prime(t)]+\big(B_0+\frac{B_1}{t}+\ldots+\frac{B_{k+1}}{t^{k+1}}\big)[-t^2w^\prime(t)]\\
\\+\big(C_0+\frac{C_1}{t}+\ldots+\frac{C_{k}}{t^{k}}\big)w(t)=0\end{array}$$
\begin{equation}\label{charsing3}
\begin{array}{c}
\big(A_1t^{k+3}+\ldots+A_{k+2}t^2\big)w^{\prime\prime}(t)+\big((2A_1-B_0)t^{k+2}+\ldots+(2A_{k+2}-B_{k+1})t\big)w^\prime(t)\\
\\+\big(C_0t^k+\ldots+C_{k}\big)w(t)=0.\end{array}
\end{equation}
Observe that the origin is a singular point of (\ref{charsing3}) and since there exist the limits
$$\begin{array}{c}\displaystyle\lim_{t\to 0}\frac{t\big((2A_1-B_0)t^{k+2}+\ldots+(2A_{k+2}-B_{k+1})t\big)}{A_1t^{k+3}+\ldots+A_{k+2}t^2}\\
\\=\displaystyle\lim_{t\to 0}\frac{(2A_1-B_0)t^{k+1}+\ldots+(2A_{k+2}-B_{k+1})}{A_1t^{k+1}+\ldots+A_{k+2} }=\frac{2A_{k+2}-B_{k+1}}{A_{k+2}}\end{array}$$
 and
$$\displaystyle\lim_{t\to 0}\frac{t^2\big(C_0t^k+\ldots+C_{k}\big)}{A_1t^{k+3}+\ldots+A_{k+2}t^2}=\displaystyle\lim_{t\to 0}\frac{C_0t^k+\ldots+C_{k}}{A_1t^{k+1}+\ldots+A_{k+2}}=\frac{C_k}{A_{k+2}}$$we have that o the origin is a regular singular point of (\ref{charsing3}) consequently the infinity is a regular singular point of (\ref{charsing2}). Conversely, assume that the origin and the infinity are regular singular points of (\ref{charsing1}). Hence by the change of coordinates $z=1/t$ and considering $v(t)=u\big(\frac{1}{t}\big)$
$$u^\prime\big(\frac{1}{t}\big)=-t^2v^\prime(t), \, \, u^{\prime\prime}\big(\frac{1}{t}\big)=t^4v^{\prime\prime}(t)+2t^3v^\prime(t),$$
hence
$$a\big(\frac{1}{t}\big)u^{\prime\prime}\big(\frac{1}{t}\big)+b\big(\frac{1}{t}\big)u^\prime\big(\frac{1}{t}\big)+c\big(\frac{1}{t}\big)u\big(\frac{1}{t}\big)=0$$
$$a\big(\frac{1}{t}\big)\big[t^4v^{\prime\prime}(t)+2t^3v^\prime(t)\big]+b\big(\frac{1}{t}\big)\big[-t^2v^\prime(t)\big]+c\big(\frac{1}{t}\big)u\big(\frac{1}{t}\big)=0$$
\begin{equation}\label{charsing4}
t^4a\big(\frac{1}{t}\big)w^{\prime\prime}(t)+\big[2t^3a\big(\frac{1}{t}\big)-t^2b\big(\frac{1}{t}\big)\big]w^\prime(t)+c\big(\frac{1}{t}\big)w(t)=0.
\end{equation}
Given that the infinity is a regular singular point of (\ref{charsing1}) then the origin is a regular singular point of (\ref{charsing4}). Hence, there exist the limits

$$\displaystyle\lim_{t\to 0}\frac{t\big[2t^3a\big(\frac{1}{t}\big)-t^2b\big(\frac{1}{t}\big)\big]}{t^4a\big(\frac{1}{t}\big)}=\displaystyle\lim_{t\to 0}\big[2-\frac{b\big(\frac{1}{t}\big)}{ta\big(\frac{1}{t}\big)}\big]=2-\lim_{t\to 0}\frac{b_0+\frac{b_1}{t}+\frac{b_2}{t^2}+\ldots}{a_1+\frac{a_2}{t}+\frac{a_3}{t^2}+\ldots}$$
 and
$$\displaystyle\lim_{t\to 0}\frac{t^2c\big(\frac{1}{t}\big)}{t^4a\big(\frac{1}{t}\big)}=\displaystyle\lim_{t\to 0}\frac{c\big(\frac{1}{t}\big)}{t^2a\big(\frac{1}{t}\big)}=\displaystyle\lim_{t\to 0}\frac{c_0+\frac{c_1}{t}+\frac{c_2}{t^2}+\ldots}{a_1t+a_2+\frac{a_3}{t}+\ldots}$$
hence we have there exists  $k=0,1,2\ldots$ in such a way that $a_{k+2}\neq0$, $0=a_{k+3}=a_{k+4}=\ldots$, $0=b_{k+1}=b_{k+2}=\ldots$ and $0=c_{k+1}=c_{k+2}=\ldots$. Given that the origin is a regular singular point of (\ref{charsing1}) we have that there exist the limits
$$\displaystyle\lim_{z\to 0}\frac{zb(z)}{a(z)}=\displaystyle\lim_{z\to 0}\frac{b_0z+b_1z^2+\ldots+b_{k+1}z^{k+2}}{a_1z+a_2z^2+\ldots+a_{k+2}z^{k+2}}=\displaystyle\lim_{z\to 0}\frac{b_0+b_1z+\ldots+b_{k+1}z^{k+1}}{a_1+a_2z+\ldots+a_{k+2}z^{k+1}}$$
 and
$$\displaystyle\lim_{z\to 0}\frac{z^2c(z)}{a(z)}=\displaystyle\lim_{z\to 0}\frac{c_0z^2+c_1z^3+\ldots+c_{k}z^{k+2}}{a_1z+a_2z^2+\ldots+a_{k+2}z^{k+2}}=\displaystyle\lim_{z\to 0}\frac{c_0z+c_1z^2+\ldots+c_{k}z^{k+1}}{a_1+a_2z+\ldots+a_{k+2}z^{k+1}}$$
 provided that  $a_1\neq0$. Thence (\ref{charsing1}) is of the form (\ref{charsing2}).
}
\end{proof}

\subsection{Convergence of formal solutions}

Consider a second order ordinary differential equation given by
$$a(z)u^{\prime\prime}+b(z)u^\prime+c(z)u=0$$
where $a,b,c$ are holomorphic in a neighborhood of the origin $0\in \mathbb C$.
We shall mainly address two questions:

\noindent Question (i): {\sl Under what conditions can we assure that the origin is an ordinary point or a
regular singular point of the equation?}
\vglue.1in
\noindent Question (ii): {\sl Is it that a formal solution of the ODE is always convergent?}

\begin{Theorem}[formal solutions order two]
\label{Theorem:twoformalsolutionsconverge}
Consider a second order ordinary differential equation given by
\begin{equation}\label{twoformal}
a(x)y^{\prime\prime}+b(x)y^\prime+c(x)y=0
\end{equation}
where $a,b,c$ are real or complex  analytic functions at $x_0 \in \mathbb R, \mathbb C$.
Suppose also that there exist two linearly independent formal
solutions  $\hat y_1(x)$ and $\hat y_2(x)$ centered at $x_0$  of  equation (\ref{twoformal}).
Then $x_0$ is an ordinary point or a regular singular point of
 (\ref{twoformal}).
Moreover, $\hat y_1(x)$ and $\hat y_2(x)$ are convergent.
\end{Theorem}

Let us give a first proof of the convergence:

\begin{proof}[Convergence in Theorem~\ref{Theorem:twoformalsolutionsconverge}]
First we consider the complex analytic case, i.e., $z$ is a complex variable and the coefficients
$a(z), b(z), c(z)$ are  complex analytic (holomorphic) functions in neighborhood $|z-z_0|<R$ of $z_0\in \mathbb C$.
   According to \cite{Reis} there is an integrable complex analytic one-form $\Omega$ in $\mathbb C^3$
 defined as follows $$\Omega=-a(z)ydx + a(z)xdy +[a(z)y^2+b(z)xy +c(z)x^2]dz.$$
 This one-form is tangent to the vector field in $\mathbb C^3$  associated to the reduction of order of the ODE.
 Moreover, given two solutions $u_1(z)$ and $u_2(z)$ of the ODE the function
 $H=\frac{xu_1^\prime(z)-yu_1(z)}{xu_2^\prime(z)-yu_2(z)}$ is a first integral for the form $\Omega$,
 ie., $dH\wedge \Omega=0$.

 By hypothesis there exist two linearly independent formal solutions $\hat u_1$ and $\hat u_2$ of  equation (\ref{twoformal}).
 Each solution writes as a formal complex power series $\hat u_j(z)=\sum\limits_{n=0}^\infty a_n ^j z^n\in \mathbb
 C\{\{z\}\}$.
According to the above,  there exists a first integral purely formal
$$H=\frac{x\hat u_1^\prime(z)-y\hat u_1(z)}{x\hat u_2^\prime(z)-y\hat u_2(z)}$$
of the integrable one-form $\Omega$ above.

Now we recall the following convergence theorem:
\begin{Theorem}[Cerveau-Mattei, \cite{Ce-Mt}, Theorem 1.1 page 106]\label{intformalpure}
 Let $\Omega$ be a germ at  $0\in\mathbb{C}^n$ of an integrable holomorphic
 1-form  and $H=\frac{f}{g}\in \hat{\mathcal{M}}_n$ a
 purely formal meromorphic first integral  of $\Omega$, i.e.,
 $\Omega\wedge dH=0$ and $H,1/H\notin\hat{\mathcal{O}}_n$. Then $H$ converges, i.e.,  $H\in\mathcal{M}_n$.
\end{Theorem}

From the above theorem  there exist $f,g\in\mathcal{O}_3$ such that $H=\frac{f}{g}$.

Putting $x=0$ and $y=y_0$ small enough and non-zero we have $\frac{\hat u_1(z)}{\hat u_2(z)}=\frac{f_1(z)}{g_1(z)}.$ Also, putting $y=0$ and $x=x_0$ small enough and non-zero we have $\frac{\hat u_1^\prime(z)}{\hat u_2^\prime(z)}=\frac{f_2(z)}{g_2(z)}$. Hence we have
\begin{equation}\label{intform1}
\hat u_1=\frac{f_1}{g_1}\hat u_2
\end{equation}
and
\begin{equation}\label{intform2}
\hat u_1^\prime=\frac{f_2}{g_2}\hat u_2^\prime.
\end{equation}
Derivating (\ref{intform1}) we have
\begin{equation}\label{intform3}
\hat u_1^\prime=\big(\frac{f_1}{g_1}\big)^\prime\hat u_2+\frac{f_1}{g_1}\hat u_2^\prime.
\end{equation}
Replacing (\ref{intform2}) in (\ref{intform3}) we have
$$\frac{f_2}{g_2}\hat u_2^\prime=\big(\frac{f_1}{g_1}\big)^\prime \hat u_2+\frac{f_1}{g_1}\hat u_2^\prime$$
\begin{equation}\label{intform4}
\big(\frac{f_2}{g_2}-\frac{f_1}{g_1}\big) \hat u_2^\prime=\big(\frac{f_1}{g_1}\big)^\prime \hat u_2.
\end{equation}
Observe that $$\frac{f_2}{g_2}-\frac{f_1}{g_1}\neq0$$ since $\hat u_1$ and $\hat u_2$ are  linearly independent.

Therefore from (\ref{intform4}) we obtain
$$\hat u_2^\prime=\frac{\big(\frac{f_1}{g_1}\big)^\prime}{\big(\frac{f_2}{g_2}-\frac{f_1}{g_1}\big)}u_2 $$
thus $$\hat u_2(z)=\exp\big(\dint^z_{z_0}\frac{\big(\frac{f_1(w)}{g_1(w)}\big)^\prime}{\big(\frac{f_2(w)}{g_2(w)}-\frac{f_1(w)}{g_1(w)}\big)}dw \big)$$is convergent and according to (\ref{intform1}) $\hat u_1$ also is convergent.

This proves the convergence part in Theorem~\ref{Theorem:twoformalsolutionsconverge}.

\end{proof}

We stress the fact that we are not assuming the ODE to be regular at $x_0$.

\subsubsection{The wronskian I}

Consider the linear homogeneous second order ODE
\begin{equation}\label{edo7}
a(x)y^{\prime\prime}+b(x)y^\prime+c(x)y=0
\end{equation}
where  $a(x),b(x),c(x)$ are differentiable real or complex functions defined in some open subset $U\subset \mathbb R, \mathbb C$. We may assume that $U$ is an open disc centered at the origin $0 \in \mathbb R, \mathbb C$. We make no hypothesis on the nature of the point $x=0$ as a singular or ordinary point of  (\ref{edo7}).
Given two solutions  $y_1$ and $y_2$ of (\ref{edo7}) their {\it wronskian} is defined by $W(y_1,y_2)(x)=y_1(x)y_2^\prime(x)-u_2(x)y_1^\prime(x)$.
\begin{Claim}
The wronskian $W(y_1,y_2)$ satisfies the following first order ODE
\begin{equation}\label{wronskiandef}
a(x)w^\prime+b(x)w=0.
\end{equation}
\end{Claim}
This is a well-known fact and we shall not present a proof, which can be done by straightforward computation.
Most important, the above fact allows us to introduce the notion of {\it wronskian} of a general second order linear homogeneous ODE as \eqref{edo7} as follows:

\begin{Definition}
The wronskian of \eqref{edo7} is defined as the general solution of \eqref{wronskiandef}.
\end{Definition}

Hence, in general the wronskian is of the form
\begin{equation}
\label{eq:wronskianform}
W(x)=K\exp\big(-\dint^x\frac{b(\eta)}{a(\eta)}d\eta\big)
\end{equation}
where  $K$ is a constant.

A well-known consequence of the above formula is the following:

\begin{Lemma}
Given solutions $y_1(x), y_2(x)$ the following conditions are equivalent:
\begin{enumerate}
\item $W(y_1,y_2)(x)$ is identically zero.
\item $W(y_1,y_2)(x)$ vanishes at some point $x=x_0$.
\item $y_1(x), y_2(x)$ are linearly dependent.
\end{enumerate}
\end{Lemma}

Let us analyze the consequences of this form. We shall consider the origin as the  center of our disc domain.
In what follows the coefficients are analytic in a neighborhood of the origin.

\noindent{\bf Case (1)}: If  $\frac{b}{a}$ has poles of order  $r>1$ at the origin:
 In this case we can write
$$\frac{b(x)}{a(x)}=\frac{A_r}{x^r}+\ldots+\frac{A_2}{x^2}+\frac{A_1}{x}+d(x)$$
where  $A_1,A_2,\ldots,A_r$ are constant, $A_r\neq0$ and  $d$ is analytic at the origin. Thus we have
$$W(x)=K\exp\big(-\dint^x\big(\frac{A_r}{w^r}+\ldots+\frac{A_2}{w^2}+\frac{A_1}{w}+d(w)\big)dw\big)$$
$$W(x)=K\exp\big(\frac{A_r}{(r-1)x^{r-1}}+\ldots+\frac{A_2}{x}-A_1\log |x|+\tilde{d}(x)\big)$$ where $\tilde{d}$ is analytic. Hence
$$W(x)=K|x|^{-A_1}\exp\big(\frac{A_r}{(r-1)x^{r-1}}+\ldots+\frac{A_2}{x}\big)\exp\big(\tilde{d}(x)\big).$$
Now observe that  $\exp\big(\frac{A_r}{(r-1)x^{r-1}}+\ldots+\frac{A_2}{x}\big)$ is neither analytic nor formal. Therefore, in this case,  $W$ is neither analytic nor formal.

\noindent{\bf Case (2)}: $\frac{b}{a}$ has poles of order $\leq 1$ at the origin.
In this case $$\frac{b(x)}{a(x)}=\frac{A_1}{x}+d(x)$$
and  $$W(x)=K|x|^{-A_1}\exp\big(\tilde{d}(x)\big).$$
If $W$ is analytic or formal then we must have $A_1\in\{0,-1,-2,-3,\ldots\}$.

Summarizing we have:

\begin{Lemma}
\label{Lemma:simplepole}
Assume that the wronskian $W$ of the ODE $a(x) y^{\prime\prime} + b(x) y^\prime + c(x) y=0$,
with analytic coefficients, is analytic or formal. Then
 $\frac{b}{a}$ has a pole of order  $r\leq 1$ at the origin. Moreover, we must have
 $W(x)=K|x|^{-A}\exp\big(f(x)\big)$, where $A\in\{0,-1,-2,-3,\ldots\}$ and $f$ is analytic.
\end{Lemma}

Now we are able to prove the remaining part of Theorem~\ref{Theorem:twoformalsolutionsconverge}:

\begin{proof}[End of the proof of Theorem~\ref{Theorem:twoformalsolutionsconverge}]
We have already proved the first part. Let us now prove that the origin is an ordinary point
or a regular singularity of the ODE. This is done by means of the two following claims:
\begin{Claim}
The quotient $\frac{b}{a}$ has poles of order $\leq 1$ at the origin.
\end{Claim}
\begin{proof}
Indeed, since by hypothesis there are two formal linearly independent functions,
the wronskian is formal. Thus, from the above discussion we conclude.
\end{proof}
The last part is done below. For simplicity we shall assume that $x=z\in \mathbb C$ and that the coefficients are complex analytic (holomorphic) functions.

\begin{Claim}
\label{Claim:poleordertwo}
We have $\displaystyle\lim_{z\to 0} z^2 \frac{c(z)}{a(z)} \in \mathbb C$.
\end{Claim}
\begin{proof}
Write $a(z) u^{\prime\prime} + b(z) u ^\prime + c(z) u=0$ and $a(z)=z^k$ according to the local form
of holomorphic functions. Since $\displaystyle\lim_{z \to 0}z\frac{b(z)}{a(z)}\in \mathbb C$ we must have
$\frac{b(z)}{a(z)}=\frac{\tilde b(z)}{z}$ for some holomorphic function $\tilde b(z)$ at $0$.
Then we have for the ODE above
$$
z^{3+ \nu} u ^{\prime \prime} + z^{ 2 + \nu} \tilde b(z) u ^\prime + \tilde c (z) u=0.
$$

Assume that the Claim is not true, then $\frac{c(z)}{a(z)}$ must have a pole of order $\geq 3$ at $0$.
Thus we may write $\frac{c(z)}{a(z)}=\frac{c(z)}{z^k}= \frac{\tilde c(z)}{z^{3+ \nu}}$ for some
holomorphic function $\tilde c(z)$ at $0$ and some $\nu \geq 0$. For sake of simplicity we will assume that $\tilde c(0)=1$ and $\nu=0$. This does not affect the argumentation
below. We write $\tilde b(z)= b_0 + b_1 z + b_2 z^2 +\ldots$ and $\tilde c(z)= 1 +c_1 z + c_2 z^2 + \ldots$ in power series.
Substituting this in the ODE we obtain
 $$
 z^{3+\nu}u^{\prime \prime} + z^{2+\nu} (b_0+b_1z+b_2 z^2+\ldots) u ^\prime +
 (1 +c_1 z + c_2 z^2 + \ldots) u=0.
$$
Now we write $u(z)=\sum\limits_{n=0}^\infty a_n z^n$ in power series.
We obtain
$$
\sum\limits_{n=2}^\infty n(n-1)a_n z^{n+ 1 + \nu} + (b_0 + b_1 z +\ldots)\sum\limits_{n=1}^\infty
na_n z^{n+1+\nu} + (1 + c_1 z+ c_2 z^2 +\ldots) \sum\limits_{n=0}^\infty a_n z^n=0.
$$

Now we start by comparing the lower powers of $z$ on each term of the above expression.
We have
$$
\sum\limits_{n=2}^\infty n(n-1) a_nz^{n+ 1 + \nu}=2a_2z^{3 +\nu} +6a_3 z^{4+ \nu}+ \ldots
$$

$$
(b_0 + b_1 z +\ldots)\sum\limits_{n=1}^\infty
na_n z^{n+1+\nu}= b_0 a_1 z^{ 2 + \nu} +(2 b_0a_2 + b_1 a_1) z^{3+\nu} +  \ldots
$$
and finally
$$
(1 + c_1 z+ c_2 z^2 +\ldots) \sum\limits_{n=0}^\infty a_n z^n= a_0 + (a_1 + c_1 a_0) z +
(a_2 + c_1 a_1 + a_0 c_2) z^2 + \ldots
$$
Starting now from the lowest powers of $z$ in the expression of the ODE above we obtain
$a_0=0$. Also we obtain $a_1 + c_1 a_0=0$ and therefore $a_1=0$. Since $\nu=0$ we have the power of $z^2$ this gives $b_0 a_1 + a_2 + c_1 a_1 + a_0 c_2=0$ and then $a_2=0$.
Now for the coefficient of $z^3$ we obtain $2a_2 + 2 b_0a_2 + b_1 a_1 + a_3=0$ and therefore $a_3=0$.
And so on we conclude that $a_n=0$, for all $n \geq 0$ ie., $u=0$ is the only possible formal solution.
This proves the claim by contradiction.
\end{proof}

The two claims above end the proof of Theorem~\ref{Theorem:twoformalsolutionsconverge}.
\end{proof}

Next we present a result that also implies a more simple proof of Theorem~\ref{Theorem:twoformalsolutionsconverge}.

\begin{Theorem}[1 formal solution second order regular singularity]
Consider a second order ordinary differential equation given by
\begin{equation}\label{oneformal}
a(x)y^{\prime\prime}+b(x)y^\prime+c(x)y=0
\end{equation}
where $a,b,c$ are analytic functions at $x_0 \in \mathbb R$.
Suppose that \eqref{oneformal} has at $x_0$  an ordinary point or a regular singular point.
Then a formal solution $\hat y(x) = \sum\limits_{n=0}^\infty a_n (x-x_0)^n$ is always convergent in some neighborhood
$|x-x_0|<R$ of the point $x_0$. Indeed, this solution  converges in the same disc type neighborhood where
the coefficients $a(x), b(x), c(x)$ are analytic.\end{Theorem}
\begin{proof}
First of all we are assuming that the origin is an ordinary point or a regular singularity of the ODE.
If it is an ordinary point, then by the classical existence theorem for ODEs there are two linearly independent analytic solutions and any solution, formal or convergent, will be a linear combination of these two solutions. Such a solution is therefore convergent.

Thus we may write the ODE as
$$
x^2 y^{\prime \prime} + xb(x)y^\prime + c(x) y=0
$$
where the new coefficients $b(x)$ and $c(x)$, obtained after
renaming $xb(x)/a(x)$ and $x^2c(x)/a(x)$ conveniently, are analytic.

Let us consider a formal solution $\hat y(x)=\sum\limits_{n=0}^\infty d_n x^n$.
We can write $\hat y(x)=x^{r_1}(1+\vr(x))$ for some $r_1 \geq 0$ and $\vr(x)$ a formal function
with $\vr(0)=0$. In other words, $r_1\in \{0,1,2,\ldots\}$ is the order of $\hat y(x)$ at the origin.
Then we have
$\hat y^\prime(x)= r_1 x^{r_1-1} ( 1 + \vr(x)) + x^{r_1} \vr^\prime(x)$ and
$\hat y^{\prime \prime}(x)=r_1 (r_1-1) x^{r_1-2} (1 + \vr(x)) + 2 r_1 x^{r_1-1} \vr^\prime(x)+x^{r_1}\vr^{\prime\prime}(x)$.
Substituting this in the ODE $x^2 \hat y^{\prime \prime}(x) + xb(x)\hat y^\prime(x) + c(x) \hat y(x)=0$ and dividing by $x^{r_1}$ we obtain
$$
r_1(r_1-1)(1 +\vr(x)) + 2r_1 x \vr^\prime(x) + x^2 \vr^{\prime \prime} (x) +
r_1 b(x) (1 + \vr(x)) + x b(x) \vr^{\prime}(x) + c(x)(1 + \vr(x))=0.
$$

For $x=0$, since $\vr(0)=0$,  we then obtain the equation
$$
r_1(r_1-1) + r_1 b(0) + c(0)=0.
$$

The above is exactly the indicial equation associated to the original ODE.
We then conclude that the original  ODE has an indicial equation with a root
$r_1$ that belongs to the set of non-negative integers.
Let now $r\in \mathbb Z$ be the other root of the indicial equation. There are two possibilities:

\noindent (i) $r\geq r_1$. In this case,  then according to Frobenius classical theorem
we conclude that there is at least one solution $y_r(x)=x^r \sum\limits_{n=0}^\infty e_n x^n$ which is convergent. There are two possibilities:

\noindent (i.1) $y_r(x)$ and $\hat y(x)$ are linearly dependent: in this case, $y_r(x)=\ell\cdot\hat y(x)$
for some constant $\ell\in \mathbb R, \mathbb C$. Then $r=r_1$ and therefore $y_r(x)$ is analytic and the same holds for $\hat y(x)$. More precisely,  $\hat y(x)$ is analytic in the same neighborhood $|x|<R$ where $b(x), c(x)$ are convergent.

\noindent (i.2) $y_r(x)$ and $\hat y(x)$ are linearly independent: Since $y_r(x)$ is analytic and seeing  $y_r(x)$ as a formal solution, we have two linearly independent formal solutions. From what we have seen above in Theorem~\ref{Theorem:twoformalsolutionsconverge} both solutions are convergent in the common disc domain of analyticity of the functions $b(x), c(x)$.

\noindent (ii) $r_1\geq r$. In this case,  then according to Frobenius classical theorem
we conclude that there is at least one solution $\tilde y_{r_1}(x)=x^{r_1}\sum\limits_{n=0}^\infty f_n x^n$, where the power series is convergent. There are two possibilities:

\noindent (ii.1) $\tilde y_{r_1}(x)$ and $\hat y(x)$ are linearly dependent: in this case, $\tilde y_{r_1}(x)=\tilde \ell \cdot\hat y(x)$ for some constant $\tilde \ell\in \mathbb R, \mathbb C$. Then $r=r_1$ and therefore $\tilde y_{r_1}(x)$ is analytic and the same holds for $\hat y(x)$. More precisely,  $\hat y(x)$ is analytic in the same neighborhood $|x|<R$ where $b(x), c(x)$ are convergent.

\noindent (ii.2) $\tilde y_{r_1}(x)$ and $\hat y(x)$ are linearly independent: in this case, $\tilde y_{r_1}(x)$ is analytic and seeing  $\tilde y_{r_1}(x)$ as a formal solution, we have two linearly independent formal solutions. From what we have seen above in Theorem~\ref{Theorem:twoformalsolutionsconverge} both solutions are convergent in the common disc domain of analyticity of the functions $b(x), c(x)$.
\end{proof}

The above proof still makes use of the convergence part in Theorem~\ref{Theorem:twoformalsolutionsconverge}, thus it cannot be used to give an alternative proof of Theorem~\ref{Theorem:twoformalsolutionsconverge}.
Let us work on a totally independent proof of Theorem~\ref{Theorem:twoformalsolutionsconverge} based only
on classical methods of Frobenius and ODEs.
For this sake we shall need a few lemmas.

\subsubsection{The wronskian II}

We consider the ODE $x^2  y^{\prime \prime}  + xb(x) y^\prime + c(x) y=0$ with  a regular
singular point at the origin.

\begin{Lemma}
Let $\hat y(x)$ be a formal solution of the ODE. Then we must have $\hat y(x)=x^r (1+ \sum\limits_{n=1}^\infty a_n x^n)$
where $r$ is a root of the indicial equation of the ODE.
\end{Lemma}
\begin{proof}
Indeed, from the last proof if we write
$\hat y(x)=x^r(1+\vr(x))$ for some $r \geq 0$ and $\vr(x)$ a formal function
with $\vr(0)=0$ then we have that $$
r(r-1) + r b(0) + c(0)=0.
$$ which  is exactly the indicial equation associated to the original ODE.
\end{proof}
\begin{Remark}
Let $r\in \{0,1,2,\ldots\}$ be a root of the indicial equation
and assume that we have two solutions $\hat y_1(x)=x^r (1+ \vr_1(x))$ and $\hat y_2(x)=x^r (1 + \vr_2(x))$
which are formal. The wronskian writes
$W(\hat y_1,\hat y_2)(x)= \hat y_1(x)\hat y_2^\prime(x)-\hat y_1^\prime(x)\hat y_2(x)=x^r(1 +\vr_1(x))[r x^{r-1} (1+ \vr_2(x)) + x^r \vr_2 ^\prime(x)] -
 x^r(1 +\vr_2(x))[r x^{r-1} (1+ \vr_1(x)) + x^r \vr_1 ^\prime(x)] =x^{2r}[\varphi^\prime_2(x)(1+\varphi_1(x))-\varphi^\prime_1(x)(1+\varphi_2(x))]$.

 Then we have two cases:

 \noindent{\bf (i)} $r \geq 1$. In this case $W(\hat y_1,\hat y_2)(0)=0$. In this situation we must have $W(\hat y_1,\hat y_2) (x)=0$ and therefore $\hat y_1, \hat y_2$ are linearly dependent.

 \noindent {\bf (ii)} $r=0$.

\end{Remark}
Let us proceed. We are assuming now that we have two formal solutions $\hat y_1,\hat y_2$ for the ODE above.
We write $\hat y_j(x)=x^{r_j}(1 + \vr_j(x))$ for some formal series $\vr_j(x)$ that satisfies $\vr_j(0)=0$.
The exponents $r_j$ are non-negative integers and from what we have seen above, these are roots of the indicial equation
$r(r-1) + r b(0) + c(0)=0$ of the ODE. We may assume that $r_1 \geq r_2$.

So we have the following possibilities:

\noindent{\bf (i)} $r_1=r_2$. If this is the case we cannot a priori assure that the indicial equation has only the
root $r=r_1=r_2$.
Anyway, if $r \ne 0$ then from what we have seen above the formal solutions $\hat y_1, \hat y_2$ are linearly dependent.
This is a contradiction. Thus we must have $r=0$. If $r=0$ is the only root of the indicial equation
then we have a basis of the solution space given by $y_1(x)=1+\sum\limits_{n=1}^\infty e_n x^n$
and $y_2(x)=y_1(x) \log |x| + \sum\limits_{n=1}^\infty f_n x^n$.
If a linear combination $\hat y(x)=c_1 y_1(x) + c_2 y_2(x)$ is a formal function then necessarily $c_2=0$.
Thus any two formal solutions are linearly dependent.
Assume now that $r=0$ is not the only root of the indicial equation. Denote by $\tilde r\in\mathbb Z^*$ the other root of the indicial equation. There are two possibilities:

\noindent (a) $\tilde r>0$ then there is a basis of solutions given by
$y_1(x) = x^{\tilde r} (1+ \sum\limits_{n=1}^\infty g_n x^n)$ and
$y_2(x) = a y_1(x)\log|x| + |x|^0( 1 + \sum\limits_{n=1}^\infty  h_n x^n)$.
Let $y(x)=c_1 y_1(x) + c_2 y_2(x)$ be a formal power series.
Then $y(x)= (c_1 + a c_2 \log|x|)x^{\tilde r} (1+ \sum\limits_{n=1}^\infty g_n x^n) + c_2
( 1 + \sum\limits_{n=1}^\infty  h_n x^n). $ If $y(x)$ is a formal power series then we must have
$a c_2=0$ and therefore  $y(x)= c_1 x^{\tilde r} (1+ \sum\limits_{n=1}^\infty g_n x^n) + c_2
( 1 + \sum\limits_{n=1}^\infty  h_n x^n)$. In particular, since $\tilde r \in \mathbb N$,  $y(x)$ is convergent. This shows that the
formal solutions $\hat y_1(x), \hat y_2(x)$ are convergent and this is the only possible case where they can be linearly independent.

\noindent (b) $\tilde r<0$ then there is a basis of solutions given by
$y_1(x)=1 + \sum\limits_{n=1}^\infty
p_n x^n$ and $y_2(x)=ay_1(x) \log|x| + x^{\tilde r} (1 + \sum\limits_{n=1}^\infty q_n  x^n)$.
Write $y(x)= c_1 y_1(x) + c_2 y_2(x)$ for a linear combination of $y_1(x)$ and $y_2(x)$. Then
$y(x)=(c_1 + a c_2 \log|x|)(1 + \sum\limits_{n=1}^\infty
p_n x^n)  + c_2 (x ^{\tilde r} (1 + \sum\limits_{n=1}^\infty q_n  x^n))$.
If $y(x)$ is a formal series then necessarily $ac_2=0$ (because of the term $\log|x|$) and
also $c_2=0$ in this case because $\tilde r<0$. Thus we get $y(x) = c_1 y_1(x)$ which is convergent.
This shows that again we must have that  $\hat y_1$ and $\hat y_2$ are multiple of $y_1$ and therefore they are linearly dependent, contradiction again.

\noindent{\bf (ii)} $0< r_1 - r_2 =N \in \mathbb N$.
This case follows from facts already used above.
Since $r_1 > r_2$ and since each $r_j$ is a root of the indicial equation, we conclude that
these are the roots of the indicial  equation.
By Frobenius theorem there is a basis of the solutions given by $y_1(x)= x^{r_1}(1+ \sum\limits_{n=1}^\infty s_n x^n) $
and $y_2(x) =a y_1(x) \log|x| + |x|^{r_2}(1+\sum\limits_{n=1}^\infty t_n x^n)$.
If $y(x)=c_1 y_1(x) + c_ 2 y_2(x)$ is a formal power series then
we must have $ac_2=0$ and $y(x)=c_1x^{r_1}(1+ \sum\limits_{n=1}^\infty s_n x^n) +
c_2 x^{r_2}(1+\sum\limits_{n=1}^\infty t_n x^n)$ which is convergent. This shows that $\hat y_1, \hat y_2$ must
be convergent.

We are now in conditions of giving a second proof to Theorem~\ref{Theorem:twoformalsolutionsconverge}.

\begin{proof}[Alternative proof of Theorem~\ref{Theorem:twoformalsolutionsconverge}]
Indeed, from the second part of the proof (which is based only on classical methods of Frobenius and ODEs)
we know that the origin is an ordinary point or a regular singular point of the ODE.
Given the two linearly independent formal solutions $\hat y_j(x),\;j=1,2$,  from the above discussion,  the solutions $\hat y_1(x),\hat y_2(x)$ are analytic.
\end{proof}

\subsubsection{The wronskian III: some examples}
The next couple of examples show that the information on the wronskian (whether it is convergent, formal,etc) is not enough to infer about the nature of the solutions.
\begin{Example}[convergent wronskian but no formal solution]{\rm
This is an example of an ODE with a convergent wronskian but admitting no formal solution.
\begin{equation}\label{wronskconv}
x^3y^{\prime\prime}-x^2y^\prime-y=0.
\end{equation}
The origin is a non-regular singular point for (\ref{wronskconv}).
From what we have observed above the wronskian  $W$ of two linearly independent
solutions of (\ref{wronskconv}) satisfies the following first order ODE

$$x^3w^\prime-x^2w=0$$ whose solution is
of the form
$$W(x)=K\exp\big(\int^x \frac{\eta^2}{\eta^3}d\eta\big)=K\exp(\log x)=Kx$$ for some constant  $K$.

Let us now check that there are no formal solutions besides the trivial. Indeed, assume that
$y(x)=\displaystyle\sum^\infty_{n=0}a_nx^n$ is a formal solution of (\ref{wronskconv}). Then we have
$$x^3\big(\sum^{\infty}_{n=2}n(n-1)a_nx^{n-2}\big)-x^2\big(\sum^\infty_{n=1}na_nx^{n-1}\big)-\sum^\infty_{n=0}a_nx^n=0$$
$$\sum^{\infty}_{n=2}\big([n(n-1)-n]a_n-a_{n+1}\big)x^{n+1}-a_1x^{2}-a_2x^2-a_1x-a_0=0$$
so that  $a_0=a_1=a_2=0$ and $(n^2-2n)a_n=a_{n+1}$ for every
$n=2,3,\ldots$. Thus $a_n=0$ for every $n=0,1,2,\ldots$.
}
\end{Example}

\begin{Example}[non-convergent wronskian no formal solution]
{\rm We shall now give an example of an ODE with non-convergent wronskian and
admitting no formal solution but the trivial one.
The ODE
\begin{equation}\label{wronsknoconv}
x^3y^{\prime\prime}-xy^\prime-y=0.
\end{equation}
has a non-regular singular point at the origin.
Indeed, this wronskian is solution of the first order ODE

$$x^3w^\prime-xw=0$$ which has solutions of the form
$$W(x)=K\exp\big(\int^x \frac{\eta}{\eta^3}d\eta\big)=K\exp\big(-\frac{1}{x}\big)$$ where  $K$ is a constant.

Let us now check that
 (\ref{wronsknoconv}) admits no non-trivial formal solutions.
 Assume that $y(x)=\sum^\infty_{n=0}a_nx^n$ is a formal solution of  (\ref{wronsknoconv}). Then we must have

$$x^3\big(\sum^{\infty}_{n=2}n(n-1)a_nx^{n-2}\big)
-x\big(\sum^\infty_{n=1}na_nx^{n-1}\big)-\sum^\infty_{n=0}a_nx^n=0$$

$$\sum^{\infty}_{n=3}\big((n-1)(n-2)a_{n-1}-(n+1)a_{n}\big)x^{n}-2a_2x^{2}-a_2x^2-a_1x-a_1x-a_0=0$$

and then $a_0=a_1=a_2=0$ and $(n-1)(n-2)a_{n-1}=(n+1)a_{n}$ for all  $n=3,4,\ldots$. Hence
 $a_n=0$ for all $n=0,1,2,\ldots$.
}

\end{Example}

\subsection{Characterization of regular singular points: proof of Theorem~\ref{Theorem:characterizationregularpointordertwo}}

\vglue.1in

We shall now prove Theorem~\ref{Theorem:characterizationregularpointordertwo}.
\begin{proof}[Proof of Theorem~\ref{Theorem:characterizationregularpointordertwo}]
We shall first consider the complex analytic case.
We start then with a complex  analytic ODE of the form $a(z) u^{\prime \prime} + b(z) u ^\prime + c(z) u=0$.
Let us assume that this equation   admits
two linearly independent solutions $u_1(z), u_2(z)$, which are of autlos type in   some  neighborhood of the origin $z=0\in \mathbb C$.

The wronskian $W(u_1,u_2)(z)$ satisfies the first order ODE
$a(z)w^\prime+b(z)w=0$ and since it is given by $W(u_1,u_2)(z)=u_1^\prime(z)u_2(z)-u_1(z)u_2^\prime(z)$, it is also of autlos type in some neighborhood of the origin $z=0\in \mathbb C$.  Using then the above first order ODE and arguments similar to those in the proof of Lemma~\ref{Lemma:simplepole} we conclude that $b(z)/a(z)$ must have a pole of order $\leq 1$ at the origin, otherwise $W(u_1,u_2)(z)$ would have an essential singularity at the origin.
Following now a similar reasoning in the proof of Claim~\ref{Claim:poleordertwo} in the second part of the proof of
Theorem~\ref{Theorem:twoformalsolutionsconverge} we conclude that
$c(z)/a(z)$ must have a pole of order $\leq 2$ at the origin.
This shows that the singularity at the origin is regular, or the origin is an ordinary point.
If we start with a real analytic ODE then we consider its complexification. The fact that there are two linearly independent solutions of autlos type for the original ODE implies that there are two linearly independent solutions for the corresponding complex  ODE, by definition these solutions will be of autlos type. Once we have concluded that the complex ODE has a regular singularity or an ordinary point at the origin, the same holds for the original real analytic ODE.
Thus $(2)\implies(3)$.
The classical Frobenius theorem shows that $(3)\implies(1)$. Finally, it is clear from the definitions that
$(1)\implies(2)$.
\end{proof}

The next examples show how sharp is the statement of Theorem~\ref{Theorem:characterizationregularpointordertwo}.

\begin{Example}{\rm Consider the equation
\begin{equation}\label{edo4}
z^3u^{\prime\prime}-zu^\prime+u=0.
\end{equation}
The origin $z_0=0$ is a singular point, but not is regular singular point, since if we multiply equation (\ref{edo4}) according to $1/z$ we have

$$z^2u^{\prime\prime}-u^\prime+\frac{1}{z}u=0$$

and since the coefficient $-1$ of $u^\prime$ does not have the form $zb(z)$, where $b$ is holomorphic for $0$.
It is easy to see that $u_1=z$ is a solution of  equation (\ref{edo4}). Making use of the method of reduction of order we can construct a second solution $u_2$ linearly independent with $u_1$. Hence we have that
$$u_2(z)=z\int^z\big(\frac{\exp\big(-\int^w \frac{-v}{v^3}dv\big)}{w^2}\big)dw=z\int^z\big(\frac{\exp\big(\int^w \frac{1}{v^2}dv\big)}{w^2}\big)dw $$
$$u_2(z)=z\int^z\big(\frac{\exp\big(-\frac{1}{w}\big)}{w^2}\big)dw =z\exp\big(-\frac{1}{z}\big).$$
Note that $u_2$ is not holomorphic.}
\end{Example}

\begin{Remark}{\rm Consider a second order differential equation of the form
\begin{equation}\label{charquasecase1}
z^3a(z)u^{\prime\prime}+z^2b(z)u^\prime+c(z)u=0
\end{equation}
where $a,b,c$ are holomorphic at the origin with $a(0)\neq0$ and $c(0)\neq0$. We shall see that (\ref{charquasecase1}) admits no formal solution. Indeed we assume that
$u(z)=\displaystyle\sum^\infty_{n=0}d_nz^n$ is a formal solution of (\ref{charquasecase1}) hence
$$z^3\big(\sum^\infty_{n=0}a_nz^n\big)\big(\sum^{\infty}_{n=2}n(n-1)d_nz^{n-2}\big)+z^2\big(\sum^\infty_{n=0}b_nz^n\big)\big(\sum^\infty_{n=1}nd_nz^{n-1}\big)+\big(\sum^\infty_{n=0}c_nz^n\big)\big(\sum^\infty_{n=0}d_nz^n\big)=0$$
$$\big(\sum^\infty_{n=0}a_nz^n\big)\big(\sum^{\infty}_{n=3}(n-1)(n-2)d_{n-1}z^{n}\big)+\big(\sum^\infty_{n=0}b_nz^n\big)\big(\sum^\infty_{n=2}(n-1)d_{n-1}z^{n}\big)+\big(\sum^\infty_{n=0}c_nz^n\big)\big(\sum^\infty_{n=0}d_nz^n\big)=0$$
$$\sum^\infty_{n=0}\tilde{a}_nz^n+\sum^\infty_{n=0}\tilde{b}_nz^n+\sum^\infty_{n=0}\tilde{c}_nz^n=0$$
where $$\tilde{a}_0=\tilde{a}_1=\tilde{a}_2=0,\;\tilde{a}_n=\sum^{n}_{k=3}a_{n-k}(k-1)(k-2)d_{k-1}\;\;\mbox{ for }n\geq 3,$$
 $$\tilde{b}_0=\tilde{b}_1=0,\;\tilde{b}_n=\sum^{n}_{k=2}b_{n-k}(k-1)d_{k-1}\mbox{ for }n\geq 2$$ and $$\tilde{c}_n=\sum^{n}_{k=0}c_{n-k}d_k\;\;\mbox{ for }n\geq 0.$$Hence we have
 $$\tilde{a}_n+\tilde{b}_n+\tilde{c}_n=0,\;\;\mbox{ for all }n\geq0.$$
For $n=0$ we have: $c_0d_0=0$ and since $c_0\neq0$ then $d_0=0$.

For $n=1$ we have: $c_1d_0+c_0d_1=0$ and then  $c_0d_1=0$ and since $c_0\neq0$ then $d_1=0$.

For $n=2$ we have: $b_0d_1+c_2d_0+c_1d_1+c_0d_2=0$ and then  $c_0d_2=0$ and since $c_0\neq0$ then $d_2=0$.

For $n=3$ we have: $2a_0d_2+b_1d_1+2b_0d_2+c_3d_0+c_2d_1+c_1d_2+c_0d_3=0$ and then  $c_0d_3=0$ and since $c_0\neq0$ then $d_3=0$.

For $n\geq4$ we have:
$$c_0d_n=-\sum^n_{k=4}[(k-1)(k-2)a_{n-k}+(k-1)b_{n-k}+c_{n-k+1}]d_{k-1}$$and since $c_0\neq0$ then we have $d_n=0$ for all $n\geq0$. Thence there exist no  non trivial formal solution. Observe that there exists
 the limit
$$\displaystyle\lim_{x\to 0}\frac{x^3b(x)}{x^3a(x)}=\displaystyle\lim_{x\to 0}\frac{b_0+b_1x+\ldots}{a_0+a_1x+\ldots}=\frac{b_0}{a_0}$$e it does not exist the limit
$$\displaystyle\lim_{x\to 0}\frac{x^2c(x)}{x^3a(x)}=\displaystyle\lim_{x\to 0}\frac{c_0+c_1x+\ldots}{a_0x+a_1x^2+\ldots}.$$

}
\end{Remark}

\begin{Remark}{\rm Consider a second order differential equation of the form
\begin{equation}\label{charquasecase2}
z^2a(z)u^{\prime\prime}+b(z)u^\prime+c(z)u=0
\end{equation}
where $a,b,c$ are holomorphic at the origin with $a(0)\neq0$, $b(0)\neq0$ and $c(0)\neq0$. We shall see that (\ref{charquasecase2}) always admits non trivial formal solution. Indeed we assume that
$u(z)=\displaystyle\sum^\infty_{n=0}d_nz^n$ is a formal solution of (\ref{charquasecase2}) hence
$$z^2\big(\sum^\infty_{n=0}a_nz^n\big)\big(\sum^{\infty}_{n=2}n(n-1)d_nz^{n-2}\big)+\big(\sum^\infty_{n=0}b_nz^n\big)\big(\sum^\infty_{n=1}nd_nz^{n-1}\big)+\big(\sum^\infty_{n=0}c_nz^n\big)\big(\sum^\infty_{n=0}d_nz^n\big)=0$$
$$\big(\sum^\infty_{n=0}a_nz^n\big)\big(\sum^{\infty}_{n=2}n(n-1)d_nz^n\big)+\big(\sum^\infty_{n=0}b_nz^n\big)\big(\sum^\infty_{n=0}(n+1)d_{n+1}z^{n}\big)+\big(\sum^\infty_{n=0}c_nz^n\big)\big(\sum^\infty_{n=0}d_nz^n\big)=0$$
$$\sum^\infty_{n=0}\tilde{a}_nz^n+\sum^\infty_{n=0}\tilde{b}_nz^n+\sum^\infty_{n=0}\tilde{c}_nz^n=0$$
where $$\tilde{a}_0=\tilde{a}_1=0,\;\tilde{a}_n=\sum^{n}_{k=2}a_{n-k}k(k-1)d_k\;\;\mbox{ for }n\geq 2,$$
 $$\tilde{b}_n=\sum^{n}_{k=0}b_{n-k}(k+1)d_{k+1}\mbox{ for }n\geq 0$$ and $$\tilde{c}_n=\sum^{n}_{k=0}c_{n-k}d_k\;\;\mbox{ for }n\geq 0.$$Hence we have
 $$\tilde{a}_n+\tilde{b}_n+\tilde{c}_n=0,\;\;\mbox{ for all }n\geq0.$$
For $n=0$ we have: $b_0d_1+c_0d_0=0$ and since $b_0\neq0$ then $d_1=-\frac{c_0d_0}{b_0}$.

For $n=1$ we have: $b_1d_1+2b_0d_2+c_1d_0+c_0d_1=0$ and then  $2b_0d_2=-(b_1+c_0)d_1-c_1d_0$ and since $b_0\neq0$ then $d_2=-\big(\frac{b_1+c_0}{2b_0}\big)\big(\frac{-c_0d_0}{b_0}\big)-\frac{c_1d_0}{2b_0}=\frac{(b_1c_0+c_0^2+c_1b_0)d_0}{b_0^2}$.

For $n\geq2$ we have:
$$b_0(n+1)d_{n+1}=-c_nd_0-\sum^{n}_{k=1}[k(k-1)a_{n-k}+kb_{n-k+1}+c_{n-k}]d_k$$and since $b_0\neq0$ we obtain
$$d_{n+1}=\frac{1}{b_0(n+1)}\big(-c_nd_0-\sum^{n}_{k=1}[k(k-1)a_{n-k}+kb_{n-k+1}+c_{n-k}]d_k\big).$$
Observe that the coefficients of the series depend on $d_0$, since we look for non trivial formal solutions it suffices to choose $d_0\neq0$. Hence, there exist non trivial formal solution. Also note that there exists  the limit
$$\displaystyle\lim_{x\to 0}\frac{x^2c(x)}{x^2a(x)}=\displaystyle\lim_{x\to 0}\frac{c_0+c_1x+\ldots}{a_0+a_1x+\ldots}=\frac{c_0}{a_0}.$$
 and the following limit is not finite
$$\displaystyle\lim_{x\to 0}\frac{xb(x)}{x^2a(x)}=\displaystyle\lim_{x\to 0}\frac{b_0+b_1x+\ldots}{a_0x+a_1x^2+\ldots}.$$
}
\end{Remark}
\begin{Example} {\rm Consider a second order differential equation given by
\begin{equation}\label{charquase3}
z^2u^{\prime\prime}+bu^\prime+cu=0
\end{equation}
where $b$ and $c$ are nonzero constants. Observe that the origin is a non regular singular point of (\ref{charquase3}). Next we shall see  that there exist  non trivial formal solutions for (\ref{charquase3}). Let us assume that
\begin{equation}\label{charquase4}
u(z)=\sum^\infty_{n=0}a_nz^n
\end{equation}
is a non trivial formal solution of (\ref{charquase3}). Hence we have

$$z^2\big(\sum^{\infty}_{n=2}n(n-1)a_nz^{n-2}\big)+b\big(\sum^\infty_{n=1}na_nz^{n-1}\big)+c\sum^\infty_{n=0}a_nz^n=0$$

$$(ca_0+ba_1)+(ca_1+2ba_2)z+\sum^{\infty}_{n=2}\big([n(n-1)+c]a_n+b(n+1)a_{n+1}\big)z^{n}=0$$
and then  $ca_0+ba_1=0$, $ca_1+2ba_2=0$ and $(n^2-n+c)a_n+b(n+1)a_{n+1}=0$ for all $n=2,3,\ldots$. Since $b\neq0$  we have $a_1=-\frac{ca_0}{b}$, $a_2=-\frac{ca_1}{2b}=\frac{c^2a_0}{2b^2}$ and
\begin{equation}\label{charquase5}
a_{n+1}=-\frac{(n^2-n+c)a_n}{b(n+1)},\;\;\mbox{ for all }n=2,3,\ldots.
\end{equation}
Observe that the coefficients of the series depend on $a_0$, since we look for non trivial formal solutions it suffices to choose $a_0\neq0$. Hence, there exist non trivial formal solution. Observe now that  this formal solution is not convergent.  Applying the ratio test to the expressions (\ref{charquase4}) and (\ref{charquase5}), we have that
$$ \big|\frac{a_{n+1}z^{n+1}}{a_nz^n}\big|=\big|\frac{n^2-n+c}{b(n+1)}\big|\cdot| z|\to\infty,$$ when $n\to\infty$, whenever $|z|\neq0$. Hence, the series converges only for $z=0$.}
\end{Example}

\subsection{Riccati  model for a  second order linear ODE}

We shall now exhibit method of associating to a homogeneous linear second order ODE a
Riccati differential equation. Consider a second order ODE given by
$$a(z)u^{\prime\prime}+b(z)u^\prime+c(z)u=0$$
where  $a,b,c$ are analytic functions, real or complex, of a variable $z$ real or complex, defined in a domain $U\subset \mathbb R, \mathbb C$.
According to \cite{Reis} there is an integrable
one-form
$$\Omega=-a(z)ydx + a(z)xdy +[a(z)y^2+b(z)xy +c(z)x^2]dz$$  that vanishes at
the vector field corresponding to the reduction of order of the ODE, i.e.,
$\omega (X)=0$ where
$$X(x,y,z)=y\frac{\partial }{\partial x}-
\big(\frac{b(z)}{a(z)}y+\frac{c(z)}{a(z)}x\big)\frac{\partial }{\partial y}+\frac{\partial}{\partial z}.$$
As a consequence the orbits of $X$ are tangent to the foliation $\fa_\Omega$ given by the Pfaff equation
 $\Omega=0$.

First of all we remark that we can write  $\Omega$ as follows

$$\frac{\Omega}{x^2}=a(z)d\big(\frac{y}{x}\big) +\big[a(z)\big(\frac{y}{x}\big)^2+b(z)\big(\frac{y}{x}\big)
+c(z)\big]dz.$$

Thus, by introducing the variable  $t=\frac{y}{x}$ we see that the same foliation $\fa_\Omega$ can be defined
by the one-form $\omega$ below:

$$\omega=a(z)dt+[a(z)t^2+b(z)t+c(z)]dz. $$

By its turn $\omega=0$ defines a Riccati foliation which writes as
$$\frac{dt}{dz}=-\frac{a(z)t^2+b(z)t+c(z)}{a(z)}.$$

\begin{Definition}
{\rm The Riccati differential equation above is called {\it Riccati model} of the ODE
\, \, $a(z)u^{\prime\prime}+b(z)u^\prime+c(z)u=0$. }

\end{Definition}

\begin{Remark}
{\rm The Riccati model can be obtained in a less geometrically clear way by setting
$t=u^\prime/u$ as a new variable. Sometimes it is also useful to consider the change of variable
$w=u/u^\prime$ which leads to the Riccati equation $\frac{dw}{dz}=\frac{c(z) w^2 + b(z) w + a(z)}{a(z)}$.

}
\end{Remark}
\subsubsection{Holonomy of a second order equation}

It is well-known that a complex rational Riccati differential equation $\frac{dy}{dx}=
\frac{a(x)y^2 + b(x)y + c(x)}{p(x)}$
induces in the complex surface $\mathbb P^1 \times \mathbb P^1$ a foliation $\fa$
with singularities, having the following
characteristics:
\begin{enumerate}
\item The foliation has a finite number of invariant vertical lines $\{x_0\} \times \mathbb P^1$.
These lines are given by the zeroes of $p(x)$ and possibly by the line $\{\infty\}\times \mathbb P^1$.

\item For each non-invariant vertical line $\{x_0\}\times \mathbb P^1$ the foliation has its leaves
transverse to this line.

\item From Ehresmann we conclude that the restriction of $\fa$ to $(\mathbb P^1\setminus \sigma )\times \mathbb P^1$,
where $\sigma\times \mathbb P^1$ is the set of invariant vertical lines, is a foliation transverse to the fibers of
the fiber space $\mathbb P^1\times \mathbb P^1 \to \mathbb P^1$ with fiber $\mathbb P^1$ and projection given by
$\pi(x,y)=x$.

\item The restriction $\pi\big|_{L}$ of the projection to each leaf $L$ of the Riccati foliation
defines a covering map $L\to \mathbb P^1 \setminus \sigma$.
\end{enumerate}

In particular, there is a global holonomy map which is defined as follows:

choose any point $x_0 \not \in \sigma$ as base point and consider the lifting of the closed paths
$\gamma \in \pi_1(\mathbb P^1 \setminus \sigma)$ to each leaf $L\in \fa$ by the restriction
$\pi\big|_{L}$ above. Denote the lift of $\gamma$ starting at the point $(x_0,z) \in \{x_0\} \times
\mathbb P^1$ by $\tilde \gamma_z$. If the end point of $\tilde \gamma_z$ is denoted by $(x_0,h_\gamma(z))$
then the map $z \mapsto h_\gamma(z)$ depends only on the homotopy class of $\gamma \in \pi_1(\mathbb P^1 \setminus \sigma)$.
Moreover, this defines a complex analytic diffeomorphism $h_{[\gamma]}\in \Diff(\mathbb P^1)$ and the map
$\pi_1(\mathbb P^1\setminus \sigma) \to \Diff(\mathbb P^1), \, [\gamma] \mapsto h_{[\gamma]}$ is a group homomorphism.
The image is called {\it global holonomy} of the Riccati equation. It is well known from the theory of foliations
transverse to fiber spaces that the global holonomy classifies the foliation up to fibered conjugacy (\cite{C-LN}).
This will be useful to us in what follows.

Let us start by observing that $\Diff(\mathbb P^1)$ as meant above is the projectivization of the special
linear group ie., $\Diff(\mathbb P^1) = \mathbb PSL(2,\mathbb C)$ meaning that every global holonomy
map can be represented by a Moebius map $T(z)= \frac{a_1 z + a_2}{a_3 z + a_4}$ where $a_1, a_2, a_3, a_4\in \mathbb  C$
and $a_1 a_4 - a_2 a_3=1$. Thus the global holonomy group of a Riccati foliation identifies with a group of Moebius maps.

\begin{Definition}[holonomy of a second order ODE]
{\rm Given a linear homogeneous second order ODE with complex polynomial coefficients
$$a(z)u^{\prime\prime}+b(z)u^\prime+c(z)u=0$$ we call the {\it holonomy}
of the ODE the global holonomy group of the corresponding  Riccati model.
}
\end{Definition}
\begin{Remark}
{\rm As we have seen above we can also obtain a Riccati model by any of the changes of variables
$t=u^\prime/u$ or $w=u/u^\prime$. From the viewpoint of ODEs these models may seem distinct. Nevertheless,
they differ only up to the change of coordinates $t=1/w$. Moreover, both have the same global holonomy group, since the point at infinity is always considered in the definition of global holonomy group. Indeed, the ideal space for considering a Riccati equation from the geometrical viewpoint, is the space $\mathbb C \times \mathbb C$.}
\end{Remark}

Next we see a concrete example of the global holonomy group of a second order ODE:

\begin{Example}{\rm Consider the  equation given by
\begin{equation}\label{Riccatihol1}
z^2u^{\prime\prime}+u=0.
\end{equation}
From what we observed above we have that for $a(z)=z^2$, $b(z)=0$ and $c(z)=1$ there exist a Riccati equation given by
\begin{equation}\label{Riccatihol2}
 \frac{dt}{dz}=\frac{z^2+t^2}{z^2}.
\end{equation}
Observe that equation (\ref{Riccatihol2}) is homogeneous therefore by the change of coordinates $w=\frac{t}{z}$ we have
$$z\frac{dw}{dz}+w=1+w^2$$
$$\frac{dw}{dz}=\frac{w^2-w+1}{z}$$
 the last equation  may be written  into separated variables  and therefore we have
$$\frac{dw}{w^2-w+1}=\frac{dz}{z}$$
$$\frac{i}{\sqrt{3}}\frac{dw}{w+\frac{1+i\sqrt{3}}{2}}-\frac{i}{\sqrt{3}}\frac{dw}{w+\frac{1-i\sqrt{3}}{2}}=\frac{dz}{z} $$
 $$d\big( \frac{i}{\sqrt{3}} \log \big(\frac{w+\frac{1+i\sqrt{3}}{2}}{w+\frac{1-i\sqrt{3}}{2}}\big)\big)=d(\log z)$$
 and then
 $$\big(\frac{w+\frac{1+i\sqrt{3}}{2}}{w+\frac{1-i\sqrt{3}}{2}}\big)^{i/\sqrt{3}}=Kz$$where $K$ is constant. Thence
  $$\frac{2t+(1+i\sqrt{3})z}{2t+(1-i\sqrt{3})z}=\tilde{K}z^{-i\sqrt{3}}$$where $\tilde{K}$ is constant.  Hence
 $$t=\frac{\tilde{K}(1-i\sqrt{3})z^{1-i\sqrt{3}}-(1+i\sqrt{3})z}{2-2\tilde{K}z^{-i\sqrt{3}}}.$$
 Let us now compute the global holonomy with basis $t=0$. For this sake we take a loop  $z(\theta)=z_0e^{i\theta}$ with $0\leq\theta\leq 2\pi$. For $\theta=0$ we have that
 $$\tilde{K}=z_0^{i\sqrt{3}}\frac{2t_0+(1+i\sqrt{3})z_0}{2t_0+(1-i\sqrt{3})z_0}$$and then  we obtain that
   $$h(t_0)=t(z(2\pi))=\frac{\tilde{K}(1-i\sqrt{3})z_0^{1-i\sqrt{3}}e^{2\pi\sqrt{3}}-
   (1+i\sqrt{3})z_0}{2-2\tilde{K}z_0^{-i\sqrt{3}}e^{2\pi\sqrt{3}}}$$ and replacing $K$ we get
$$h(t_0)=t(z(2\pi))=z_0\frac{(t_0(1-i\sqrt{3})+2z_0)e^{2\pi\sqrt{3}}-t_0(1+i\sqrt{3})-2z_0}{2t_0+(1-i\sqrt{3})z_0-\big(2t_0+(1+i\sqrt{3})z_0\big)e^{2\pi\sqrt{3}}}.$$
 }
\end{Example}

\subsubsection{Trivial Holonomy}

Let us investigate some interesting cases.
First  consider a Riccati foliation $\fa$
assuming that  $\sigma$ is a single point. Thus  we may assume that in affine coordinates $(x,y)$
the ramification point is the point $x=\infty, y=0$. Then we may write $\fa$ as given by
a polynomial differential equation $\frac{dy}{dx}=a(x)y^2 + b(x)y + c(x)$.
The global holonomy of $\fa$ is given by an homomorphism $\phi\colon \pi(\mathbb P^1\setminus \sigma)
\to \Diff(\mathbb P^1)$. Since $\sigma$ is a single point we have
$\mathbb P^1\setminus \sigma =\mathbb C$ is simply-connected and therefore the global  holonomy is trivial.
By the classification of foliations transverse to fibrations (\cite{C-LN} Chapter V) there is a
fibered biholomorphic map $\Phi\colon  \mathbb C \times \mathbb P^1 \to
\mathbb C \times \mathbb P^1$ that takes the foliation $\fa$ into the foliation $\mathcal H $ given by
the horizontal fibers $\mathbb C \times \{y\}, y \in \mathbb P^1$.

\begin{Lemma}
\label{Lemma:aut}{\rm
A holomorphic diffeomorphism $\Phi \colon \mathbb C \times \mathbb P^1 \to
\mathbb C \times \mathbb P^1$ preserving the vertical fibration  writes
in affine coordinates $(x,y)\in \mathbb C^2 \subset \mathbb C \times \mathbb P^1$
as $\Phi(x,y)=\big( Ax+B, \frac{a(x)y+ b(x)}{c(x)y + d(x)}\big)$ where $a,b,c,d $
are entire functions satisfying $ad-bc=1$, $0 \ne A,B \in \mathbb C$.}
\end{Lemma}
\begin{proof}[Proof of Lemma~\ref{Lemma:aut}]
 Picard's theorem and the fact that $\Phi$ preserves the fibration $x=const$ show
 that it is of the form $\Phi(x,y)=(f(x),g(x,y))$ where $f(x)=Ax+B$ is an affine map.
 Finally, for each fixed $x\in \mathbb C$ the map $\mathbb P^1 \ni y \mapsto g(x,y) \in \mathbb P^1$ is a
 diffeomorphism so it must write as $g(x,y)=\frac{a(x)y +b(x)}{c(x)y + d(x)}$ for some entire
 functions $a,b,c,d$ satisfying $ad - bc=1$.
\end{proof}

In particular we conclude that the leaves of  $\fa$ are diffeomorphic with $\mathbb C$
(including the one contained in the invariant fiber $\{(0,\infty)\}\times \mathbb P^1$,
and $\fa$ admits a holomorphic first integral $g\colon \mathbb C \times \mathbb P^1 \to \mathbb P^1$
of the above form $g(x,y)=\frac{a(x)y +b(x)}{c(x)y + d(x)}$.

Let us now apply this to our framework of second order linear ODEs.

\begin{proof}[Proof of Theorem~\ref{Theorem:Riccatitrivialholonomy}]
Beginning with the ODE $a(z)u^{\prime\prime}+b(z)u^\prime+c(z)u=0$ the Riccati model is
$$\frac{dt}{dz}=-\frac{a(z)t^2+b(z)t+c(z)}{a(z)}.$$
Thus if we assume that $a(z)=1$ then we have for this Riccati equation that $\sigma=\{\infty\}$ as considered above.
This implies that $\fa$ admits a holomorphic first integral $g\colon \mathbb C \times \mathbb P^1 \to \mathbb P^1$
of the above form $g(z,t)=\frac{A(z)t +B(z)}{C(z)t + D(z)}$. Given a leaf $L$ of the Riccati
 foliation there is a constant $\ell \in \mathbb P^1$  such that $g(z,t)=\ell$ for all $(t,z)\in L$. Hence
 $t=\frac{\ell D(z) - B(z)}{A(z) - \ell C(z)}$ for all $(t,z)\in L$. This defines a {\it meromorphic} parametrization
 $z\mapsto t(z)$ of the leaf. Since we have $t=\frac{y}{x}=\frac{u^\prime}{u}$
 therefore $u(z)=k \exp\big(\int_0^z t(\xi)d\xi\big)$
 is a solution of the ODE with $k \in \mathbb C$ a constant. This gives
 $$
u_{\ell,k}(z)=k \exp\big(\int_0^z \frac{\ell D(\xi) - B(\xi)}{A(\xi) - \ell C(\xi)}d\xi\big), \, k , \ell \in \mathbb C;
$$
 as general solution of the original ODE.
Notice that $\frac{u^\prime _{\ell,k}(z)}{u_{\ell,k}(z)}=\frac{\ell D(z) - B(z)}{A(z) - \ell C(z)}$ so that
if $\ell_1 \ne \ell_2$ then the corresponding solutions $u_{\ell_1,k_1}$ and $u_{\ell_2,k_2}$ generate a nonzero wronskian,
and therefore they are linearly independent solutions for all $k_1 \ne 0 \ne k_2$.
\end{proof}

Next we investigate the case where $\sigma$ consists of two points. In this case the holonomy group
of the ODE is cyclic generated by a single Moebius map. A first (regular singularity type) example is given below:

\begin{Example}[Bessel equation]\label{besselriccati}{\rm
Consider the complex Bessel equation given by
$$z^2u^{\prime\prime}+zu^\prime+(z^2-\nu^2)u=0$$ where  $z,\nu\in\mathbb{C}$.
Since  $a(z)=z^2$, $b(z)=z$ and $c(z)=z^2-\nu^2$ the corresponding Riccati model is
 $$\frac{dt}{dz}=-\frac{z^2t^2+zt+z^2-\nu^2}{z^2}$$
If we change coordinates to  $w=\frac{1}{z}$ then we obtain a Riccati equation
of the form
 $$\frac{dt}{dw}=\frac{t^2+tw+1-\nu^2w^2}{w^2}.$$
}
\end{Example}

A non-regular singularity example is given below:

\begin{Example}
{\rm
Let us consider the following polynomial ODE
$$
z^n  u^{\prime\prime} + b(z) u ^\prime + c(z) u =0.
$$
If $n \geq 2$ and $b(0)\ne 0$ or if $n \geq 3$ and $c(0)\ne 0$ or $b(0).b^\prime(0)\ne 0$ then
$z=0$ is a non-regular singular point. Let us assume that this is the case. The corresponding Riccati equation is
$$
\frac{dt}{dz}=-\frac{z^nt^2 + b(z)t + c(z)}{z^n}.
$$
Changing coordinates $w=1/z$ we obtain
$$
\frac{dt}{dw}=\frac{ t^2 + w^n b(1/w) t + w^n c(1/w)}{w^2}=\frac{w^kt^2 + \tilde b(w) t + \tilde c(w)}{w^{2 + k}}
$$
for some polynomials $\tilde b (w), \tilde c(w)$ and some $ k \in \mathbb N$.
This shows that the ramification set $\sigma\subset \mathbb P^1$ consists of the points $z=0$ and $z=\infty$.
The fundamental group of the basis $\mathbb P^1 \setminus \sigma$ is therefore cyclic generated by a single homotopy class.
The holonomy of the ODE is then generated by a single Moebius map.
}
\end{Example}

The following is an example with a holonomy group generated by two Moebius maps.

\begin{Example}[Legendre equation]\label{legendrericcati}{\rm
Consider the equation of Legendre given by
$$(1-z^2)u^{\prime\prime}-2zu^\prime+\alpha(\alpha+1)u=0$$where $\alpha\in\mathbb{C}$. From what we observed above we have that for $a(z)=1-z^2$, $b(z)=-2z$ and $c(z)=\alpha(\alpha+1)$ there exists a Riccati equation given by
 $$\frac{dt}{dz}=-\frac{(1-z^2)t^2-2zt+\alpha(\alpha+1)}{1-z^2}.$$
 Putting  $w=\frac{1}{z}$ we have
 $$\frac{dt}{dw}=\frac{(w^2-1)t^2-2wt+\alpha(\alpha+1)w^2}{w^2(w^2-1)}.$$
}
\end{Example}

\begin{Example}[an equation without solutions]
\label{Riccatisemsolformal}{\rm Let us consider the following
 Riccati equation
 $$\frac{du}{dz}=-\frac{zu^2+u+z}{z}$$
 which was obtained from Besssel equation  (Example~\ref{besselriccati})
 for $\nu=0$. Rewriting this equation we have
 \begin{equation}\label{edo9}
 zu^\prime(z)=-zu^2(z)-u(z)-z
\end{equation}

\begin{Claim}
Equation (\ref{edo9}) admits non-trivial formal solution.
\end{Claim}
\begin{proof}
Indeed, let us assume that $u(z)=\sum^\infty_{n=0}a_nz^n$ is a formal solution of (\ref{edo9}). Hence

  $$\sum^{\infty}_{n=1}na_nz^{n}=-z\big(\sum^\infty_{n=0}c_nz^n\big)-\sum^\infty_{n=0}a_nz^n-z$$
  where $c_n=\sum^{n}_{j=0}a_{n-j}a_j$. Thus we have
    $$\sum^{\infty}_{n=2}[(n+1)a_n+c_{n-1}]z^{n}+(2a_1+c_0+1)z+a_0=0$$
 then $a_0=0$, $2a_1+c_0+1=0$ and $(n+1)a_n+c_{n-1}$ for all $n=2,3,\ldots$.  Hence $a_n=0$ for all $n=0,1,2,\ldots$.
 \end{proof}

The corresponding ODE may be written as $zu^{\prime \prime} + u^\prime+zu=0$ and in the Euler form
it is $z^2u^{\prime \prime} +  zu ^\prime + z^2u=0$. This last has indicial equation $r(r-1) +r=0$ which
gives as only solution $r=0$. Then Frobenius theorem assures the existence of a solution of the form
$u(z)=z^0 \sum\limits_{n=0}^\infty a_n z^n$.
}
\end{Example}



\begin{Remark}{\rm Consider a second order ordinary differential equation given by
\begin{equation}\label{edo10}
a(z)u^{\prime\prime}+b(z)u^\prime+c(z)u=0
\end{equation}
where $a,b,c$  are polynomials. We shall see that (\ref{edo10}) is invariant under Moebius transformations. Indeed,
by the change of coordinates $z=\frac{\alpha w+\beta}{\gamma w+\delta}$ with $\alpha \delta-\beta\gamma=1$ and considering $\tilde{\varphi}(w)=\varphi\big(\frac{\alpha w+\beta}{\gamma w+\alpha}\big)$ where $\varphi$ is a solution of (\ref{edo10}). Taking derivative we have
$$\varphi^\prime\big(\frac{\alpha w+\beta}{\gamma w+\alpha}\big)=(\gamma w+\delta)^2\tilde{\varphi}^\prime(w)$$
$$\varphi^{\prime\prime}\big(\frac{\alpha w+\beta}{\gamma w+\alpha}\big)=(\gamma w+\delta)^4\tilde{\varphi}^{\prime\prime}(w)+2\gamma(\gamma w+\delta)^3\tilde{\varphi}^\prime(w).$$
Given that, by (\ref{edo10})
\begin{equation}\label{edo11}
a\big(\frac{\alpha w+\beta}{\gamma w+\delta}\big)\varphi^{\prime\prime}\big(\frac{\alpha w+\beta}{\gamma w+\delta}\big)+b\big(\frac{\alpha w+\beta}{\gamma w+\delta}\big)\varphi^\prime\big(\frac{\alpha w+\beta}{\gamma w+\delta}\big)+c\big(\frac{\alpha w+\beta}{\gamma w+\delta}\big)\varphi\big(\frac{\alpha w+\beta}{\gamma w+\delta}\big)=0.
\end{equation}
 Given that $a,b,c$  are polynomials we have
$$a\big(\frac{\alpha w+\beta}{\gamma w+\delta}\big)=\frac{\tilde{a}(w)}{(\gamma w+\delta)^n},\;\;b\big(\frac{\alpha w+\beta}{\gamma w+\delta}\big)=\frac{\tilde{b}(w)}{(\gamma w+\delta)^m},\;\;c\big(\frac{\alpha w+\beta}{\gamma w+\delta}\big)=\frac{\tilde{c}(w)}{(\gamma w+\delta)^p}$$
where $\tilde{a},\tilde{b},\tilde{c}$  are polynomials and $n,m,p\in\mathbb{N}$.

Hence back to equation (\ref{edo11}) we obtain
$$(\gamma w+\delta)^{m+p+4}\tilde{a}(w)\tilde{\varphi}^{\prime\prime}(w)+[2\gamma(\gamma w+\delta)^{m+p+3} \tilde{a}(w)+(\gamma w+\delta)^{n+p+2}\tilde{b}(w)]\tilde{\varphi}^{\prime}(w)+(\gamma w+\delta)^{m+n}\tilde{c}(w)\tilde{\varphi}(w)=0 $$
Hence $\tilde{\varphi}$ satisfies equation:
\begin{equation}\label{edo12}
\hat{a}(w)u^{\prime\prime}(w)+\hat{b}(w)u^{\prime}(w)+\hat{c}(w)u(w)=0.
\end{equation}
where $$\hat{a}(w)=(\gamma w+\delta)^{m+p+4}\tilde{a}(w),\;\;\;\tilde{\tilde{b}}(w)=2\gamma(\gamma w+\delta)^{m+p+3} \tilde{a}(w)+(\gamma w+\delta)^{n+p+2}\tilde{b}(w)$$ and $$\tilde{\tilde{c}}(w)=(\gamma w+\delta)^{m+n}\tilde{c}(w)$$ are polynomials. Hence equation (\ref{edo11}) is transformed by a Moebius map into equation  (\ref{edo12}).}
\end{Remark}

\begin{Example}{\rm Consider the equation given by
\begin{equation}\label{edo13}
u^{\prime\prime}-zu^\prime-u=0.
\end{equation}
From what we have observed above we known that for $a(z)=1$, $b(z)=-z$ and $c(z)=-1$ there exists a
Riccati equation given by
\begin{equation}\label{edo14}
 \frac{dt}{dz}=1+zt-t^2.
\end{equation}
It is not difficult to see that  $t=z$ is a solution of the differential Riccati model. We assume that  $t=s+z$ is a solution of (\ref{edo14}) and then  we obtain a Bernoulli equation
 \begin{equation}\label{edo15}
 \frac{ds}{dz}=-sz-s^2
 \end{equation}
 by the change of coordinates $v=s^{-1}$ in equation (\ref{edo15}) we have equation
$$ \frac{dv}{dz}=vz+1. $$
  Thus $v$ is of the form
  $$v=\exp\big(\frac{z^2}{2}\big)\big(A+\int^z \exp\big(-\frac{\eta^2}{2}\big)d\eta\big)$$
where $A$ is constant. Hence
$$ s=\frac{\exp\big(-\frac{z^2}{2}\big)}{A+\int^z \exp\big(-\frac{\eta^2}{2}\big)d\eta}$$ and consequently

$$t=\frac{\exp\big(-\frac{z^2}{2}\big)}{A+\int^z \exp\big(-\frac{\eta^2}{2}\big)d\eta}+z.$$
In the construction of the Riccati equation associate to (\ref{edo13}) it is considered that
$t=\frac{u^\prime}{u}$ where
$u$ is a solution of (\ref{edo13}). Hence we have
$$\frac{u^\prime}{u}=\frac{\exp\big(-\frac{z^2}{2}\big)}{A+\int^z \exp\big(-\frac{\eta^2}{2}\big)d\eta}+z$$
that may be written  as

$$(\log u)^\prime =\big(\log\big(A+\int^z \exp\big(-\frac{\eta^2}{2}\big)d\eta \big)+\frac{z^2}{2}  \big)^\prime.$$

Hence
$$u(z)=K\exp\big(\frac{z^2}{2}\big) \big(A+\int^z \exp\big(-\frac{\eta^2}{2}\big)d\eta \big)$$where $K$ is constant. It is a straightforward computation to show that $u$ is a solution of (\ref{edo13}).
}
\end{Example}

\subsection{Examples and counterexamples}

We start with an example.
\begin{Example}\label{2exeminf} {\rm Consider the equation
\begin{equation}\label{1eqs1}
x^2y^{\prime\prime}-y^\prime-\frac{1}{2}y=0.
\end{equation}
The origin $x_0=0$ is a singular point, but not is regular singular point, since the coefficient -1 of $y^\prime$ does not have the form $xb(x)$, where $b$ is analytic for $0$. Nevertheless, we can formally solve this equation by power series $
\sum^{\infty}_{k=0} a_kx^k$,
where the coefficients $a_k$ satisfy the following recurrence formula
$(k+1)a_{k+1}=\big[k^2-k-\frac{1}{2}\big]a_k,\;\;\;\mbox{ for every }k=0,1,2,\ldots.$
If $a_0\neq0$, applying the quotient test to this expression we have that
$$ \big|\frac{a_{k+1}x^{k+1}}{a_kx^k}\big|=\big|\frac{k^2-k-\frac{1}{2}}{k+1}\big|\cdot|x|\to\infty,$$when $k\to\infty$, provided that $|x|\neq0$. Hence, the series converges  only for $x=0$, and therefore does not represent a function in a neighborhood of $x=0$.}
\end{Example}

\subsection{Liouvillian solutions}

In this section we shall refer to the notion of Liouvillian function as introduced in \cite{Singer}.
We stress the fact that the generating basis field is the one of rational functions. Thus
a Liouvillian function of $n$ complex variables $x_1,\ldots,x_n$ will be a function belonging to
a Liouvillian tower of differential extensions $k_0\subset k_1\subset \cdots\subset k_r$ starting
the field $k_0$ of rational functions $k_0=\mathbb C(x_1,\ldots,x_n)$ equipped with the
partial derivatives $\frac{\partial}{\partial  x_j}$.

Recall that a Liouvillian function is always holomorphic in some Zariski open subset of the
space $\mathbb C^n$. Nevertheless, it may have several branches. Let us denote by
$Dom(F)\subset \mathbb C^n$ the {\it domain of $F$} as the biggest open subset where
$F$ has local holomorphic branches. This allows the following definition:
\begin{Definition}[Liouvillian solution, Liouvillian first integral, Liouvillian relation]
{\rm Given an equation $a(z) u^{\prime \prime}  + b(z) u^\prime + c(z) u =0$, a
Liouvillian function $u(z)$ of the variable $z$ will be called {\it a solution}
of the ODE if we have $a(z) u^{\prime \prime} + b(z) u^\prime + c(z) u=0$
in some nonempty open subset where $u(z)$ is holomorphic.
A three variables Liouvillian function $F(x_1,x_2,x_3)$ will be called {\it a first integral}
of the ODE if given any local solution $u_0(z)$ of the ODE, defined for $z$ in a disc $D(z_0,r)\subset \mathbb C$,  we have that $F(z,u_0(z),u_0^\prime(z))$ is constant for $|z-z_0|<r$ provided that $(z,u_0(z),u_0^\prime(z))\subset Dom(F),$ for all $z \in D(z_0,r)$. Similarly we shall say that a solution
$u_0(z)$ of the ODE, defined for $z \in Dom(u_0)\subset \mathbb C$, {\it satisfies a Liouvillian relation}
if there is a Liouvillian function $F(x_1,x_2,x_3)$ such that $\{(z,u(z),u^\prime(z))\in \mathbb C^3, z \in Dom(u)\}\cap Dom(F) \ne \emptyset$ and $F(z,u_0(z),u_0^\prime(z))=0$ in some dense open subset of $Dom(u_0)$.
}
\end{Definition}

Let us recall a couple of classical results:

\begin{Theorem}[Singer, \cite{Singer}]
\label{Theorem:singer}
Assume that the polynomial first order ODE $\frac{dx}{dz}=P(x,y), \, \frac{dy}{dz}=Q(x,y)$
admits a Liouvillian first integral. Then there are rational functions $U(x,y), \, V(x,y)$
such that $\frac{\partial U}{\partial y}=\frac{\partial V}{\partial x}$ and the differential
form $Q(x,y)dx - P(x,y)dy$ admits the integrating factor $R(x,y)=\exp\big[\int_{(x_0,y_0)}^{(x,y)}
U(x,y) dx + V(x,y)dy\big]$.

\end{Theorem}

\begin{Theorem}[Rosenlitch, Singer]
Let $p(z), q(z)$ be Liouvillian functions and
$L(y)=y^{\prime \prime} + p(z) y ^\prime + q(z)y$. If $L(y)=0$ has
a Liouvillian first integral then all solutions are Liouvillian.
If $L(y)=0$ has a nontrivial Liouvillian solution, then this equation
has  a Liouvillian first integral.

\end{Theorem}

\begin{Example}[Bernoulli ODEs]
{\rm
Recall that a {\it Bernoulli differential equation of power $1$} is one of the form
$\frac{dy}{dx}=\frac{a_1(x)y +a_2(x) y^2}{p(x)}$. If we perform a change of variables as
$(x,y) \mapsto (x,y^{k})$ then we obtain an equation of the form
$\frac{dy}{dx}= \frac{ y^{k+1} a(x) + y b(x)}{p(x)}$ which will be called a Bernoulli equation
of power $k$.

We prove the
 existence of a first integral for $\Omega=0$  of Liouvillian type. First
we observe that $\Omega=0$ can be given by
\[
(k-1) \frac{\Omega}{p(x)y^k} = (k-1) \frac{dy}{y^k} -
(k-1) \big(\frac{a(x)}{p(x)} - \frac{b(x)}{p(x)y^{k-1}}\big)dx = 0.
\]
 Let now
$f(x)$ be such that $\frac{f^\prime(x)}{f(x)} = (k-1)
\frac{b(x)}{p(x)}$ and let $g(x)$ be such that $g^\prime(x) =
-\frac{a(x)}{p(x)f(x)}\cdot(k-1).$ Then $\Omega=0$ can be given by
\[
(k-1)
\frac{dy}{y^k} - (k-1) \frac{a(x)}{p(x)}dx +
\frac{f^\prime(x)}{y^{k- 1}f(x)}\, dx = 0.
\]
 Therefore $F(x,y) = g(x)
- \frac{1}{f(x) y^{k-1}}$ defines a first integral for $\Omega=0$  which
is clearly  of  Liouvillian type.
}
\end{Example}

Before proving Theorem~\ref{Theorem:characterizationliouville}  we shall need a lemma:

\begin{Lemma}
Let $\frac{dy}{dx}=\frac{c(x) y^2 + b(x) y + a(x)}{a(x)}$ be a rational Riccati ODE,
where $a(x), b(x), c(x)$ are complex polynomials. Assume that there is a Liouvillian
first integral. Then we have the following possibilities:
\begin{enumerate}
\item The equation is linear of the form $a(x)y^\prime - b(x)y=a(x)$.
\item Up to a rational change of coordinates of the form $Y= y-A(x)/B(x)$, the equation is a Bernoulli equation $\frac{dY}{dx}= \frac{\tilde c(x) Y^2 + \tilde b(x) Y}{\tilde a(x)}$.
\end{enumerate}

\end{Lemma}
\begin{proof}
Let $\Omega=(c(x)y^2 + b(x)y + a(x))dx - a(x)dy$. The ODE is equivalent to $\Omega=0$.
According to Singer \cite{Singer} (Theorem~\ref{Theorem:singer} above) there is a rational 1-form
$\eta=U(x,y) dx + V(x,y)dy$ such that $d \eta=(\frac{\partial U}{\partial y}- \frac{\partial V}{\partial x})dy \wedge dx=0$ and $\exp(\int \eta)$ is an integrating factor for $\Omega$.  This means that $d (\Omega/\exp(\int\eta))=0$ and therefore $d \Omega= \eta\wedge \Omega$.

\noindent{\bf Case 1}.  $\Omega=0$ admits some invariant algebraic curve which is not a vertical line $x=c\in \mathbb C$.

In this case we may choose an irreducible polynomial $f(x,y)$ such that $f(x,y)=0$ describes this
non-vertical algebraic solution. Now we observe that  the leaves of the Riccati foliation defined by $\Omega=0$ on $\mathbb P^1 \times \mathbb P^1$ are, except for those contained in the invariant vertical fibers, all transverse to the vertical fibers $\{x\} \times \mathbb P^1 \subset \mathbb P^1 \times \mathbb P^1$. Thus we conclude that $\frac{\partial f}{\partial y}(x,y)$ never vanishes for each $x$ such that $a(x) \ne 0$.
Since $f(x,y)$ is polynomial, this implies that  $f(x,y)=A(x) - B(x)y$ for some polynomials $A(x), B(x)$:
Look at the function $f_y=\frac{\partial f}{\partial y}(x,y)$. This function is constant along almost all the fibers of $x\colon \mathbb C^2 \to \mathbb C$. Therefore we must have that $d f_y \wedge dx=0$ almost everywhere. Thus $df_y \wedge dx=0$ everywhere and this gives $f_{yy}=\frac{\partial ^2 f}{\partial y^2}=0$.
Hence $f(x,y)=A(x)- B(x)y$ by standard integration.
This shows that the non-vertical solution is a graph of the form $y(x)=\frac{A(x)}{B(x)}$.

In this case we may perform a change of variables as follows:
write $Y=y - y(x)$ to obtain:
\[
\frac{dY}{dx}=\frac{c(x)Y^2 + (2c(x) y(x) + b(x))Y}{a(x)}=\frac{B(x)c(x)Y^2 +(b(x)B(x)+2 c(x)A(x))Y}{a(x)B(x)}.
\]

This is a Bernoulli equation.

\noindent{\bf Case 2}. If there is no invariant curve other the vertical lines.

Denote by $(\eta)_\infty$ the polar divisor of $\eta$ in $\mathbb C^2$.
\begin{Claim}
We have $(\eta)_\infty=\{x \in \mathbb C, a(x)=0\} \times \mathbb C$.
\end{Claim}
\begin{proof}
First of all we recall that the polar set of $\eta$ is invariant by $\Omega=0$ (\cite{Camacho-Scardua}).
By the hypothesis we then conclude that   $(\eta)_\infty\supseteq \{x \in \mathbb C, a(x)=0\} \times \mathbb C$ is a union of vertical invariant lines.
Thus, from the integration lemma in \cite{C-LN-S1} we have
\[
\eta=\sum\limits_{j=1}^r \lambda_j \frac{dx}{x-x_j} + d \bigg(\frac{g(x,y)}{\prod\limits_{j=1}^r (x-x_j)^{n_j -1}}\bigg)
\]
where $g(x,y)$ is a polynomial function, $\lambda_j\in \mathbb C$ and $n_j\in \mathbb N$ is the order of
the poles of $\Omega$ in the component $(x=x_j)$ of the polar set.
Now we use the equation $d\Omega= \eta \wedge \Omega$ to obtain
\[
-[a^\prime(x) + 2y c(x) + b(x)] dx \wedge dy =
-a(x) \sum\limits_{j=1}^r \frac{\lambda_j}{x-x_j} dx \wedge dy + d \bigg(\frac{g(x,y)}{\prod\limits_{j=1}^r (x-x_j)^{n_j -1}}\bigg) \wedge \Omega
\]
where
\[
d \bigg(\frac{g(x,y)}{\prod\limits_{j=1}^r (x-x_j)^{n_j -1}}\bigg) \wedge \Omega= g d\bigg(\frac{1}{\prod\limits_{j=1}^r (x-x_j)^{n_j -1}}\bigg)\wedge a(x) dy +
\frac{dg}{\prod\limits_{j=1}^r (x-x_j)^{n_j -1}} \wedge \Omega.
\]

Notice that
\[
\frac{dg}{\prod\limits_{j=1}^r (x-x_j)^{n_j -1}} \wedge \Omega=\frac{g_x}{\prod\limits_{j=1}^r (x-x_j)^{n_j -1}}a(x) dx \wedge dy - \frac{g_y}{\prod\limits_{j=1}^r (x-x_j)^{n_j -1}} (c(x) y^2 + b(x) y + a(x))dx\wedge dy
\]

Notice that the left side is $-[a^\prime(x) + 2y c(x) + b(x)]$ has no term in $y^2$ so that we must have
$g_y c(x)=0$ in the right side. If $g_y=0$ then $g=g(x)$ and from the left side  we must have
$c(x)=0$. This shows that we must always have $c(x)=0$. This implies that the original equation is
a linear homogeneous  equation of the form
\[
\frac{dy}{dx}=\frac{ b(x)y + a(x)}{a(x)}
\]
which can be written as $a(x) y^\prime  - b(x)y  = a(x)$.

\end{proof}

\end{proof}

\begin{proof}[Proof of Theorem~\ref{Theorem:characterizationliouville}]
Let us prove the second part, ie., the equivalence. We assume that  $L[u](z)=a(z)u^{\prime\prime}+b(z)u^\prime+c(z)u=0$ admits a Liouvillian first integral.
We consider the change of coordinates $t=\frac{u^\prime}{u}$ which gives the following
 Riccati model $\mathcal R: \frac{dt}{dz}=-\frac{a(z) t^2 + b(z) t + c(z)}{a(z)}$.
 We claim:
 \begin{Claim}
 The Riccati model $\mathcal R$ also  admits a
Liouvillian first integral.
\end{Claim}

\begin{proof}
By hypothesis the ODE $L[u]= a(z) u^{\prime \prime} + b(z) u ^\prime + c(z) u=0$ has
a Liouvillian first integral. By the Corollary in \cite{Singer} page 674 all solutions of
$L[u]=0$ are Liouvillian. This implies, by Theorem~1 in \cite{Singer} page 674, that
$\mathcal R$ admits a Liouvillian first integral.
\end{proof}

From the above lemma we have then two possibilities:

\noindent{\bf Case 1}.  There is a solution $\gamma(z)=A(z)/B(z)$ for the Riccati equation, where $A, B$ are polynomials. In this case there is a rational change of coordinates of the form $T= t - A(z)/B(z)$
 that takes the Riccati model $\mathcal R$ into  a Bernoulli foliation $\mathcal B: \frac{dT}{dz}=
 -\frac{B(z)a(z)T^2 +(b(z)B(z)+2a(z) A(z))T}{a(z)B(z)}=-T^2 -\tilde b(z)T$ where
  $\tilde b(z)=\frac{b(z)B(z)+2a(z) A(z)}{ a(z) B(z)}=\frac{b(z)}{a(z)}+2 \frac{A(z)}{B(z)}=
  \frac{b(z)}{a(z)}+ 2\gamma(z)$. In this case the original ODE $L[u](z)=0$ becomes $\tilde L[U](z)=U^{\prime \prime} + \tilde b (z) U^\prime=0$ after a rational change of coordinates
 $U=\exp(\int^z Td\eta)=\exp \big(\int^z\big(t - \frac{A(\eta)}{B(\eta)}\big)d\eta\big)=\exp (\int^z td\eta)\exp\big(-\int^z\frac{A(\eta)}{B(\eta)}d\eta\big)=
 \exp\big(\int^z -\frac{A(z)}{B(z)}dz\big)\cdot u=u\exp(-\int^z\gamma(\eta)d\eta) $ where  $\gamma(z)=\frac{A(z)}{B(z)}$.
 This shows that we have  Liouvillian solutions to the ODE which are given by
$$\begin{array}{rl}\exp\big(-\int^z\gamma(\eta)d\eta\big) u(z)&=U(z)=\ell +k \int^z \exp\big(-\int^\eta \tilde b(\xi)d\xi\big) d\eta \\&\\
&=\ell +k\int^z \exp\big(-\int^\eta \frac{b(\xi)}{a(\xi)}d\xi\big)\cdot \exp\big(\int^\eta -2\gamma(\xi)d\xi\big) d\eta\end{array}$$
for constants $k, \ell\in \mathbb C$. So that

 $$u(z)=\exp\big(\int^z\gamma(\eta)d\eta\big)\big[\ell +k\int^z \exp\big(-\int^\eta \frac{b(\xi)}{a(\xi)}d\xi\big)\cdot \exp\big(\int^\eta -2\gamma(\xi)d\xi\big) d\eta\big]$$
for constants $k, \ell \in \mathbb C$.

 \noindent{\bf Case 2}. We have $c(z)=0$ and therefore the original ODE is of the form
 $L[u]=a(z) u^{\prime \prime} + b(z) u^\prime=0$. Thus the solutions are Liouvillian
 given by
 $u(z) =k_1 + k_2 \int^z \exp(-\int^\eta \frac{b(\xi)}{a(\xi)}d\xi)d\eta$ for constants $k_1,k_2\in\mathbb{C}$.
\end{proof}

\section{Third order equations}

\subsection{Third order Euler equations}

\begin{Theorem}\label{chareulerequation2}{\rm Consider the third order differential equation \begin{equation}\label{2chareuler1}
z^3u^{\prime\prime\prime}+z^2a(z)u^{\prime\prime}+zb(z)u^\prime+c(z)u=0
\end{equation}
where $a,b,c$ are entire functions in the complex plane. Thus the origin and the infinity are regular singular points of (\ref{chareuler1}) if and only if equation (\ref{2chareuler1}) is an  Euler equation.}
\end{Theorem}
\begin{proof}{\rm We already know that every Euler equation has the origin and the infinity as
singular points. Let us see the converse. By the change of coordinates $z=1/t$ and considering $w(t)=u\big(\frac{1}{t}\big)$ we have
$$u^\prime\big(\frac{1}{t}\big)=-t^2w^\prime(t), \, \, u^{\prime\prime}\big(\frac{1}{t}\big)=t^4w^{\prime\prime}(t)+2t^3w^\prime(t), \, \, u^{\prime\prime\prime}\big(\frac{1}{t}\big)=-t^6w^{\prime\prime\prime}(t)-6t^5w^{\prime\prime}(t)-6t^4w^\prime(t),$$
hence
$$\frac{1}{t^3}u^{\prime\prime\prime}\big(\frac{1}{t}\big)+\frac{1}{t^2}a\big(\frac{1}{t}\big)u^{\prime\prime}\big(\frac{1}{t}\big)+\frac{1}{t}b\big(\frac{1}{t}\big)u^\prime\big(\frac{1}{t}\big)+c\big(\frac{1}{t}\big)u\big(\frac{1}{t}\big)=0$$

$$\frac{1}{t^3}[-t^6w^{\prime\prime\prime}(t)-6t^5w^{\prime\prime}(t)-6t^4w^\prime(t)]+\frac{1}{t^2}a\big(\frac{1}{t}\big)[t^4w^{\prime\prime}(t)+2t^3w^\prime(t)]+\frac{1}{t}b\big(\frac{1}{t}\big)[-t^2w^\prime(t)]+c\big(\frac{1}{t}\big)w(t)=0$$
\begin{equation}\label{2chareuler2}
t^3w^{\prime\prime\prime}(t)+t^2\big[6-a\big(\frac{1}{t}\big)\big]w^{\prime\prime}(t)+t\big[6-2a\big(\frac{1}{t}\big)+b\big(\frac{1}{t}\big)\big]w^\prime(t)-c\big(\frac{1}{t}\big)w(t)=0.
\end{equation}
Given that the infinity is a regular singular point of (\ref{2chareuler1}) we have that the origin
is a regular singular point of (\ref{2chareuler2}) consequently there exist the limits
$$\displaystyle\lim_{t\to 0}\frac{t^3\big[6-a\big(\frac{1}{t}\big)\big]}{t^3}=\displaystyle\lim_{t\to 0}\big[6-a\big(\frac{1}{t}\big)\big]=\displaystyle\lim_{t\to 0}\big[(6-a_0)-\frac{a_1}{t}+\frac{a_2}{t^2}+\ldots\big],$$

$$\displaystyle\lim_{t\to 0}\frac{t^3\big[6-2a\big(\frac{1}{t}\big)+b\big(\frac{1}{t}\big)\big]}{t^3}=\displaystyle\lim_{t\to 0}\big[6-2a\big(\frac{1}{t}\big)+b\big(\frac{1}{t}\big)\big]=\displaystyle\lim_{t\to 0}\big[(6-2a_0+b_1)-\frac{a_1-b_1}{t}+\frac{a_2-b_2}{t^2}+\ldots\big]$$
 and
$$\displaystyle\lim_{t\to 0}\frac{t^3c\big(\frac{1}{t}\big)}{t^3}=\displaystyle\lim_{t\to 0}c\big(\frac{1}{t}\big)=\displaystyle\lim_{t\to 0}\big[c_0+\frac{c_1}{t}+\frac{c_2}{t^2}+\ldots\big].$$
and then  $0=a_1=a_2=\ldots$, $0=b_1=b_2=\ldots$ and $0=c_1=c_2=\ldots$. Hence equation (\ref{2chareuler1}) is of the form
$$z^3u^{\prime\prime\prime}(z)+a_0z^2u^{\prime\prime}+b_0zu^\prime(z)+c_0u(z)=0$$
  which is an  Euler equation.}
\end{proof}

\begin{Theorem}\label{2charsing}{\rm Consider the third order differential equation
\begin{equation}\label{2charsing1}
a(z)u^{\prime\prime\prime}+b(z)u^{\prime\prime}+c(z)u^\prime+d(z)u=0
\end{equation}
where $a,b,c,d$ are entire functions with $a^\prime(0)\neq0$. Thus the origin and the infinity are regular singular points of (\ref{2charsing1}) if and only if there exist  $k=0,1,2,\ldots$ in such a way that equation (\ref{2charsing1}) is of the form
\begin{equation}\label{2charsing2}
\begin{array}{c}
(A_1z+\ldots+A_{k+3}z^{k+3})u^{\prime\prime\prime}(z)+(B_0+B_1z+\ldots+B_{k+2}z^{k+2})u^{\prime\prime}(z)\\ \\+(C_0+C_1z+\ldots+C_{k+1}z^{k+1})u^\prime(z)+(D_0+D_1z+\ldots+D_{k}z^k)u(z)=0\end{array}
\end{equation}
where $A_1,\ldots,A_{k+3},B_0,\ldots,B_{k+2},C_0,\ldots,C_{k+1},D_0,\ldots,D_k$ are constants such that $A_1,A_{k+3}\neq0$.
.}
\end{Theorem}
\begin{proof}{\rm Let us first see  that (\ref{2charsing2}) has the origin and the infinity as regular singular point. Clearly the origin é singular point of (\ref{2charsing2}) and since there exist the limits

$$\displaystyle\lim_{z\to 0}\frac{z(B_0+B_1z+\ldots+B_{k+2}z^{k+2})}{(A_1z+\ldots+A_{k+3}z^{k+3})}=\displaystyle\lim_{z\to 0}\frac{B_0+B_1z+\ldots+B_{k+2}z^{k+2}}{(A_1+\ldots+A_{k+3}z^{k+2})}=\frac{B_0}{A_1},$$
$$\displaystyle\lim_{z\to 0}\frac{z^2(C_0+C_1z+\ldots+C_{k+1}z^{k+1})}{A_1z+\ldots+A_{k+3}z^{k+3}}=\displaystyle\lim_{z\to 0}\frac{z(C_0+C_1z+\ldots+C_{k+1}z^{k+1})}{A_1+\ldots+A_{k+3}z^{k+2}}=0$$
 and
$$\displaystyle\lim_{z\to 0}\frac{z^3(D_0+D_1z+\ldots+D_{k}z^{k})}{A_1z+\ldots+A_{k+3}z^{k+3}}=\displaystyle\lim_{z\to 0}\frac{z^2(D_0+D_1z+\ldots+D_{k}z^{k})}{A_1+\ldots+A_{k+3}z^{k+2}}=0$$
we have that the origin is a regular singular point. Putting  $z=1/t$ and considering $w(t)=u\big(\frac{1}{t}\big)$ we have
$$u^\prime\big(\frac{1}{t}\big)=-t^2w^\prime(t), \, \, u^{\prime\prime}\big(\frac{1}{t}\big)=t^4w^{\prime\prime}(t)+2t^3w^\prime(t), \, \, u^{\prime\prime\prime}\big(\frac{1}{t}\big)=-t^6w^{\prime\prime\prime}(t)-6t^5w^{\prime\prime}(t)-6t^4w^\prime(t),$$
hence
$$\begin{array}{c}
\big(\frac{A_1}{t}+\ldots+\frac{A_{k+3}}{t^{k+3}}\big)u^{\prime\prime\prime}\big(\frac{1}{t}\big)+\big(B_0+\frac{B_1}{t}+\ldots+\frac{B_{k+2}}{t^{k+2}}\big)u^{\prime\prime}\big(\frac{1}{t}\big)\\
\\+\big(C_0+\frac{C_1}{t}+\ldots+\frac{C_{k+1}}{t^{k+1}}\big)u^\prime\big(\frac{1}{t}\big)+\big(D_0+\frac{D_1}{t}+\ldots+\frac{D_{k}}{t^{k}}\big)u\big(\frac{1}{t}\big)=0\end{array}$$

$$\begin{array}{c}
\big(\frac{A_1}{t}+\ldots+\frac{A_{k+3}}{t^{k+3}}\big)[-t^6w^{\prime\prime\prime}(t)-6t^5w^{\prime\prime}(t)-6t^4w^\prime(t)]+\big(B_0+\frac{B_1}{t}+\ldots+\frac{B_{k+2}}{t^{k+2}}\big)[t^4w^{\prime\prime}(t)+2t^3w^\prime(t)]\\
\\+\big(C_0+\frac{C_1}{t}+\ldots+\frac{C_{k+1}}{t^{k+1}}\big)[-t^2w^\prime(t)]+\big(D_0+\frac{D_1}{t}+\ldots+\frac{D_{k}}{t^{k}}\big)w(t)=0\end{array}$$
\begin{equation}\label{2charsing3}
\begin{array}{c}
\big(A_1t^{k+5}+\ldots+A_{k+3}t^{3}\big)w^{\prime\prime\prime}(t)+\big((6A_1-B_0)t^{k+4}+\ldots+(6A_{k+3}-B_{k+2})t^2\big)w^{\prime\prime}(t)\\
\\+\big( (6A_1-2B_0)t^{k+3}+(6A_2-2B_1+C_0)t^{k+2}+\ldots+(6A_{k+3}-2B_{k+2}+C_{k+1})t\big)w^\prime(t)\\
\\-\big(D_0t^k+\ldots+D_{k}\big)w(t)=0\end{array}
\end{equation}
Observe that the origin is a singular point of (\ref{2charsing3}) and since there exist the limits
$$\displaystyle\lim_{t\to 0}\frac{t\big((6A_1-B_0)t^{k+4}+\ldots+(6A_{k+3}-B_{k+2})t^2\big)}{A_1t^{k+5}+\ldots+A_{k+3}t^3}=\frac{2A_{k+3}-B_{k+2}}{A_{k+3}}$$
$$\displaystyle\lim_{t\to 0}\frac{t^2\big((6A_1-2B_0)t^{k+3}+\ldots+(6A_{k+3}-2B_{k+2}+C_{k+1})t\big)}{A_1t^{k+5}+\ldots+A_{k+3}t^3}=\frac{6A_{k+3}-2B_{k+2}+C_{k+1}}{A_{k+3}}$$
 and
$$\displaystyle\lim_{t\to 0}\frac{t^3\big(D_0t^k+\ldots+D_{k}\big)}{A_1t^{k+5}+\ldots+A_{k+2}t^3}=\displaystyle\lim_{t\to 0}\frac{D_0t^k+\ldots+D_{k}}{A_1t^{k+2}+\ldots+A_{k+3}}=\frac{D_k}{A_{k+3}}$$we have that o the origin
is a regular singular point of (\ref{2charsing3}) consequently the infinity is a regular singular point of (\ref{2charsing2}). Conversely, assume that the origin and the infinity are regular singular points of (\ref{2charsing1}). Hence by the change of coordinates $z=1/t$ and considering $v(t)=u\big(\frac{1}{t}\big)$
$$u^\prime\big(\frac{1}{t}\big)=-t^2v^\prime(t),\, \, u^{\prime\prime}\big(\frac{1}{t}\big)=t^4v^{\prime\prime}(t)+2t^3v^\prime(t), \, \, u^{\prime\prime\prime}\big(\frac{1}{t}\big)=-t^6v^{\prime\prime\prime}(t)-6t^5v^{\prime\prime}(t)-6t^4v^\prime(t),$$

hence
$$a\big(\frac{1}{t}\big)u^{\prime\prime\prime}\big(\frac{1}{t}\big)+b\big(\frac{1}{t}\big)u^{\prime\prime}\big(\frac{1}{t}\big)+c\big(\frac{1}{t}\big)u^\prime\big(\frac{1}{t}\big)+d\big(\frac{1}{t}\big)u\big(\frac{1}{t}\big)=0$$
$$a\big(\frac{1}{t}\big)[-t^6v^{\prime\prime\prime}(t)-6t^5v^{\prime\prime}(t)-6t^4v^\prime(t)]+b\big(\frac{1}{t}\big)[t^4v^{\prime\prime}(t)+2t^3v^\prime(t)]+c\big(\frac{1}{t}\big)[-t^2v^\prime(t)]+d\big(\frac{1}{t}\big)v(t)=0$$
\begin{equation}\label{2charsing4}
t^6a\big(\frac{1}{t}\big)v^{\prime\prime\prime}(t)+\big[6t^5a\big(\frac{1}{t}\big)-t^4b\big(\frac{1}{t}\big)\big]v^{\prime\prime}(t)+\big[6t^4a\big(\frac{1}{t}\big)-2t^3b\big(\frac{1}{t}\big)+t^2c\big(\frac{1}{t}\big)\big]v^\prime(t)-d\big(\frac{1}{t}\big)v(t)=0.
\end{equation}
Given that the infinity is a regular singular point of (\ref{2charsing1}) then the origin is a regular singular point of (\ref{2charsing4}). Hence, there exist the limits

$$\displaystyle\lim_{t\to 0}\frac{t\big[6t^5a\big(\frac{1}{t}\big)-t^4b\big(\frac{1}{t}\big)\big]}{t^6a\big(\frac{1}{t}\big)}=\displaystyle\lim_{t\to 0}\big[6-\frac{b\big(\frac{1}{t}\big)}{ta\big(\frac{1}{t}\big)}\big]=6-\lim_{t\to 0}\frac{b_0+\frac{b_1}{t}+\frac{b_2}{t^2}+\ldots}{a_1+\frac{a_2}{t}+\frac{a_3}{t^2}+\ldots}$$,

$$\begin{array}{c l}\displaystyle\lim_{t\to 0}\frac{t^2\big[6t^4a\big(\frac{1}{t}\big)-2t^3b\big(\frac{1}{t}\big)+t^2c\big(\frac{1}{t}\big)\big]}{t^6a\big(\frac{1}{t}\big)}&=\displaystyle\lim_{t\to 0}\big[6-\frac{2tb\big(\frac{1}{t}\big)-c\big(\frac{1}{t}\big)}{t^2a\big(\frac{1}{t}\big)}\big]\\ &\\&=6-\displaystyle\lim_{t\to 0}\frac{2b_0t+(2b_1-c_0)+\frac{2b_2-c_1}{t}+\frac{2b_3-c_2}{t^2}+\ldots}{a_1t+a_2+\frac{a_3}{t}+\frac{a_4}{t^2}+\ldots}\end{array}$$
 and
$$\displaystyle\lim_{t\to 0}\frac{t^3d\big(\frac{1}{t}\big)}{t^6a\big(\frac{1}{t}\big)}=\displaystyle\lim_{t\to 0}\frac{d\big(\frac{1}{t}\big)}{t^3a\big(\frac{1}{t}\big)}=\displaystyle\lim_{t\to 0}\frac{d_0+\frac{d_1}{t}+\frac{d_2}{t^2}+\frac{d_3}{t^3}+\ldots}{a_1t^2+a_2t+a_3+\frac{a_4}{t}+\ldots}$$
and then  we have there exists  $k=0,1,2\ldots$ in such a way that $a_{k+3}\neq0$, $0=a_{k+4}=a_{k+5}=\ldots$, $0=b_{k+3}=b_{k+4}=\ldots$, $0=c_{k+2}=c_{k+3}=\ldots$ and $0=d_{k+1}=d_{k+2}=\ldots$. Given that the origin is a regular singular point of (\ref{2charsing1}) we have that there exist the limits
$$\displaystyle\lim_{z\to 0}\frac{zb(z)}{a(z)}=\displaystyle\lim_{z\to 0}\frac{b_0z+b_1z^2+\ldots+b_{k+2}z^{k+3}}{a_1z+a_2z^2+\ldots+a_{k+3}z^{k+3}}=\displaystyle\lim_{z\to 0}\frac{b_0+b_1z+\ldots+b_{k+2}z^{k+2}}{a_1+a_2z+\ldots+a_{k+3}z^{k+2}}$$
$$\displaystyle\lim_{z\to 0}\frac{z^2c(z)}{a(z)}=\displaystyle\lim_{z\to 0}\frac{c_0z^2+c_1z^3+\ldots+c_{k+1}z^{k+3}}{a_1z+a_2z^2+\ldots+a_{k+3}z^{k+3}}=\displaystyle\lim_{z\to 0}\frac{c_0z+c_1z^2+\ldots+c_{k+1}z^{k+2}}{a_1+a_2z+\ldots+a_{k+3}z^{k+2}}$$
 and
$$\displaystyle\lim_{z\to 0}\frac{z^3d(z)}{a(z)}=\displaystyle\lim_{z\to 0}\frac{d_0z^3+d_1z^4+\ldots+d_{k}z^{k+3}}{a_1z+a_2z^2+\ldots+a_{k+3}z^{k+3}}=\displaystyle\lim_{z\to 0}\frac{d_0z^2+d_1z^3+\ldots+d_{k}z^{k+2}}{a_1+a_2z+\ldots+a_{k+2}z^{k+1}}$$
 provided that  $a_1\neq0$. Thence (\ref{2charsing1}) is of the form (\ref{2charsing2}).
}
\end{proof}

\subsection{Regular singularities in order three}

The simplest example  of a third order equation that has a regular singular point at the origin is a \textit{equation of Euler}
\begin{equation}\label{eqs7}
L(y):=x^3y^{\prime\prime\prime}+ax^2y^{\prime\prime}+bxy^\prime+cy=0,
\end{equation}
where $a,b,c$ are constants. Let us consider that $x>0$. The idea is to change variables  $x=e^z$. Suppose that $\varphi$ is a solution of (\ref{eqs7}) and define $\tilde{\varphi}(z)=\varphi(e^z)$ thus we have

$$\frac{d\varphi}{dx}\big(e^z\big)=e^{-z}\frac{d\tilde{\varphi}}{dz}(z),$$
$$\frac{d^2\varphi}{dx^2}\big(e^z\big)=e^{-2z}\big(\frac{d^2\tilde{\varphi}}{dz^2}(z)-\frac{d\tilde{\varphi}}{dz}(z)\big),$$

$$\frac{d^3\varphi}{dx^3}\big(e^z\big)=e^{-3z}\big(\frac{d^3\tilde{\varphi}}{dz^3}(z)-3\frac{d^2\tilde{\varphi}}{dz^2}(z)+2\frac{d\tilde{\varphi}}{dz}(z)\big).$$

Since, from (\ref{eqs7})

$$ e^{3z}\frac{d^3\varphi}{dx^3}\big(e^z\big)+ae^{2z}\frac{d^2\varphi}{dx^2}\big(e^z\big)+be^z\frac{d\varphi}{dx}\big(e^z\big)+c\varphi\big(e^z\big)=0,$$
we have:
$$\tilde{\varphi}^{\prime\prime\prime}(z)+(a-3)\tilde{\varphi}^{\prime\prime}(z)+(2-a+b)\tilde{\varphi}^{\prime}(z)+c\tilde{\varphi}(z)=0 $$
Hence $\tilde{\varphi}$ satisfies the equation:
\begin{equation}\label{eqs8}
\tilde{L}(y)=y^{\prime\prime\prime}+(a-3)y^{\prime\prime}+(2-a+b)y^{\prime}+cy=0.
\end{equation}
The characteristic polynomial associate to (\ref{eqs8}) is
$$r^3+(a-3)r^2+(2-a+b)r+c=0$$o qual is can be written  equivalently
\begin{equation}\label{eqs9}
q(r)=r(r-1)(r-2)+ar(r-1)+br+c=0
\end{equation}
this polynomial is called \textit{indicial polynomial} associate to (\ref{eqs7}). Let $r_1,r_2,r_2$ be the roots of the polynomial (\ref{eqs9}). There are three cases:
\begin{itemize}
\item All roots are equal: $r_1=r_2=r_3$. Hence $\{e^{r_1z},ze^{r_1z},z^2e^{r_1z}\}$ is base of the solution space of  equation (\ref{eqs8}) and considering the change of coordinate $x=e^z$ we have that $\{x^{r_1},\log x\;x^{r_1},(\log x)^2\;x^{r_1}\}$ is
a base of the solution space of  equation (\ref{eqs7}).
\item Two equal roots and one different: $r_1=r_2\neq r_3$. Hence $\{e^{r_1z},ze^{r_1z},e^{r_3z}\}$ is base of the solution space of  equation (\ref{eqs8}) and performing the change of coordinate $x=e^z$ we have that $\{x^{r_1},\log x\;x^{r_1},x^{r_3}\}$ is a base of the solution space of  equation (\ref{eqs7}).
\item All roots are different: $r_1\neq r_2$, $r_2\neq r_3$ and $r_1\neq r_3$ Hence $\{e^{r_1z},e^{r_2z},e^{r_3z}\}$ is base of the solution space of  equation (\ref{eqs8}) and using the change of coordinate $x=e^z$ we have that $\{x^{r_1},x^{r_2},x^{r_3}\}$ is a base of the solution space of  equation (\ref{eqs7}).
\end{itemize}
\subsection{Third order equations with regular singular points}

A  equation of third order with a  regular singular point at $x_0$, has the form:
\begin{equation}\label{eqs10}
(x-x_0)^3y^{\prime\prime\prime}+a(x)(x-x_0)^2y^{\prime\prime}+b(x)(x-x_0)y^\prime+c(x)y=0,
\end{equation}
where $a,b$ are analytic at $x_0$. Hence, $a,b,c$ can be written  in power series of the following form:
$$a(x)=\sum^\infty_{k=0}\alpha_k(x-x_0)^k,\;b(x)=\sum^\infty_{k=0}\beta_k(x-x_0)^k,\;c(x)=\sum^\infty_{k=0}\gamma_k(x-x_0)^k,$$
which are convergent in a certain interval $|x-x_0|<R_0$, for some $R_0>0$. We shall look for solutions of (\ref{eqs10}) in a neighborhood  of $x_0$. For the sake of simplicity of notation, we shall assume that $x_0=0$.
If $x_0\neq0$, is easy to transform (\ref{eqs10}) in  an equivalent  equation  with a  regular singular point at the origin. Put $t=x-x_0$, and
$$\tilde{a}(t)=a(x_0+t)=\sum^\infty_{k=0}\alpha_k t^k,\;\tilde{b}(t)=b(x_0+t)=\sum^\infty_{k=0}\beta_kt^k,\;\tilde{c}(t)=c(x_0+t)=\sum^\infty_{k=0}\gamma_kt^k,$$
The power series for $\tilde{a},\tilde{b},\tilde{c}$ converge in the interval $|t|<R_0$ centered at $t=0$. Let $\varphi$ be  any solution of (\ref{eqs10}), and define $\tilde{\varphi}$ of the following form:
$$\tilde{\varphi}(t)=\varphi(x_0+t).$$
Then
$$\frac{d\tilde{\varphi}}{dt}(t)=\frac{d\varphi}{dx}(x_0+t),\;\;\frac{d^2\tilde{\varphi}}{dt^2}(t)=\frac{d^2\varphi}{dx^2}(x_0+t),\;\;\;\frac{d^3\tilde{\varphi}}{dt^3}(t)=\frac{d^3\varphi}{dx^3}(x_0+t),$$
and therefore we see that $\tilde{\varphi}$ satisfies the equation:
\begin{equation}\label{eqs11}
t^3u^{\prime\prime\prime}+\tilde{a}(t)t^2u^{\prime\prime}+\tilde{b}(t)tu^\prime+\tilde{c}(t)u=0,
\end{equation}
where now $u^\prime=\frac{du}{dt}$. This is an equation with a  regular singular point at $t=0$. Conversely, if $\tilde{\varphi}$ satisfies (\ref{eqs11}), a function $\varphi$ given by
$$\varphi(x)=\tilde{\varphi}(x-x_0)$$satisfies (\ref{eqs10}). In this sense (\ref{eqs11}) is equivalent to (\ref{eqs10}).

Com $x_0=0$ in (\ref{eqs10}), we can write this equation in the following form:
\begin{equation}\label{eqs12}
L(y)=x^3y^{\prime\prime\prime}+a(x)x^2y^{\prime\prime}+b(x)xy^\prime+c(x)y=0,
\end{equation}
where $a,b,c$ are analytic at the origin, and moreover has a power series expansion, of the following form:
\begin{equation}\label{eqs13}
a(x)=\sum^\infty_{k=0}\alpha_kx^k,\;b(x)=\sum^\infty_{k=0}\beta_kx^k,\;c(x)=\sum^\infty_{k=0}\gamma_kx^k,
\end{equation}
which are convergent in an interval $|x|<R_0$, $R_0>0$. Equation of Euler is a special case of (\ref{eqs12}), with $a,b,c$ constants. The fact that the term of greater order (terms that have  $x$ as factor) in the series (\ref{eqs13}), is to introduce series in the  solutions of (\ref{eqs12}).

\subsection{Proof of theorem~\ref{thmfrobb3I}}
\begin{proof}{\rm
Let $\varphi$ be a solution of (\ref{eq1}) of the form
\begin{equation}\label{eq2}
\varphi(x)=x^r\sum^{\infty}_{n=0}d_nx^n
\end{equation}
where $d_0\neq0$. Since $a(x),b(x)$ and $c(x)$ are analytic in $|x|<R$ we have that
\begin{equation}\label{eq3}
a(x)=\sum^\infty_{n=0}a_nx^n,\;\;\;b(x)=\sum^\infty_{n=0}b_nx^n\;\;\;\mbox{ and }\;\;\;c(x)=\sum^\infty_{n=0}c_nx^n.
\end{equation}
Then
$$\varphi^\prime(x)=x^{r-1}\sum^{\infty}_{n=0}(n+r)d_nx^n $$
$$\varphi^{\prime\prime}(x)=x^{r-2}\sum^{\infty}_{n=0}(n+r)(n+r-1)d_nx^n$$
$$\varphi^{\prime\prime\prime}(x)=x^{r-3}\sum^{\infty}_{n=0}(n+r)(n+r-1)(n+r-2)d_nx^n$$
and thus we have
$$\begin{array}{l c l}c(x)\varphi(x)&=&x^r\big(\sum^\infty_{n=0}d_nx^n\big)\big(\sum^\infty_{n=0}c_nx^n\big)=x^r\sum^\infty_{n=0}\tilde{c}_nx^n\end{array}$$where
$\tilde{c}_n=\sum^{n}_{j=0}d_jc_{n-j}$.
$$\begin{array}{l c l}xb(x)\varphi^\prime(x)&=&x^r\big(\sum^\infty_{n=0}(n+r)d_nx^n\big)\big(\sum^\infty_{n=0}b_nx^n\big)= x^r\sum^\infty_{n=0}\tilde{b}_nx^n\end{array}$$where
$\tilde{b}_n=\sum^{n}_{j=0}(j+r)d_jb_{n-j}$.
$$\begin{array}{l c l}x^2a(x)\varphi^{\prime\prime}(x)&=&x^r\big(\sum^\infty_{n=0}(n+r)(n+r-1)d_nx^n\big)\big(\sum^\infty_{n=0}a_nx^n\big)=x^r\sum^\infty_{n=0}\tilde{a}_nx^n\end{array}$$where
$\tilde{a}_n=\sum^{n}_{j=0}(j+r)(j+r-1)d_ja_{n-j}$.

$$\begin{array}{l c l}x^3\varphi^{\prime\prime\prime}(x)&=&x^r\sum^{\infty}_{n=0}(n+r)(n+r-1)(n+r-2)d_nx^n.\end{array}$$
Therefore
$$L(\varphi)(x)=x^r\sum^\infty_{n=0}\big[(n+r)(n+r-1)(n+r-2)d_n+\tilde{a}_n+\tilde{b}_n+\tilde{c}_n\big]x^n$$e according to (\ref{eq1}) we have

$$[\;\;\;]_n=(n+r)(n+r-1)(n+r-2)d_n+\tilde{a}_n+\tilde{b}_n+\tilde{c}_n=0,\;\;\;n=0,1,2,\ldots$$
Using the definitions of $\tilde{a}_n,\tilde{b}_n$ and $\tilde{c}_n$ we can write $[\;\;\;]_n$ since
$$\begin{array}{lcl} [\;\;\;]_n&=& (n+r)(n+r-1)(n+r-2)d_n+\sum^{n}_{j=0}(j+r)(j+r-1)d_ja_{n-j}\\&&\\&&+\sum^{n}_{j=0}(j+r)d_jb_{n-j}+\sum^{n}_{j=0}d_jc_{n-j}\\&&\\&=&[(n+r)(n+r-1)(n+r-2)+(n+r)(n+r-1)a_0+(n+r)b_0+c_0]d_n\\&&\\&&+\sum^{n-1}_{j=0}[(j+r)(j+r-1)a_{n-j}+(j+r)b_{n-j}+c_{n-j}]d_j
\end{array}$$
For $n=0$ we have
\begin{equation}\label{eq4} r(r-1)(r-2)+r(r-1)a_0+rb_0+c_0=0
\end{equation}
provided that $d_0\neq0$. We see that
\begin{equation}\label{eq5}
[\;\;\;]_n=q(n+r)d_n+e_n,\;\;\;\;\;n=1,2,\ldots
\end{equation}
where
\begin{equation}\label{eq6}
e_n=\sum^{n-1}_{j=0}[(j+r)(j+r-1)a_{n-j}+(j+r)b_{n-j}+c_{n-j}]d_j,\;\;\;\;\;n=1,2,\ldots
\end{equation}
Note that $e_n$ is a linear combination of $d_0,d_1,\ldots,d_{n-1}$, whose coefficients involve the functions $a,b,c$ and $r$. Letting  $r$ and $d_0$ be unknown, for the moment we solve equations (\ref{eq5}) and (\ref{eq6}) in terms of $d_0$ and $r$. These solutions are represented  according to $D_n(r)$, and a $e_n$ by $E_n(r)$.
Hence
$$E_1(r)=(r(r-1)a_1+rb_1+c_1)d_0\;\;\;\;\;\;D_1(r)=-\frac{E_1(r)}{q(1+r)},$$
and in general:
\begin{equation}\label{eq7}
E_n(r)=\sum^{n-1}_{j=0}[(j+r)(j+r-1)a_{n-j}+(j+r)b_{n-j}+c_{n-j}]D_j(r)
\end{equation}

\begin{equation}\label{eq8}
D_n(r)=-\frac{E_n(r)}{q(n+r)},\;\;\;\;\;n=1,2,\ldots
\end{equation}

The terms $D_n$ thus given, are rational functions of $r$, whose only indefinite points are the points $r$ for which $q(r+n)=0$ for some $n=1,2,\ldots$. Among these, there exist only two possible points. Let us define $\varphi$ as follows:
\begin{equation}\label{eq9}
\varphi(x,r)=d_0x^r+x^r\sum^{\infty}_{n=1}D_n(r)x^n.
\end{equation}
If the series (\ref{eq9}) converges for $0<x<R$, then we have:
\begin{equation}\label{eq10}
L(\varphi)(x,r)=d_0q(r)x^r.
\end{equation}
Now we have a following situation: If a $\varphi$ given by (\ref{eq1}) is a solution of (\ref{eq2}), then $r$ must be a root of the indicial polynomial $q$, and then $d_n$ ($n\geq1$) are given exclusively in terms of $d_0$ and $r$ according to os $D_n(r)$ of (\ref{eq8}), provided that $q(n+r)\neq0$, $n=1,2,\ldots$.

Conversely, if $r$ is a root of $q$ and if the $D_n(r)$ are given (i.e., $q(n+r)\neq0$ for $n=1,2,\ldots$) then a function $\varphi$ given by $\varphi(x)=\varphi(x,r)$ is a solution of (\ref{eq2}) for every choice of $d_0$, provided that  that the series (\ref{eq9}) is convergent.

We have $r_1,r_2,r_3$ are as roots of $q$ with $\mbox{Re}(r_1)\geq \mbox{Re}(r_2)\geq\mbox{Re}(r_3)$. Then $q(n+r_1)\neq0$ for every $n=1,2,\ldots$.
Hence, $D_n(r_1)$ exists for every $n=1,2,\ldots$, and putting $d_0=D_0(r_1)=1$ we have that a function $\varphi_1$ given by
\begin{equation}\label{eq11}
\varphi_1(x)=x^{r_1}\sum^\infty_{n=0}D_n(r_1)x^n,\;\;\;\;D_0(r_1)=1,
\end{equation}
is a solution of (\ref{eq2}), provided that the series converges. Since $r_1-r_2\notin\mathbb{Z}^+_0$, $r_1-r_3\notin\mathbb{Z}^+_0$ and $r_2-r_3\notin\mathbb{Z}^+_0$
then $$q(n+r_2)\neq0\;\;\;\;\mbox{ and }\;\;\;\;q(n+r_3)\neq0 $$for every $n=1,2,\ldots$, then $D_n(r_2)$ and $D_n(r_3)$ are well defined for $n=1,2,\ldots$ hence the functions $\varphi_2$ and $\varphi_3$ are given by

\begin{equation}\label{eq12}
\varphi_2(x)=x^{r_2}\sum^\infty_{n=0}D_n(r_2)x^n,\;\;\;\;D_0(r_2)=1,
\end{equation}
and

\begin{equation}\label{eq13}
\varphi_3(x)=x^{r_3}\sum^\infty_{n=0}D_n(r_3)x^n,\;\;\;\;D_0(r_3)=1.
\end{equation}
Now let us show the convergence of the series
\begin{equation}\label{eq14}
\sum^{\infty}_{n=0}D_n(r)x^n
\end{equation}
where $D_n(r)$ are given recursively by
\begin{equation}\label{eq15}
\begin{array}{c }D_0(r)=1,\\ \\ q(n+r)D_n(r)=-\sum^{n-1}_{j=0}[(j+r)(j+r-1)a_{n-j}+(j+r)b_{n-j}+c_{n-j}]D_j(r),\;n=1,2,\ldots\end{array}
\end{equation}
to see (\ref{eq7}) and (\ref{eq8}). We need to show that the series (\ref{eq14}) converge for $|x|<R$ if $r=r_1$, if $r=r_2$ and if $r=r_3$, provided that $r_1-r_2\notin\mathbb{Z}^+_0$, $r_1-r_3\notin\mathbb{Z}^+_0$ and $r_2-r_3\notin\mathbb{Z}^+_0$.
Note that
$$q(r)=(r-r_1)(r-r_2)(r-r_3),$$
and according to that we have
$$\begin{array}{c}  q(n+r_1)=n(n+r_1-r_2)(n+r_1-r_3),\\ \\
q(n+r_2)=n(n+r_2-r_1)(n+r_2-r_3),\\ \\
q(n+r_3)=n(n+r_3-r_1)(n+r_3-r_2).
\end{array}$$
Henceforth:
\begin{equation}\label{eq16}
\begin{array}{c}  |q(n+r_1)|\geq n(n-|r_1-r_2|)(n-|r_1-r_3|),\\ \\
|q(n+r_2)|\geq n(n-|r_1-r_2|)(n-|r_2-r_3|),\\ \\
|q(n+r_3)|\geq n(n-|r_1-r_3|)(n-|r_2-r_3|).
\end{array}
\end{equation}
Let now $\rho$ be any number that satisfies the inequality $0<\rho<R$. Since the series defined in (\ref{eq14}) are convergent for $|x|=\rho$ exists a constant $M>0$ such that

\begin{equation}\label{eq17}
|a_j|\rho^j\leq M,\;\;\;\;\;|b_j|\rho^j\leq M\;\;\;\;\;|c_j|\rho^j\leq M\;\;\;\;\;j=0,1,2,\ldots
\end{equation}
Replacing (\ref{eq16}) and (\ref{eq17}) in equation (\ref{eq15}), we obtain:
\begin{equation}\label{eq18}
\begin{array}{l}n(n-|r_1-r_2|)(n-|r_1-r_3|)|D_n(r_1)|\leq \\ \\
\hspace{3cm}M\sum^{n-1}_{j=0}( (j+|r_1|)(j+|r_1-1|)+(j+|r_1|)+1)\rho^{j-n}|D_j(r_1)|,\end{array}
\end{equation}
for $n=1,2,\ldots$. Let $N_1$ be the integer that satisfies the inequality
$$N_1-1\leq|r_1-r_2|<N_1, $$ and $N_2$ the integer that satisfies the inequality
$$N_2-1\leq|r_1-r_3|<N_2, $$. Consider the $N=\max\{N_1,N_2\}$ and we shall define $\gamma_0,\gamma_1,\ldots$ of the following form:
$$\gamma_0=D_0(r_1)=1,\;\;\;\gamma_n=|D_n(r_1)|,\;\;\;n=1,2,\ldots,N-1,$$
and
\begin{equation}\label{eq19}
n(n-|r_1-r_2|)(n-|r_1-r_3|)\gamma_n=M\sum^{n-1}_{j=0}( (j+|r_1|)(j+|r_1-1|)+(j+|r_1|)+1)\rho^{j-n}\gamma_j,
\end{equation}
for $n=N,N+1,\ldots$. Then, comparing the definition of $\gamma_n$ with (\ref{eq18}), we see that
\begin{equation}\label{eq20}
|D_n(r_1)|\leq \gamma_n,\;\;\;\;\;\;n=0,1,2,\ldots
\end{equation}
Hence we will shows that the series
\begin{equation}\label{eq21}
\sum^\infty_{n=0}\gamma_nx^n
\end{equation}
is convergent for $|x|<\rho$. Replacing $n$ according to $n+1$ in (\ref{eq19}) we have:
$$
\rho(n+1)(n+1-|r_1-r_2|)(n+1-|r_1-r_3|)\gamma_{n+1}=$$
$$
[n(n-|r_1-r_2|)(n-|r_1-r_3|)
+ M((n+|r_1|)(n+|r_1-1|)+(n+|r_1|)+1)]\gamma_n
$$
for $n\geq N$. Hence
$$\big|\frac{\gamma_{n+1}x^{n+1}}{\gamma_{n}x^{n}}\big|=\frac{[n(n-|r_1-r_2|)(n-|r_1-r_3|)+M((n+|r_1|)(n+|r_1-1|)+(n+|r_1|)+1)]}{\rho(n+1)(n+1-|r_1-r_2|)(n+1-|r_1-r_3|)}|x|$$converge a $|x|/\rho$ when $n\to\infty$. Hence according to the quotient test  the series (\ref{eq21}) converges for $|x|<\rho$. Using (\ref{eq20}) and according to the comparison test for series, we conclude that the series
$$\sum^\infty_{n=0}D_n(r_1)x^n,\;\;\;\;D_0(r_1)=1,$$ also converges for $|x|<\rho$. But since $\rho$ is any number that satisfies the inequality $0<\rho<R$, this already shows that  that this series converges for $|x|<R$.
Replacing $r_1$ by   $r_2$ in all the above computations, we show that
$$\sum^\infty_{n=0}D_n(r_2)x^n,\;\;\;\;D_0(r_2)=1,$$converge for $|x|<R$ provided that $r_1-r_2\notin\mathbb{Z}^+_0$ and $r_2-r_3\notin\mathbb{Z}^+_0$.
Also, replacing $r_1$ by $r_3$ in the above computations, we show that
$$\sum^\infty_{n=0}D_n(r_3)x^n,\;\;\;\;D_0(r_3)=1,$$converge for $|x|<R$ provided that $r_1-r_3\notin\mathbb{Z}^+_0$ and $r_2-r_3\notin\mathbb{Z}^+_0$.}
\end{proof}
\begin{Remark}{\rm Since we already know that in (\ref{eq11}), (\ref{eq12}) and (\ref{eq13}) the coefficients $d_n,\tilde{d_n},\hat{d_n}$ that appear in the  solutions $\varphi_1,\varphi_2,\varphi_3$ of theorem~\ref{thmfrobb3I} are given according to
$$d_n=D_n(r_1),\;\;\;\tilde{d_n}=D_n(r_2)\;\;\;\mbox{ and }\;\;\;\hat{d_n}=D_n(r_3),\;\;\\;n=1,2,\ldots,$$where the $D_n(r)$ are solutions of equations (\ref{eq7}) and (\ref{eq8}) with $D_0(r)=1$.}
\end{Remark}

\subsection{Proof of theorem~\ref{thmfrobb3II}}
We shall work with a  formal method for finding out the form of  the solutions. For such $x$, we have, according to (\ref{eq9}) and (\ref{eq10}),
\begin{equation}\label{eq22}
L(\varphi)(x,r)=d_0q(r)x^r,
\end{equation}
where $\varphi$ is  given by
\begin{equation}\label{eq23}
\varphi(x,r)=d_0x^r+x^r\sum^\infty_{n=1}D_n(r)x^n.
\end{equation}
The functions $D_n(r)$ are given by the recurrence given by the formulas:
\begin{equation}\label{eq24}
\begin{array}{c} D_0(r)=d_0\neq0\\ \\ q(n+r)D_n(r)=-E_n(r)\\ \\ E_n(r)=\sum^{n-1}_{j=0}[(j+r)(j+r-1)a_{n-j}+(j+r)b_{n-j}+c_{n-j}]D_j(r),\;\;\;n=1,2,\ldots;\end{array}
\end{equation}
to see (\ref{eq7}) and (\ref{eq8}).
\begin{proof}$\;${\rm
\begin{enumerate}
\item[(i)] We have
$$q(r_1)=0,\;\;\;\;q^\prime(r_1)=0\;\;\;\;\;q^{\prime\prime}(r_1)=0,$$
and this clearly suggests the derivation of (\ref{eq22}) with respect to $r$.
We obtain
\begin{equation}\label{eq25}
\begin{array}{lcl} \frac{\partial}{\partial r}L(\varphi)(x,r)&=&L\big(\frac{\partial\varphi}{\partial r}\big)(x,r)=d_0[q^\prime(r)+(\log x)q(r)]x^r, \end{array}
\end{equation}
and we see that if $r=r_1=r_2=r_3$, $d_0=1$ we have
$$\varphi_2(x)=\frac{\partial \varphi}{\partial r}(x,r_1)$$ which will give us a solution of the equation, provided that the series converges.
A straightforward computation from (\ref{eq23}), we have:
$$\begin{array}{lcl} \varphi_2(x)&=& x^{r_1}\sum^{\infty}_{n=0}D_n^\prime(r_1)x^n+(\log x)x^{r_1}\sum^\infty_{n=0}D_n(r_1)x^n\\&&\\&=&x^{r_1}\sum^{\infty}_{n=0}D_n^\prime(r_1)x^n+(\log x)\varphi_1(x) \end{array}$$
where $\varphi_1$ is the already obtained solution:
$$\varphi_1(x)=x^{r_1}\sum^{\infty}_{n=0}D_n(r_1)x^n,\;\;\;D_0(r_1)=1.$$
Note that $D_n^\prime(r_1)$ exists for every $n=0,1,2,\ldots$, since os $D_n$ are rational functions of $r$ whose denominator does not vanish in $r=r_1$. Also $D_0(r)=1$ implies that $D_0^\prime(r_1)=0$, and therefore the series that in $\varphi_2$ is multiplying $x^{r_1}$ starts with a first power of $x$.
In order to find other solution we take the derivative of (\ref{eq25}) with respect to $r$.
We obtain
\begin{equation}\label{eq26}
\begin{array}{lcl} \frac{\partial^2}{\partial r^2}L(\varphi)(x,r)&=&L\big(\frac{\partial^2\varphi}{\partial r^2}\big)(x,r)\\&&\\&=&d_0[q^{\prime\prime}(r)+2(\log x)q^\prime(r)+(\log x)^2q(r)]x^r, \end{array}
\end{equation}
and we see that if $r=r_1=r_2=r_3$, $d_0=1$ we have
$$\varphi_3(x)=\frac{\partial^2 \varphi}{\partial r^2}(x,r_1)$$which will give us a solution of the equation, provided that the series converges.
A straightforward computation from (\ref{eq23}), we have:
$$\varphi_3(x)=x^{r_1}\sum^{\infty}_{n=0}D_n^{\prime\prime}(r_1)x^n+(\log x)x^{r_1}\sum^\infty_{n=0}D_n^\prime(r_1)x^n+(\log x)[x^{r_1}\sum^{\infty}_{n=0}D_n^\prime(r_1)x^n+(\log x)x^{r_1}\sum^\infty_{n=0}D_n(r_1)x^n]$$
and therefore $$\varphi_3(x)=
x^{r_1}\sum^{\infty}_{n=0}D_n^{\prime\prime}(r_1)x^n+(\log x)x^{r_1}\sum^\infty_{n=0}D_n^\prime(r_1)x^n+(\log x)[x^{r_1}\sum^{\infty}_{n=0}D_n^\prime(r_1)x^n+(\log x)\varphi_1(x)]$$
$$=x^{r_1}\sum^{\infty}_{n=0}D_n^{\prime\prime}(r_1)x^n+2(\log x)x^{r_1}\sum^\infty_{n=0}D_n^\prime(r_1)x^n+(\log x)^2\varphi_1(x)$$
$$=x^{r_1}\sum^{\infty}_{n=0}D_n^{\prime\prime}(r_1)x^n+2(\log x)\big[x^{r_1}\sum^\infty_{n=0}D_n^\prime(r_1)x^n+(\log x)\varphi_1(x)\big]-(\log x)^2\varphi_1(x)$$
$$=x^{r_1}\sum^{\infty}_{n=0}D_n^{\prime\prime}(r_1)x^n+2(\log x)\varphi_2(x)-(\log x)^2\varphi_1(x).$$

Note that $D_n^{\prime\prime}(r_1)$ also exists for every $n=0,1,2,\ldots$, since os $D_n^\prime$ are rational functions of $r$ whose denominator does not vanish in $r=r_1$. Also $D_0(r)=1$ implies that $D_0^{\prime\prime}(r_1)=D_0^{\prime}(r_1)=0$, and therefore the series that in $\varphi_3$ is multiplying $x^{r_1}$ starts with a first power of $x$.
\item[(ii)] Since $r_1=r_2$ proceed as in item [(i)] and obtain:
$$\varphi_2(x)=x^{r_1}\sum^{\infty}_{n=0}D_n^\prime(r_1)x^n+(\log x)\varphi_1(x)$$
where $\varphi_1$ is the already obtained solution:
$$\varphi_1(x)=x^{r_1}\sum^{\infty}_{n=0}D_n(r_1)x^n,\;\;\;D_0(r_1)=1.$$

Suppose that that $r_1=r_3+m$, where $m\in\mathbb{Z}^+$. If $d_0$ is given,
$$D_1(r_3),\cdots,D_{m-1}(r_3)$$ all  exist and have  finite values, but since
\begin{equation}\label{eq27}
q(r+m)D_m(r)=-E_m(r),
\end{equation}
we encounter some difficulties in the computation  of $D_m(r_3)$. Now
$$q(r)=(r-r_1)^2(r-r_3),$$
and therefore:
$$q(r+m)=(r-r_3)^2(r+m-r_3).$$
If $E_m(r)$ has $(r-r_3)^2$ as a factor (i.e., $E_m(r_3)=E^\prime_m(r_3)=0$) this implies that we can cancel the same factor in $q(r+m)$, and then (\ref{eq27}) gives $D_m(r_3)$ in form of finite number. Then:
$$D_{m+1}(r_3),\;D_{m+2}(r_3),\ldots$$ all of them  exist. In a  situation also especial, we obtain a solution $\varphi_3$ of the form
$$\varphi_3(x)=x^{r_3}\sum^\infty_{n=0}D_n(r_3)x^n,\;\;\;D_0(r_3)=1. $$
It is always possible to  arrange  of such form that $\tilde{E}_m(r_3)=0$, choosing
$$\tilde{D}_0(r)=(r-r_3)^2.$$
Observing (\ref{eq24}) we see that $\tilde{E}_n(r)$ is linear homogeneous em
$$\tilde{D}_0(r),\ldots,\tilde{D}_{n-1}(r),$$
and therefore that $\tilde{E}_n(r)$ has the $\tilde{D}_0(r)=(r-r_3)^2$ as a factor. Hence, $\tilde{D}_m(r_3)$ will exist in form of finite number. Putting
\begin{equation}\label{eq28}
\psi(x,r)=x^r\sum^{\infty}_{n=0}\tilde{D}_n(r)x^n,\;\;\;\;\tilde{D}_0(r)=(r-r_3)^2,
\end{equation}
we find formally that
\begin{equation}\label{eq29}
L(\psi)(x,r)=(r-r_3)^2q(r)x^r.
\end{equation}
Putting $r=r_3$ we obtain formally a solution $\psi$ given by
$$\psi(x)=\psi(x,r_3).$$
Although, $\tilde{D}_0(r_3)=\tilde{D}_1(r_3)=\cdots=\tilde{D}_{m-1}(r_3)=0$. Hence, a series that define $\psi$ really starts with a $m$-th power of $x$, and then $\psi$ has the form:
$$\psi(x)=x^{r_3+m}\sigma(x)=x^{r_1}\sigma(x),$$
where $\sigma$ is a power series.
It is not difficult to see that $\psi$ is precisely a constant multiple of the solution $\varphi_1$ that  is already known.
In order to find a solution really associate with $r_3$, we take the derivative of (\ref{eq29}) with respect to $r$, we obtain:

$$\begin{array}{lcl} \frac{\partial}{\partial r}L(\psi)(x,r)&=&L\big(\frac{\partial\psi}{\partial r}\big)(x,r)=[2(r-r_3)q(r)+(r-r_3)^2(q^\prime(r)+(\log x)q(r))]x^r, \end{array}$$
Now, putting $r=r_3$ we find a function $\varphi_3$ given by
$$\varphi_3(x)=\frac{\partial \psi}{\partial r}(x,r_3)$$
is a solution, provided that the series involved are convergent. We have the form
$$\varphi_3(x)=x^{r_3}\sum^\infty_{n=0}\tilde{D}^\prime_n(r_3)x^n+(\log x)x^{r_3}\sum^\infty_{n=0}\tilde{D}_n(r_3)x^n, $$
where $\tilde{D}_0(r)=(r-r_3)^2$. Since
$$\tilde{D}_0(r_3)=\cdots=\tilde{D}_{m-1}(r_3)=0,$$we can rewrite this in the following form:
$$\varphi_3(x)=x^{r_3}\sum^\infty_{n=0}\tilde{D}^\prime_n(r_3)x^n+c(\log x)\varphi_1(x), $$where $c$ is some constant.
\item[(iii)] Since in item (ii) we obtain:
$$\varphi_3(x)=x^{r_2}\sum^{\infty}_{n=0}D_n^\prime(r_2)x^n+(\log x)\varphi_2(x)$$
where $\varphi_2$ is the already obtained solution:
$$\varphi_2(x)=x^{r_2}\sum^{\infty}_{n=0}D_n(r_2)x^n,\;\;\;D_0(r_2)=1.$$

Also since in item (ii) we have that $r_1=r_2+m$, where $m\in\mathbb{Z}^+$. The other desired solution has the form
$$\varphi_1(x)=x^{r_2}\sum^\infty_{n=0}\tilde{D}^\prime_n(r_2)x^n+(\log x)x^{r_3}\sum^\infty_{n=0}\tilde{D}_n(r_2)x^n, $$
where $\tilde{D}_0(r)=(r-r_2)^2$. Since
$$\tilde{D}_0(r_2)=\cdots=\tilde{D}_{m-1}(r_2)=0,$$we can write this in the following form:
$$\varphi_1(x)=x^{r_2}\sum^\infty_{n=0}\tilde{D}^\prime_n(r_2)x^n+c(\log x)\varphi_2(x), $$where $c$ is some constant.

\item[(iv)] Since in item (ii) we have that $r_1=r_2+m$, where $m\in\mathbb{Z}^+$ and $r_2=r_3+p$, , where $p\in\mathbb{Z}^+$. The desired solution has the form
$$\varphi_2(x)=x^{r_2}\sum^\infty_{n=0}\tilde{D}^\prime_n(r_2)x^n+(\log x)x^{r_2}\sum^\infty_{n=0}\tilde{D}_n(r_2)x^n, $$
where $\tilde{D}_0(r)=r-r_2$. Since
$$\tilde{D}_0(r_2)=\cdots=\tilde{D}_{m-1}(r_2)=0,$$we can write this in the following form:
$$\varphi_2(x)=x^{r_2}\sum^\infty_{n=0}\tilde{D}^\prime_n(r_2)x^n+c(\log x)\varphi_1(x),$$where $c$ is some constant and
$$\varphi_1(x)=x^{r_1}\sum^{\infty}_{n=0}D_n(r_1)x^n,\;\;\;D_0(r_1)=1.$$ Analogously another solution has the form
$$\varphi_3(x)=x^{r_3}\sum^\infty_{n=0}\hat{D}^\prime_n(r_3)x^n+(\log x)x^{r_3}\sum^\infty_{n=0}\hat{D}_n(r_3)x^n, $$
where $\hat{D}_0(r)=r-r_3$. Since
$$\hat{D}_0(r_3)=\cdots=\hat{D}_{p-1}(r_3)=0,$$ we can write this in the following form:
$$\varphi_3(x)=x^{r_3}\sum^\infty_{n=0}\hat{D}^\prime_n(r_3)x^n+\tilde{c}(\log x)\varphi_2(x),$$where $\tilde{c}$ is some constant.
\end{enumerate}}
\end{proof}
\begin{Remark}$\;${\rm
\begin{enumerate}
\item The method used above is based in the classical ideas of Frobenius and we shall refer to it as  \textit{the method of Frobenius for order three}. All the series obtained above converge for $|x|<R$.
\item The solutions for $x<0$, can be obtained by replacing
$$x^{r_1},x^{r_2},x^{r_3},\log x$$ in the expansion  according to
$$|x|^{r_1},|x|^{r_2},|x|^{r_3},\log |x|$$ respectively.
\end{enumerate}  }
\end{Remark}

\subsection{Examples}

\begin{Example}
\label{Example:realcomplex}{\rm Consider the equation
\begin{equation}\label{exemtiob}
x^3y^{\prime\prime\prime}+x^2y^{\prime\prime}+xy^\prime+x^3y=0,\;x>0.
\end{equation}
Note that in this case, $a(x)=1$, $b(x)=1$ and $b(x)=x^3$ which are analytic at $0$. The indicial polynomial is given
$$q(r)=r(r-1)(r-2)+r(r-1)a(0)+rb(0)+c(0)=r(r-1)(r-2)+r(r-1)+r=r(r^2-2r+2).$$
The roots are $r_1=1+i$, $r_2=1-i$ and $r_3=0$. Since $r_1-r_2=2i\notin\mathbb{Z}^+_0$, $r_1-r_3=1+i\notin\mathbb{Z}^+_0$ and $r_2-r_3=1-i\notin\mathbb{Z}^+_0$ according to Theorem~\ref{thmfrobb3I} we have three solutions of the form
$$\varphi_1(x)=x^{1+i}\sum^{\infty}_{n=0}a_nx^n\;\;\;(a_0=1),\;\;\;\varphi_2(x)=x^{1-i}\sum^{\infty}_{n=0}b_nx^n\;\;\;(b_0=1)$$
and
$$\varphi_3(x)=x^{0}\sum^{\infty}_{n=0}c_nx^n\;\;\;(c_0=1).$$
Substituting the solutions of  equation (\ref{exemtiob}) we obtain the following:
\begin{itemize}
\item For $\varphi_3$ we have the following recurrence relation for the coefficients
$$c_1=0,c_2=0\;\;\;\;\mbox{ and }\;\;\;n((n-1)^2+1)c_{n}+c_{n-3}=0\;\mbox{ for every }n\geq3.$$
Therefore, since $c_0=1$ we have that
$$\varphi_3(x)=1+\sum^{\infty}_{k=1}\frac{(-1)^k}{3^{k}k!(2^2+1)(5^2+1)\cdots((3k-1)^2+1) }x^{3k}.$$
\item For $\varphi_1$ we have the following recurrence relation for the coefficients
$$a_1=0,a_2=0\;\;\;\;\mbox{ and }\;\;\;n\big[\big(n+\frac{1+3i}{2}\big)^2+\frac{i}{2}\big]a_n+a_{n-3}=0\;\mbox{ for every }n\geq3.$$
Denote by $\alpha_0=\frac{1+3i}{2}$, $\beta_0=\frac{i}{2}$ and since $a_0=1$ we have that
$$\varphi_1(x)=x^{1+i}+\sum^{\infty}_{k=1}\frac{(-1)^k}{3^k\cdot k![(3+\alpha_0)^2+\beta_0][(6+\alpha_0)^2+\beta_0]\cdots [(3k+\alpha_0)^2+\beta_0]}x^{3k+1+i}.$$
\item For $\varphi_2$ we have the following recurrence relation for the coefficients
$$b_1=0,\;b_2=0\;\;\;\;\mbox{ and }\;\;\;n\big[\big(n+\frac{1-3i}{2}\big)^2-\frac{i}{2}\big]b_n+b_{n-3}=0\;\mbox{ for every }n\geq3.$$
Therefore, since $b_0=1$ we have that
$$\varphi_2(x)=x^{1-i}+\sum^{\infty}_{k=1}\frac{(-1)^k}{3^k\cdot k![(3+\overline{\alpha_0})^2+\overline{\beta_0}][(6+\overline{\alpha_0})^2+\overline{\beta_0}]\cdots [(3k+\overline{\alpha_0})^2+\overline{\beta_0}]}x^{3k+1-i}.$$
\end{itemize}
}
\end{Example}
\begin{Remark}{\rm Recall that, from complex numbers theory we have  $$x^{1\pm i}=x\cdot x^{\pm i}=xe^{\pm i\log x}=x\cos(\log x)\pm ix\sin(\log x).$$}
\end{Remark}

\begin{Example}[Bessel equation of  order zero for third order ODEs]{\rm Consider the equation
\begin{equation}\label{eqsl5}
x^3y^{\prime\prime\prime}+3x^2y^{\prime\prime}+xy^\prime+(x^3-\alpha^3)y=0,\;x>0
\end{equation}
where $\mbox{Re}(\alpha)\geq0$. Note that in this case, $a(x)=3$, $b(x)=1$ and $b(x)=x^3-\alpha^3$ which are analytic at $0$. Also
$$q(r)=r(r-1)(r-2)+r(r-1)a(0)+rb(0)+c(0)=r(r-1)(r-2)+3r(r-1)+r-\alpha^3=r^3-\alpha^3.$$
Let us study the case $\alpha=0$, in this case an equation is given by
$$x^3y^{\prime\prime\prime}+3x^2y^{\prime\prime}+xy^\prime+x^3y=0.$$
The indicial polynomial is given in this case according to $r^3=0$. Hence as roots are $r_1=r_2=r_3=0$, according to Theorem~\ref{thmfrobb3II} we have three solutions of the form
$$\varphi_1(x)=x^{0}\big(\sum^\infty_{n=0}a_nx^n\big),\;\;\varphi_2(x)=x^{1}\big(\sum^\infty_{n=0}b_nx^n\big)+(\log x)\varphi_1(x)$$
and
$$\varphi_3(x)=x^{1}\big(\sum^\infty_{n=0}c_nx^n\big)+2(\log x)\varphi_2(x)-(\log x)^2\varphi_1(x), $$
where $a_0\neq0$.

Substituting the solutions of  equation (\ref{eqsl5}) we obtain the following:
\begin{itemize}
\item For $\varphi_1$ we have the following recurrence relation for the coefficients
$$a_1=0,a_2=0\;\;\;\;\mbox{ and }\;\;\;n^3a_{n}+a_{n-3}=0\;\mbox{ for every }n\geq3.$$
Therefore, if we choose $a_0=1$ we have that
\begin{equation}\label{eqsl6}
\varphi_1(x)=\sum^{\infty}_{k=0}\frac{(-1)^k}{3^{3k}(k!)^3}x^{3k}=\sum^{\infty}_{k=0}\frac{(-1)^k}{(k!)^3}\big(\frac{x}{3}\big)^{3k}.
\end{equation}
\item For $\varphi_2$ we have the following recurrence relation for the coefficients:
$$b_0=b_1=0,\;b_2=\frac{1}{3^3},$$
$$(3n-2)^3b_{3n-3}+b_{3n-6}=0\;\;\mbox{ for every }n\geq 2,\;\;\;(3n-1)^3b_{3n-2}+b_{3n-5}=0\;\;\mbox{ for every }n\geq 2,$$
and
$$(3n)^3b_{3n-1}+b_{3n-4}=\frac{(-1)^{n+1}n^2}{(n!)^33^{3n-3}} \;\;\mbox{ for every }n\geq 2.$$
Therefore, we have that
\begin{equation}\label{eqsl7}
\varphi_2(x)=\sum^{\infty}_{n=1}\frac{(-1)^{n+1}}{3^{3n}(n!)^3}\big[1+\frac{1}{2}+\frac{1}{3}+
\cdots+\frac{1}{n}\big]x^{3n}+(\log x)\varphi_1(x).
\end{equation}
\item For $\varphi_3$ we have the following recurrence relation for the coefficients:
$$c_0=c_1=0,\;\;c_2=\frac{2^3}{3^4},$$
$$(3n-2)^3c_{3n-3}+c_{3n-6}=0\;\;\mbox{ for every }n\geq 2,\;\;\;(3n-1)^3c_{3n-2}+c_{3n-5}=0\;\;\mbox{ for every }n\geq 2,$$
and
$$(3n)^3c_{3n-1}+c_{3n-4}=\frac{(-1)^{n+1} (18n)}{3^{3n}(n!)^3}\big[(3n)\big(1+\frac{1}{2}+\frac{1}{3}+\cdots+\frac{1}{n}\big)+1\big] \;\;\mbox{ for every }n\geq 2.$$
Therefore, we have that
$$\varphi_3(x)=x\sum^{\infty}_{n=0}c_nx^n+2(\log x)\varphi_2(x)-(\log x)^2\varphi_1(x), $$
where $c_n$ given by the recurrence above.
\end{itemize}
}
\end{Example}

\begin{Definition}{\rm Equation (\ref{eqsl5}) will called \textit{Bessel equation of  order zero for third order ODEs}, due to the fact that  the solutions (\ref{eqsl6}) and (\ref{eqsl7}) have a similarity  with the functions of Bessel of order zero for second order differential equations.}
\end{Definition}

\begin{Example}[Laguerre differential equation of third order]{\rm Consider the equation
\begin{equation}\label{eqsl8}
x^2y^{\prime\prime\prime}+3xy^{\prime\prime}+(1-x)y^\prime+\alpha y=0,\;x>0
\end{equation}
where $\alpha\in\mathbb{R}$. Observe that equation (\ref{eqsl8}) has $0$ as a regular  singular point, since multiplying  by $x$ both sides  of  equation (\ref{eqsl8}) we have
$$x^3y^{\prime\prime\prime}+3x^2y^{\prime\prime}+(1-x)xy^\prime+\alpha x y=0.$$
Note that in this case, $a(x)=3$, $b(x)=1-x$ and $b(x)=\alpha x$ which are analytic at $0$. Also the indicial polynomial has the form
$$q(r)=r(r-1)(r-2)+r(r-1)a(0)+rb(0)+c(0)=r(r-1)(r-2)+3r(r-1)+r=r^3.$$
Hence as roots are $r_1=r_2=r_3=0$, according to Theorem~\ref{thmfrobb3II} we have that a solution of (\ref{eqsl8}) is of the form
$$\varphi(x)=x^{0}\big(\sum^\infty_{n=0}a_nx^n\big)$$
where $a_0\neq0$. Calculating the coefficients we obtain
 $$a_1=-\alpha a_0,a_2=\frac{(1-\alpha)(-\alpha)}{2^3}\;\;\;\;\mbox{ and }\;\;\;k^3a_k-((k-1)-\alpha)a_{k-1}=0\;\mbox{ for every }k\geq3.$$
Therefore, if we choose $a_0=1$ we have that
\begin{equation}\label{eqsl9}
\varphi(x)=1+\sum^{\infty}_{k=1}\frac{((k-1)-\alpha)\cdots(2-\alpha)(1-\alpha)(-\alpha)}{(k!)^3}x^{k}.
\end{equation}
}
\end{Example}
\begin{Definition}{\rm Equation (\ref{eqsl8}) will be called  \textit{Laguerre differential equation of third order} since if $\alpha=n-1$ where $n\in\mathbb{Z}^+$ then a solution (\ref{eqsl9}) is polynomial.}
\end{Definition}
\begin{Example}{\rm Consider the equation
\begin{equation}\label{exem2}
x^3y^{\prime\prime\prime}+x^2y^{\prime\prime}+x^2y^\prime+xy=0,\;x>0.
\end{equation}
Note that in this case, $a(x)=1$, $b(x)=x$ and $c(x)=x$ which are analytic at $0$. The indicial polynomial is given
$$q(r)=r(r-1)(r-2)+r(r-1)a(0)+rb(0)+c(0)=r(r-1)(r-2)+r(r-1)=r^3-2r^2+r.$$
The roots are $r_1=r_2=1$ and $r_3=0$. Since $r_1-r_3=1\in\mathbb{Z}^+$ according to Theorem~\ref{thmfrobb3II} there exist three solutions $\varphi_1,\varphi_2,\varphi_3$ defined, which has the form:
$$\varphi_1(x)=x\big(\sum^\infty_{n=0}a_nx^n\big),\;\;\varphi_2(x)=x^{2}\big(\sum^\infty_{n=0}b_nx^n\big)+(\log x)\varphi_1(x)$$
and
$$\varphi_3(x)=x^{2}\big(\sum^\infty_{n=0}c_nx^n\big)+c\;(\log x)\varphi_1(x), $$
where $c$ constant, $a_0\neq0$ and $c_0\neq0$.
Substituting the solutions of  equation (\ref{exem2}) we obtain the following:
\begin{itemize}
\item For $\varphi_1$ we have the following recurrence relation for the coefficients
 $$a_1=-a_0\;\;\;\;\mbox{ and }\;\;\; n^2(n+1)a_n+(n+1)a_{n-1}=0\;\mbox{ for every }n\geq2.$$
Therefore, if we choose $a_0=1$ we have that
$$\varphi_1(x)=\sum^{\infty}_{n=0}\frac{(-1)^n}{(n!)^2}x^{n+1}.$$
\item For $\varphi_2$ we have the following recurrence relation for the coefficients:
$$b_0=2,\;\;\;(n+1)^2b_n+b_{n-1}=\frac{(-1)^{n}2(n+1)}{((n+1)!)^2}\;\;\mbox{ for every }n\geq 1.$$Therefore, we have that
$$\varphi_2(x)=\sum^{\infty}_{n=0}b_nx^{n+2}+(\log x)\varphi_1(x),$$
where $b_n$ given by the recurrence above.
\item For $\varphi_3$ we have the following recurrence relation for the coefficients:
$$c_0=2c,\;\;\;(n+1)^2c_n+c_{n-1}=\frac{(-1)^{n}2c(n+1)}{((n+1)!)^2}\;\;\mbox{ for every }n\geq 1.$$
Therefore, we have that
$$\varphi_3(x)=\sum^{\infty}_{n=0}c_nx^{n+2}+(\log x)\varphi_1(x),$$
where $c_n$ given by the recurrence above.
\end{itemize}
}
\end{Example}

\begin{Example}{\rm Consider the equation
\begin{equation}\label{exem3}
x^3y^{\prime\prime\prime}+x^3y^{\prime\prime}+x^2y^\prime-xy=0,\;x>0.
\end{equation}
 Note that in this case, $a(x)=x$, $b(x)=x$ and $c(x)=-x$ which are analytic at $0$. The indicial polynomial is given
$$q(r)=r(r-1)(r-2)+r(r-1)a(0)+rb(0)+c(0)=r(r-1)(r-2).$$
The roots are $r_1=2,r_2=1$ and $r_3=0$. Since $r_1-r_2=1\in\mathbb{Z}^+$ and $r_2-r_3=1\in\mathbb{Z}^+$  according to Theorem~\ref{thmfrobb3II} there exist three solutions $\varphi_1,\varphi_2,\varphi_3$ defined, which has the form:
$$\varphi_1(x)=x^{2}\big(\sum^\infty_{n=0}a_nx^n\big),\;\;\varphi_2(x)=x\big(\sum^\infty_{n=0}b_nx^n\big)+c\;(\log x)\varphi_1(x)$$
and
$$\varphi_3(x)=x^{0}\big(\sum^\infty_{n=0}c_nx^n\big)+\tilde{c}\;(\log x)\varphi_1(x), $$
where $c$, $\tilde{c}$ are constants, $a_0\neq0$, $b_0\neq0$ and $c_0\neq0$.
Substituting the solutions of  equation (\ref{exem3}) we obtain the following:
\begin{itemize}
\item For $\varphi_1$ we have the following recurrence relation for the coefficients
$$(n+2)a_{n+1}+a_{n}=0\;\mbox{ for every }n\geq0.$$
Therefore, if we choose $a_0=1$ we have that
$$\varphi_1(x)=\sum^{\infty}_{n=0}\frac{(-1)^n}{(n+1)!}x^{n+2}.$$

\item For $\varphi_2$ we have the following recurrence relation for the coefficients:
$$c=0\;\;\;(n+1)nb_{n+1}+nb_n=0\;\;\mbox{ for every }n\geq 1,$$
we have that
$$\varphi_2(x)=\sum^{\infty}_{n=0}b_nx^{n+1},$$
where $b_n$ given by the recurrence above.
\item For $\varphi_3$ we have the following recurrence relation for the coefficients:

 $$\tilde{c}=-c_0\;\;\;nc_{n+1}+c_{n}=\frac{(-1)^{n}c_0}{n!}\;\;\mbox{ for every }n\geq 2.$$
Therefore, we have that
$$\varphi_3(x)=\sum^{\infty}_{n=0}c_nx^{n}-c_0(\log x)\varphi_2(x),$$where
where $c_n$ given by the recurrence above.
\end{itemize} }
\end{Example}

\subsection{Regular singular points at infinity}

We now proceed to investigate the solutions of an equation:
\begin{equation}\label{eq32}
L(y):=y^{\prime\prime\prime}+a_1(x)y^{\prime\prime}+a_2(x)y^\prime+a_3(x)y=0
\end{equation}
 for great values  of $|x|$. A simple way of doing this is through the change of variable $x=\frac{1}{t}$, and then analyze, in a neighborhood of  $t=0$, the solutions of the resulting  equation. Then we may apply, according to example, the previous results about analytic equations and to equations with a  regular singular point at $t=0$.
If $\varphi$ is a solution of (\ref{eq32}) for $|x|>R_0$, for some $R_0>0$, we put:

$$\tilde{\varphi}(t)=\varphi\big(\frac{1}{t}\big),\;\;\;\tilde{a}_1(t)=a_1\big(\frac{1}{t}\big),\;\;\;\tilde{a}_2(t)=a_2\big(\frac{1}{t}\big),\;\;\;\tilde{a}_3(t)=a_3\big(\frac{1}{t}\big).$$
These functions must exist for $|t|<\frac{1}{R_0}$, and
$$\frac{d\varphi}{dx}\big(\frac{1}{t}\big)=-t^2\frac{d\tilde{\varphi}}{dt}(t),$$
$$\frac{d^2\varphi}{dx^2}\big(\frac{1}{t}\big)=t^4\frac{d^2\tilde{\varphi}}{dt^2}(t)+2t^3\frac{d\tilde{\varphi}}{dt}(t),$$

$$\frac{d^3\varphi}{dx^3}\big(\frac{1}{t}\big)=-t^6\frac{d^3\tilde{\varphi}}{dt^3}(t)-6t^5\frac{d^2\tilde{\varphi}}{dt^2}(t)-6t^4\frac{d\tilde{\varphi}}{dt}(t).$$
Since, according to (\ref{eq32})

$$ \frac{d^3\varphi}{dx^3}\big(\frac{1}{t}\big)+
a_1\big(\frac{1}{t}\big)\frac{d^2\varphi}{dx^2}\big(\frac{1}{t}\big)+
a_2\big(\frac{1}{t}\big)\frac{d\varphi}{dx}\big(\frac{1}{t}\big)+
a_3\big(\frac{1}{t}\big)\varphi\big(\frac{1}{t}\big)=0,$$
we have:
$$t^6\tilde{\varphi}^{\prime\prime\prime}(t)+[6t^5-t^4\tilde{a}_1(t)]\tilde{\varphi}^{\prime\prime}(t)+[6t^4-2t^3\tilde{a}_1(t)+t^2\tilde{a}_2(t)]\tilde{\varphi}^{\prime}(t)-\tilde{a}_3(t)\tilde{\varphi}(t)=0 $$
Hence $\tilde{\varphi}$ satisfies the equation:
\begin{equation}\label{eq33}
\tilde{L}(y)=t^6y^{\prime\prime\prime}+[6t^5-t^4\tilde{a}_1(t)]y^{\prime\prime}+[6t^4-2t^3\tilde{a}_1(t)+t^2\tilde{a}_2(t)]y^{\prime}-\tilde{a}_3(t)y=0.
\end{equation}
Conversely, if $\tilde{\varphi}$ satisfies the equation $\tilde{L}(y)=0$, then a function $\varphi$ satisfies the equation $L(y)=0$. Equation (\ref{eq33}) is called {\it associate induced equation} with $L(y)=0$ and a substitution $x=1/t$.
We shall say that the \textit{the infinity is a regular singular point} of (\ref{eq32}), if the origin $t=0$ is a regular singular point of the induced equation (\ref{eq33}). Writing   equation (\ref{eq33}) as:
$$t^3y^{\prime\prime\prime}+\big[6-\frac{\tilde{a}_1(t)}{t}\big]t^2y^{\prime\prime}+\big[6-2\frac{\tilde{a}_1(t)}{t}+\frac{\tilde{a}_2(t)}{t^2}\big]ty^{\prime}-\frac{\tilde{a}_3(t)}{t^3}y=0$$
we see that $\tilde{L}(y)=0$ has the $t=0$ as a regular  singular point if and only if $\frac{\tilde{a}_1}{t},\frac{\tilde{a}_2}{t^2}$ and $\frac{\tilde{a}_3}{t^2}$ are analytic at $t=0$. This means
 that
$$\tilde{a}_1(t)=t\sum^\infty_{k=0}\alpha_kt^k,\;\tilde{a}_2(t)=t^2\sum^\infty_{k=0}\beta_kt^k,\; \tilde{a}_3(t)=t^3\sum^\infty_{k=0}\gamma_kt^k,$$
where the series converge for $|t|<\frac{1}{R_0}$, $R_0>0$. Translated into a condition that involves $a_1,a_2,a_3$, this means that
$$a_1(x)=\frac{1}{x}\sum^\infty_{k=0}\frac{\alpha_k}{x^k},\;a_2(x)=\frac{1}{x^2}\sum^\infty_{k=0}\frac{\beta_k}{x^k},\; a_3(x)=\frac{1}{x^3}\sum^\infty_{k=0}\frac{\gamma_k}{x^k},$$
where this series converges for $|x|>R_0$. Hence, the infinity is a regular singular point for equation (\ref{eq32}), if and only if (\ref{eq33}) can be written  as
$$x^3y^{\prime\prime\prime}+a(x)x^2y^{\prime\prime}+b(x)xy^{\prime}+c(x)y=0,$$
where $a,b,c$ has expansion in power series converging expressed  in powers of $1/x$, with $|x|>R_0$ for some $R_0>0$.
The simplest example  of an equation with regular singular point at the infinity is
$$x^3y^{\prime\prime\prime}+ax^2y^{\prime\prime}+bxy^\prime+cy=0,$$
where $a,b,c$ are constants, i.e., an Euler equation. Hence, this equation has the origin and the infinity as  regular  singular points, and clearly we see that are no  other singular points.
\begin{Example}{\rm Consider the equation of the Example \ref{exeminf}, we see that there exists no solution in a neighborhood of the origin. Nevertheless observe that if we change variables as $x=\frac{1}{t}$ equation (\ref{eqs4}) has the infinity as a regular  singular point since after the change we obtain the following equation:
$$ t^3y^{\prime\prime\prime}+7t^2y^{\prime\prime}+(8t-t^2)y^\prime+\frac{1}{2}y=0.$$}
\end{Example}
An example of equation that has three regular singular points (uniquely), is of the hypergeometric equation of third order (see \cite{S}, Chapter II Section 2.6):
$$(x^2-x^3)y^{\prime\prime\prime}+[\delta+\eta+1-(\alpha+\beta+\gamma+3)x]xy^{\prime\prime}+[\delta\eta-((\beta+\gamma+1)\alpha+(\beta+1)(\gamma+1))x]y^{\prime}-\alpha\beta\gamma y=0$$
where $\alpha, \beta, \gamma, \delta,\eta$ are constants. We can easily verify that $0$, $1$ and the infinity
$\infty$ are  regular singular points.
\begin{Remark}{\rm Consider the equation
\begin{equation}\label{eq34}
a(x)y^{\prime\prime}+b(x)y^{\prime\prime}+c(x)y^\prime+d(x)y=0
\end{equation}
 having the origin as non regular singular point. Then (\ref{eq34}) has the infinity as a regular  singular point if and only if
$$ \frac{b\big(\frac{1}{t}\big)}{ta\big(\frac{1}{t}\big)},\;\;\;\;\frac{c\big(\frac{1}{t}\big)}{t^2a\big(\frac{1}{t}\big)}\;\;\;\mbox{ and }\;\;\;\frac{d\big(\frac{1}{t}\big)}{t^3a\big(\frac{1}{t}\big)}$$are analytic at $t=0$. E in this case, there exist solutions of the Frobenius-Laurent type away from the origin. }
\end{Remark}

\subsection{Third order non-homogeneous equation}
A  \textit{non-homogeneous equation of third order with non-constant coefficients} is a
 differential equation of the form
\begin{equation}\label{eqsl34}
a_{0}(x)y^{\prime\prime\prime}+a_{1}(x)y^{\prime\prime}+a_2(x)y^{\prime}+a_3(x)y=f(x)
\end{equation}
where $a_0,a_1,a_2,a_3, f$ are functions defined in an interval $I$. Therefore the general solution will be of the form
$$y(x)=\varphi_h(x)+\varphi_p(x)$$where $\varphi_h$ is general solution of the homogeneous equation
\begin{equation}\label{eqsl35}
a_{0}(x)y^{\prime\prime\prime}+a_{1}(x)y^{\prime\prime}+a_2(x)y^{\prime}+a_3(x)y=0
\end{equation}
and $\varphi_p$ is particular solution of (\ref{eqsl34}). The problem is to find a solution particular of (\ref{eqsl34}) and a general solution of the homogeneous equation (\ref{eqsl35}). The technic  that we use for finding a particular solution is variation of parameters.
\subsubsection{Method of variation of parameters}\label{varmeth}
Let $\{\varphi_1,\varphi_2,\varphi_3\}$ be a basis of the solution space of the homogeneous equation (\ref{eqsl35}). We shall assume that
\begin{equation}\label{eqsl36}
\varphi_p(x)=C_1(x)\varphi_1(x)+C_2(x)\varphi_2(x)+C_3(x)\varphi_3(x)
\end{equation}
where $C_1,\;C_2,C_3$ are functions that verify
\begin{equation}\label{eqsl37}
\begin{array}{c c l}
C_1^\prime(x)\varphi_1(x)+C_2^\prime(x)\varphi_2(x)+C_3^\prime(x)\varphi_3(x) &=&0
\\
\\
C_1^\prime(x)\varphi_1^\prime(x)+C_2^\prime(x)\varphi_2^\prime(x)+C_3^\prime(x)\varphi_3^\prime(x)&=&0\\
\\
C_1^\prime(x)\varphi_1^{\prime\prime}(x)+C_2^\prime(x)\varphi_2^{\prime\prime}(x)+C_3^\prime(x)\varphi_3^{\prime\prime}(x)&=&\frac{f(x)}{a_0(x)}
\end{array}
\end{equation}
since $\{\varphi_1,\varphi_2,\varphi_3\}$ are linearly independent then the Wronskian of third order of $\varphi_1,\varphi_2,\varphi_3$ never vanishes, i.e.,
$$W(\varphi_1,\varphi_2,\varphi_3)(x)=\det\left(\begin{array}{c c c}
\varphi_1(x)&\varphi_2(x)&\varphi_3(x)\\&&
\\
\varphi_1^\prime(x)&\varphi_2^\prime(x)&\varphi_3^\prime(x)
\\&&\\
\varphi_1^{\prime\prime}(x)&\varphi_2^{\prime\prime}(x)&\varphi_3^{\prime\prime}(x)
\end{array}\right)\neq0$$
for every $x$. Therefore the system (\ref{eqsl37}) has a unique solution given by
\begin{equation}\label{eqsl38}
\begin{array}{c}
C_1^\prime(x)=\frac{\det\left(\begin{array}{c c c}
0&\varphi_2(x)&\varphi_3(x)\\&&\\
0&\varphi_2^\prime(x)&\varphi_3^\prime(x)
\\&&\\
f(x)&\varphi_2^{\prime\prime}(x)&\varphi_3^{\prime\prime}(x)
\end{array}\right)}{W(\varphi_1,\varphi_2,\varphi_3)(x)}
\\
\\
C_2^\prime(x)=\frac{\det\left(\begin{array}{c c c}
\varphi_1(x)&0&\varphi_3(x)\\&&
\\
\varphi_1^\prime(x)&0&\varphi_3^\prime(x)
\\&&\\
\varphi_1^{\prime\prime}(x)&f(x)&\varphi_3^{\prime\prime}(x)
\end{array}\right)}{W(\varphi_1,\varphi_2,\varphi_3)(x)}
\\ \\
C_3^\prime(x)=\frac{\det\left(\begin{array}{c c c}
\varphi_1(x)&\varphi_2(x)&0\\&&
\\
\varphi_1^\prime(x)&\varphi_2^\prime(x)&0
\\&&\\
\varphi_1^{\prime\prime}(x)&\varphi_2^{\prime\prime}(x)&f(x)
\end{array}\right)}{W(\varphi_1,\varphi_2,\varphi_3)(x)}.
\end{array}
\end{equation}
We shall now construct as other linearly independent solutions from a solution of the
homogeneous equation (\ref{eqsl35}).
\subsubsection{Reduction of order}
Let us consider a homogeneous third order equation with non-constant coefficients (\ref{eqsl35}).
 We shall apply the method of \textit{reduction of order} that consists in finding another solution of the differential equation from an already known solution as we shall see. Suppose that we know a solution $\varphi$ of  equation (\ref{eqsl34}). Let us consider $\psi=\mu \varphi$ where $\mu$ is a
function. We shall assume that $\psi$ is a solution of (\ref{eqsl34}). The substitution
leads to the  following second order equation for $u^\prime$:
$$a_0(x)\varphi(x)u^{\prime\prime\prime}+(3a_0(x)\varphi^\prime(x)+a_1(x)\varphi(x))u^{\prime\prime}+ (3a_0(x)\varphi^{\prime\prime}(x)+2a_1(x)\varphi^\prime(x)+a_2(x)\varphi(x))u^\prime=0.$$
From two linearly independent solutions we can construct a third solution. The proof of this construction can be found in \cite{MLT}. We shall state this technique  in what follows:

\subsubsection{Construction of a third solution from two linearly independent solutions}\label{wrosnk}
Let $y_1,y_2$ be linearly independent solutions of (\ref{eqsl35}). Denote by
$$W_{ij}=-W_{ji}=y_iy_j^\prime-y_jy_i^\prime,\;i\neq j.$$ The idea is to verify that $y_1,y_2$ are solutions of the following equation
$$ W_{12}(x)z^{\prime\prime}-W^\prime_{12}(x)z^\prime+(y_1^\prime(x)y^{\prime\prime}_2(x)-y^\prime_2(x)y^{\prime\prime}_1(x))z=0$$
and then consider a non-homogeneous equation
\begin{equation}\label{eqsl39}
W_{12}(x)z^{\prime\prime}-W^\prime_{12}(x)z^\prime+(y_1^\prime(x)y^{\prime\prime}_2(x)-y^\prime_2(x)y^{\prime\prime}_1(x))z=W(x)
\end{equation}
where $$W(s)=\exp\big(-\dint^s_{x_0}\frac{a_1(t)}{a_0(t)}dt\big)$$com $x_0\in I$.

The particular solution of this equation is obtained by the method of variation of parameters (see \cite{C}, Chapter II Section 2.6) and this given by
$$y_3(x)=y_2(x)\dint^x_{x_0}\frac{y_1(s)W(s)}{(W_{12}(s))^2}ds-y_1(x)\dint^x_{x_0}\frac{y_2(s)W(s)}{(W_{12}(s))^2}ds.$$

It is not difficult to see that $y_3$ is a solution of the homogeneous equation (\ref{eqsl35}) and $W(y_1,y_2,y_3)(x)\neq0$ for every $x\in I$.

Also we have:
\begin{equation}\label{eqsl40}
\begin{array}{c c l}
y_1W_{23}+y_2W_{31}+y_3W_{12} &=&0
\\
\\
y_1^\prime W_{23}+y_2^\prime W_{31}+y_3^\prime W_{12}&=&0\\
\\
y_1^{\prime\prime}W_{23}+y_2^{\prime\prime}W_{31}+y_3^{\prime\prime}W_{12}&=&W.
\end{array}
\end{equation}
Now for finding a particular solution of (\ref{eqsl34}) we use the method of variation of parameters used in section \ref{varmeth}. Hence we consider
$$y_p(x)=C_1(x)y_1(x)+C_2(x)y_2(x)+C_3(x)y_3(x)$$
where $C_1,\;C_2,C_3$ are functions that verify
\begin{equation}\label{eqsl41}
\begin{array}{c c l}
C_1^\prime(x)y_1(x)+C_2^\prime(x)y_2(x)+C_3^\prime(x)y_3(x) &=&0
\\
\\
C_1^\prime(x)y_1^\prime(x)+C_2^\prime(x)y_2^\prime(x)+C_3^\prime(x)y_3^\prime(x)&=&0\\
\\
C_1^\prime(x)y_1^{\prime\prime}(x)+C_2^\prime(x)y_2^{\prime\prime}(x)+C_3^\prime(x)y_3^{\prime\prime}(x)&=&\frac{f(x)}{a_0(x)}.
\end{array}
\end{equation}
Observe that according to (\ref{eqsl40}) we have that
$$C_1^\prime(x)=\frac{f(x)W_{23}(x)}{a_0(x)W(x)},\;\;\;C_2^\prime(x)=\frac{f(x)W_{32}(x)}{a_0(x)W(x)}\;\;\;\;\mbox{ and }\;\;\;C_3^\prime(x)=\frac{f(x)W_{12}(x)}{a_0(x)W(x)} $$
are solutions of (\ref{eqsl41}).
Therefore a particular solution of (\ref{eqsl34}) is given by
$$y_p(x)=y_1(x)\dint^{x}_{x_0}\frac{W_{23}(s)f(s)}{W(s)a_0(s)}ds+y_2(x)\dint^{x}_{x_0}\frac{W_{31}(s)f(s)}{W(s)a_0(s)}ds+y_3(x)\dint^{x}_{x_0}\frac{W_{12}(s)f(s)}{W(s)a_0(s)}ds.$$

On the other hand, according to  Theorem~\ref{thmfrobb3I} and \ref{thmfrobb3II} we have that in the case of
third order differential equations with a   regular singular point at the origin it is always possible to find the linearly independent solutions of the homogeneous equation and making use of the method introduced in this section we find a general solution of the non-homogeneous equation since segue:
\begin{Corollary}{\rm Every solution of
$$L(y):=x^3y^{\prime\prime\prime}+x^2a(x)y^{\prime\prime}+xb(x)y^\prime+c(x)y=f(x),$$
with  $a(x),b(x), c(x)$ and $f(x)$ analytic for $|x|<R$, $R>0$, is of the form
$$\begin{array}{l}y(x)=c_1y_1(x)+c_2y_2(x)+c_3y_3(x)+y_1(x)\dint^{x}_{x_0}\frac{(y_2(s)y_3^\prime(s)-y_3(s)y_2^\prime(s))f(s)}{s^3\exp\big(-\dint^{s}_{x_0}\frac{a(t)}{t}dt\big)}ds\\
\\\;\;\;+y_2(x)\dint^{x}_{x_0}\frac{(y_3(s)y_1^\prime(s)-y_1(s)y_3^\prime(s))f(s)}{s^3\exp\big(-\dint^{s}_{x_0}\frac{a(t)}{t}dt\big)}ds+y_3(x)\dint^{x}_{x_0}\frac{(y_1(s)y_2^\prime(s)-y_2(s)y_1^\prime(s))f(s)}{s^3\exp\big(-\dint^{s}_{x_0}\frac{a(t)}{t}dt\big)}ds\end{array},$$
where $y_1,y_2,y_3$ are linearly independent solutions of $L(y)=0$ and $c_1,c_2,c_3$ constants.}
\end{Corollary}

\subsection{Convergence of formal solutions for third order ODEs}

Now we are in conditions to prove Theorem~\ref{Theorem:1formalregularorderthree}:

\begin{proof}[Proof of Theorem~\ref{Theorem:1formalregularorderthree}]
This is proved like Theorem~\ref{Theorem:B} once we have the description of the solutions given for order three by Theorems~\ref{thmfrobb3I} and ~\ref{thmfrobb3II}. 
\end{proof}
Theorem~\ref{Theorem:1formalregularorderthree} above cannot be improved by simply removing the hypothesis of regularity on the singular point,
as shows the following example:

\begin{Example}\label{exeminf} {\rm Consider the equation
\begin{equation}\label{eqs4}
x^3y^{\prime\prime\prime}-x^2y^{\prime\prime}-y^\prime-\frac{1}{2}y=0.
\end{equation}
The origin $x_0=0$ is a singular point, but not is regular singular point, since the coefficient -1 of $y^\prime$ does not have the form $xb(x)$, where $b$ is analytic for $0$. Nevertheless, we can formally solve this equation by power series
\begin{equation}\label{eqs5}
\sum^{\infty}_{k=0} a_kx^k,
\end{equation}
where the coefficients $a_k$ satisfy the following recurrence formula
\begin{equation}\label{eqs6}
(k+1)a_{k+1}=\big[k^3-4k^2+3k-\frac{1}{2}\big]a_k,\;\;\;\mbox{ for every }k=0,1,2,\ldots.
\end{equation}
If $a_0\neq0$, applying the quotient test to expressions (\ref{eqs5}) and (\ref{eqs6}), we have that
$$ \big|\frac{a_{k+1}x^{k+1}}{a_kx^k}\big|=\big|\frac{k^3-4k^2+3k-\frac{1}{2}}{k+1}\big|\cdot|x|\to\infty,$$when $k\to\infty$, provided that $|x|\neq0$. Hence, the series converges  only for $x=0$.}
\end{Example}

\subsection{Further questions}

Other convergence aspects of solutions of third order ODEs as well as the characterization of those
admitting regular singular points, in terms of the space of solutions, will be studied
in a forthcoming work (\cite{Leon-Scardua}).
We shall also discuss some more general notions of regularity for a singular point, under which we still
have the existence of solutions. An interesting question is the classification of the third order ODEs which admit a Liouvillian solution or a Liouvillian first integral. Another intriguing problem is  the
search of a first or second  order  model for a third order ODE. This would probably lead to the introduction of a type of holonomy group for such ODEs. Finally, it seems reasonable to imagine that more general versions of Frobenius methods are valid for ODEs having coefficients in a function field $\mathbb K(x)$ where $\mathbb K$ is an algebraically closed ordered field, of characteristic $k\geq 0$. This is also treated in the continuation of our work.


\bibliographystyle{amsalpha}

\begin{thebibliography}{31}
\frenchspacing




\bibitem{A} G. Arfken,  {\em Series Solutions--Frobenius' Method.} §8.5 in Mathematical Methods for Physicists, 3rd ed. Orlando, FL: Academic Press, pp. 454-467, 1985.



\bibitem{BS}  H. A. Bethe, E. E. Salpeter,  Quantum mechanics of one- and two-electron atoms, (Berlin
Germany, Springer), 1957.

\bibitem{Boyce} W.W.Boyce \& R.C.DiPrima, Elementary Differential Equations and Boundary Value Problems (10th. ed.), Wiley.




\bibitem{C-LN-S1} C. Camacho, A. Lins Neto and P. Sad: {\em   Foliations with algebraic limit
sets}; Ann. of Math. 136 (1992), 429--446.


\bibitem{Camacho-Scardua} {C. Camacho, B. Sc\'ardua};
{\em Holomorphic foliations with Liouvillian first integrals},
{Ergodic Theory and Dynamical Systems} (2001), 21, pp.717-756.


\bibitem{C} E. A. Coddington, { An Introduction to Differential Equations}, Dover Publications, 1989.

\bibitem{Dirac} P. A. M. Dirac, C{\em Classical Theory of Radiating Electrons}. Proceedings of the Royal Society of London. Series A, Mathematical and Physical Sciences. 167 (929): 148-169, 1938.



\bibitem{C-LN} C. Camacho, A. Lins Neto;
{\em Geometric theory of foliations}, Translated from the Portuguese
by Sue E. Goodman. Birkhauser Boston, Inc., Boston, MA (1985).

\bibitem{Ce-Mt} D. Cerveau, J.-F. Mattei, Formes intégrables
holomorphes singuli\`eres;  Ast\'erisque 97 (1982), Societ\'e
Math\'ematique de France.


\bibitem{Cordani}  B. Cordani, The Kepler Problem (Basel Switzerland, Birkhauser), 2013.


\bibitem{Dushman} S. Dushman, {\em Elements of the quantum theory; VI the hydrogen atom}, Journal of Chemical
Education, 12 (11), 529 - 539, 1935.




\bibitem{F} R.P. Feynman, R. B. Leighton, M. Sands,  The Feynman Lectures on Physics. Vol. 1 (2nd ed.). Pearson/Addison-Wesley, 2005.


\bibitem{Fowles2}
G. R. Fowles, {\em Solution of the Schroedinger equation for the hydrogen atom in rectangular
coordinates}, American Journal of Physics, 30 (4), 308-309, 1962
17.


\bibitem{Frobenius} F. G. Frobenius. {\it Ueber die Integration der linearen Differentialgleichungen durch Reihen.} J. reine angew. Math. 76, 214-235, 1873.

    \bibitem{Fowles} G. R. Fowles, G. L. Cassiday,  Analytical Mechanics (6th ed.). Saunders College Publishing, 1999.


\bibitem{Gr} David J. Griffiths,  Introduction to Quantum Mechanics. Second Edition. Pearson: Upper Saddle River, NJ, 2006.
    \bibitem{Hill} G. Hill, \textit{On the part of the motion of
the lunar perigee which is a function of the mean motions
of the sun and moon}, Press of J. Wilson and son, Cambridge, (1877).

\bibitem{Ince} E. L.  Ince,  Ordinary Differential Equations. New York: Dover, 1956.


\bibitem{Kepler}  J. Kepler, Astronomia Nova, 1609, edited in
J. Kepler, Gesammelte Werke, vol III, C.H.Beck, M¨unchen, 2nd edition, 1990.


\bibitem{Leon-Scardua} V. León, B. Scárdua: {\em  On singular Frobenius  for linear  differential equations of second and third order, part 2}. In preparation.

\bibitem{L} Peter W. Likins,  Elements of Engineering Mechanics. McGraw-Hill Book Company, 1973.


\bibitem{malgrangeI}  B. Malgrange: {\it Frobenius avec singularités, 1. Codimension un}, Public. Sc. I.H.E.S., 46 (1976), pp. 163-173.

\bibitem{malgrangeII} B. Malgrange: {\it  Frobenius avec singularites. 2. Le cas general}. Inventiones mathematicae (1977)
Volume: 39, page 67-90.


\bibitem{MLT} B. L. Moreno-Ley, J. López-Bonilla, B. Man Tuladhar, {\it On the 3rd order linear differential equation}, Kathmandu Univ. J. of Sci. Eng. Tech. 8(2) (2012) 7-10.


\bibitem{newton} Issac Newton, {\it Axioms or Laws of Motion}. Mathematical Principles of Natural Philosophy. 1, containing Book 1 (1729 English translation based on 3rd Latin edition (1726) ed.). p. 19.
Newton, Isaac. "Axioms or Laws of Motion". Mathematical Principles of Natural Philosophy. 2, containing Books 2 \& 3 (1729 English translation based on 3rd Latin edition (1726) ed.). p. 19.


\bibitem{Omori} Toshiaki Omori, Toru Aonishi and Masato Okada, {\em Switch of encoding characteristics in single neurons by subthreshold and suprathreshold stimuli},
Phys. Rev. E 81, 021901 (2010) - Published 1 February 2010.

\bibitem{Pain}  P. Painlev\'e, Leçons sur la th\'eorie analytique des \'equations
diff\'erentielles. Paris, Librairie Scientifique A. Hermann, 1897.



\bibitem{Reis} Fernando Reis, {\it Methods from holomorphic foliations for second order ordinary differential equations.}  Phd thesis IM-UFRJ July, 2019.

\bibitem{Rosenlicht} M. Rosenlicht; {\em On Liouville's theory of elementary
functions}; Pacific J. Math. 65 (1976).

\bibitem{FR} Fritz Rohrlich: {\em The dynamics of a charged sphere and the electron}, Am. J. Phys. 65 (11) p. 1051 (1997).



\bibitem{S} J. B.  Seaborn , {\it Hypergeometric Functions and Their Applications}, Text in Applied Mathematics, 8, Springer-Verlag, 1991.

\bibitem{Singer} M.F. Singer; {\em Liouvillian First Integrals of
Differential Equations}; Transactions of the American Mathematical
Society, vol. 333, number 2, october 1992, pp.673-688.


\bibitem{Troy} William C. Troy
SIAM Journal on Applied Mathematics
Vol. 32, No. 1 (Jan., 1977), pp. 146-153



\end{thebibliography}

\end{document}